\documentclass{article}

\usepackage{hyperref}

\usepackage{amsfonts, amsmath, amsthm}

\usepackage[margin=3cm]{geometry}

\usepackage{tikz}
\usetikzlibrary{calc, arrows, decorations.markings, decorations.pathmorphing, positioning, decorations.pathreplacing}
\makeatletter
\tikzset{nomorepostaction/.code={\let\tikz@postactions\pgfutil@empty}}
\makeatother

\newtheorem{thm}{Theorem}[section]
\newtheorem{cor}[thm]{Corollary}
\newtheorem{lem}[thm]{Lemma}
\newtheorem{prop}[thm]{Proposition}

\newtheorem{defn}[thm]{Definition}

\newcommand{\To}{\longrightarrow}
\newcommand{\C}{\mathcal{C}}
\newcommand{\E}{\mathcal{E}}

\newcommand{\Z}{\mathbb{Z}}

\newcommand{\D}{\mathcal{D}}

\newcommand{\B}{\mathcal{B}}

\newcommand{\X}{\mathcal{X}}
\newcommand{\Y}{\mathcal{Y}}
\newcommand{\A}{\mathcal{A}}

\renewcommand{\S}{\mathcal{S}}

\DeclareMathOperator{\im}{Im}

\renewcommand{\Im}{\text{Im}\;}

\newcommand{\An}{\mathbb{A}}

\DeclareMathOperator{\Byp}{Byp}

\DeclareMathOperator{\Tor}{Tor}
\DeclareMathOperator{\Conj}{Conj}

\begin{document}

\title{Strings, fermions and the topology of curves on annuli}

\author{Daniel V. Mathews}

%If a date is wanted, put it here.
%\date{}

\maketitle

% If an abstract is desired.
\begin{abstract}
In previous work with Schoenfeld, we considered a string-type chain complex of curves on surfaces, with differential given by resolving crossings, and computed the homology of this complex for discs.

In this paper we consider the corresponding ``string homology" of annuli. We find this homology has a rich algebraic structure which can be described, in various senses, as fermionic. While for discs we found that an isomorphism between string homology and the sutured Floer homology of a related 3-manifold, in the case of annuli we find the relationship is more complex, with string homology containing further higher-order structure.
\end{abstract}

%If a table of contents is desired.
\tableofcontents

\section{Introduction}

\subsection{Overview}

This paper considers basic objects of classical topology --- namely, curves on surfaces --- and some rich algebraic structure that arises from the simple operation of resolving their crossings. It is a continuation of previous work of the author with Schoenfeld \cite{Mathews_Schoenfeld12_string}.

A \emph{marked surface} $(\Sigma,F)$ is a compact oriented surface $\Sigma$ with non-empty boundary, together with a set $F$ of \emph{marked points}, which are signed points on $\partial \Sigma$. A \emph{string diagram} on $(\Sigma,F)$ is a collection of oriented curves up to homotopy on $\Sigma$, including both closed curves (``closed strings") and arcs (``open strings"), with endpoints given by $F$. 

In \cite{Mathews_Schoenfeld12_string}, we introduced a \emph{vector space} $\widehat{CS}(\Sigma,F)$ generated by string diagrams on $(\Sigma,F)$ (up to two different types of homotopy). We defined a \emph{differential} operator $\partial$ on these vector spaces: given a string diagram $s$ with transverse intersections, its differential $\partial s$ is the sum, over all intersection points, of the string diagram obtained by \emph{resolving} that intersection as in figure \ref{fig:resolving_crossing}. 

\begin{figure}[h]
\begin{center}
\begin{tikzpicture}[scale=1, string/.style={thick, draw=red, -to}]

\draw [string] (-1,0) -- (1,0);
\draw [string] (0,-1) -- (0,1);

\draw [shorten >=1mm, -to, decorate, decoration={snake,amplitude=.4mm, segment length = 2mm, pre=moveto, pre length = 1mm, post length = 2mm}]
(1.5,0) -- (2.5,0);

\draw [string] (3,0) -- (3.7,0) to [bend right=45] (4,0.3) -- (4,1);
\draw [string] (4,-1) -- (4,-0.3) to [bend left=45] (4.3,0) -- (5,0);
\end{tikzpicture}
\caption{Resolution of a crossing}
\label{fig:resolving_crossing}
\end{center}
\end{figure}
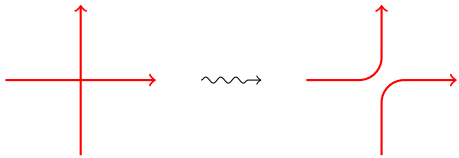

In \cite{Mathews_Schoenfeld12_string} we showed that we obtained chain complexes and explicitly computed the resulting ``string homology" $\widehat{HS}(\Sigma,F)$ in the case where $\Sigma$ is a disc $D^2$.

In this paper, we extend our results to \emph{annuli}. In a certain sense, we are able to calculate ``string homology" explicitly in \emph{all but one case}. We find that the resulting algebraic structure is much more intricate in the annular than in the disc case. As the title suggests, this algebraic structure is curiously ``fermionic", in a sense that we can make precise.

We found in \cite{Mathews_Schoenfeld12_string} a close connection between our ``string homology" for $D^2$, and Floer-theoretic invariants such as \emph{sutured Floer homology} and \emph{embedded contact homology} for $D^2 \times S^1$. We will show here that the ``string homology" of an annulus $\An$ is a more elaborate object than sutured Floer homology of $\An \times S^1$, though still clearly closely related. This paper focuses on computations of ``string homology" of annuli and lays a basis for further work examining connections to, and possible extensions of, Floer theory.

In this introduction, we give a brief overview of our results and an indication of their connections to other work. Throughout this paper, $\An$ denotes an annulus and $\Sigma$ denotes a general compact oriented surface with nonempty boundary. We always work over $\Z_2$ coefficients.

\subsection{Alternating marked surfaces}
\label{sec:alternating_surfaces}

A marked surface is \emph{alternating} if its marked points alternate in sign around each boundary component. In \cite{Mathews_Schoenfeld12_string} we showed that ``string homology" is trivial for many non-alternating marked surfaces, and in particular for any non-alternating marked disc. 

In this paper we show that this result carries over to annuli.
\begin{thm}
\label{thm:non-sutured_zero_iso}
Let $(\An,F)$ be a weakly marked annulus which is not alternating. Then $\widehat{HS}(\An,F) = 0$.
\end{thm}
The proof, which appears in section \ref{sec:non-alternating_annuli}, uses some techniques from our previous work \cite{Mathews_Schoenfeld12_string}, and also some new ideas. After this result, we only need consider alternating $(\An,F)$.

The alternating case is important because a marked alternating surface, also known as a \emph{sutured background} (\cite{Me10_Sutured_TQFT, Me11_torsion_tori}, see also \cite{Zarev09}), naturally forms a boundary condition for a \emph{set of sutures} on a surface. The study of sutures goes back at least to work of Gabai on 3-manifolds and foliations \cite{Gabai83}; sutures also play a crucial role as dividing sets in contact geometry \cite{Gi91, HKM02, Hon00I} and sutured 3-manifolds are the subject matter of \emph{sutured Floer homology} ($SFH$) \cite{Ju06}. In our previous work we have shown that sutures on surfaces can be used to give insight into contact categories \cite{Me09Paper}. 

Moreover, previous work of ourselves and Honda--Kazez--Mati\'{c} \cite{HKM08, Me09Paper, Me10_Sutured_TQFT, Me11_torsion_tori, Me12_elementary, Me12_itsy_bitsy, Me14_twisty_itsy_bitsy} has shown that $SFH$ of products $(\Sigma \times S^1, F \times S^1)$ is isomorphic, or at least closely related to (depending on the context) a vector space based on sets of sutures, modulo \emph{bypass triples} --- triples of sutures which are related by the natural contact-geometric operation of \emph{bypass surgery}. This vector space is the Grothendieck group of the contact category of $(\Sigma,F)$ (as defined by Honda \cite{HonCat}). As discussed in \cite{Mathews_Schoenfeld12_string}, such triples naturally arise in our ``string homology" as boundaries. Thus our ``string homology" should be strongly related to sutured Floer homology via the construction of the ``sutures modulo bypass triples" Grothendieck group. In \cite{Mathews_Schoenfeld12_string} we showed a direct isomorphism when $\Sigma$ is a disc, so that our elementary ``string homology", based on nothing more than curves and their crossings, is equivalent to a Floer-theoretic invariant based on symplectic and contact geometry and holomorphic curves. 

In this paper, as discussed below in section \ref{sec:relations_to_floer}, we show that ``string homology" is a richer and more complicated structure than the contact Grothendieck group or sutured Floer homology (even with twisted coefficients) of $(\Sigma \times S^1, F \times S^1)$, in the case of annuli.

\subsection{No marked points and homology of fermions}

After theorem \ref{thm:non-sutured_zero_iso}, we only need consider alternating marked surfaces --- hence with an even number of marked points on each boundary components. We can write $F_{2m,2n}$ for an alternating set of marked points with $2m$ and $2n$ points on the two boundary components of the annulus $\An$. We find a great deal of interesting --- and not yet fully understood --- structure in $\widehat{HS}(\An, F_{2m,2n})$.

In cases where $F$ is small, we can compute $\widehat{HS}(\An, F)$ explicitly; when $F$ is larger, we can still give some description.

First, consider the case of \emph{no} marked points: $F_{0,0} = \emptyset$. (In \cite{Mathews_Schoenfeld12_string} we required marked surfaces to have marked points on each boundary component; in this paper we relax this requirement and allow the case of an empty marked point set.) In this case, we show that $\widehat{HS}(\An, \emptyset)$ has the structure of a (commutative) $\Z_2$-\emph{algebra}, with multiplication corresponding to juxtaposition of curves. A string which runs $k$ times around the core of the annulus is denoted $x_k$, and its homology class (if it exists) in $\widehat{HS}(\An, \emptyset)$ by $\bar{x}_k$. Throughout we will denote by $\bar{z}$ the homology class of an element $z$.
\begin{thm}
\label{thm:annulus_no_marked_points_calculation}
The homology $\widehat{HS}(\An, \emptyset)$ is generated as a $\Z_2$-algebra by the homology classes $\bar{x}_k$ of $x_k$, over all odd integers $k$, subject to the relation that each $\bar{x}_k^2 = 0$. That is,
\[
\widehat{HS}(\An, \emptyset) = \frac{\Z_2[\ldots, \bar{x}_{-3}, \bar{x}_{-1}, \bar{x}_1, \bar{x}_3, \ldots] }{ (\ldots, \bar{x}_{-3}^2, \bar{x}_{-1}^2, \bar{x}_1^2, \bar{x}_3^2, \ldots) }.
\]
\end{thm}
In other words, $\widehat{HS}(\An, \emptyset)$ has the structure of a polynomial algebra over $\Z_2$, in infinitely many commuting variables $\bar{x}_{2j+1}$, such that each variable satisfies $\bar{x}_{2j+1}^2 = 0$.

This algebra is ``fermionic" in two senses:
\begin{enumerate}
\item
Only curves with ``odd spin" $x_{2j+1}$ survive in homology: curves which proceed around the annulus an even number of times are not cycles, and do not even have homology classes.
\item
Even for the ``odd spin" curves $x_{2j+1}$, which do survive in homology, their square is zero, $\bar{x}_{2j+1}^2 = 0$. Two such curves cause a ``pair annihilation" and give zero in homology. Thus there is a ``Pauli exclusion principle" for these curves.
\end{enumerate}
This ``fermionic polynomial" algebra, which we denote $H(\X)$ as explained in section \ref{sec:algebraic_structure}, is central to all further calculations. This is partly because of the technicalities of the algebra, but can be seen explicitly for example in the following proposition, which will be proved (in greater generality) in section \ref{sec:differential_algebra}.
\begin{prop}
\label{prop:HX-modules}
For any $n \geq 0$, $\widehat{HS}(\An, F_{0,2n})$ is an $H(\X)$-module.
\end{prop}

The computation of $\widehat{HS}(\An, \emptyset)$ is rather involved and occupies section \ref{sec:annuli_no_marked_points} of this paper. Along the way we see an analogy to ``decay chains" of particles, with corresponding ``decay" and ``fusion" operators (sections \ref{sec:decay_chains} and \ref{sec:creation_operators}) and Weyl algebra representations (section \ref{sec:Weyl_hierarchy}).

\subsection{Further homology computations: Two and four marked points}

We then proceed to annuli with the next smallest set of alternating marked points, namely two marked points, both on the same boundary component, $(\An, F_{0,2})$. From proposition \ref{prop:HX-modules}, this is a $H(\X)$-module, i.e. a module over a polynomial ring in infinitely many variables $\bar{x}_{2j+1}$ where each $\bar{x}_{2j+1}^2=0$.

A string diagram on $(\An, F_{0,2})$ contains a single open string between the two points of $F_{0,2}$, as well as some (possibly none) closed strings.
\begin{thm}
\label{thm:annulus_two_marked_points_same_boundary_calculation}
The homology $\widehat{HS}(\An, F_{0,2})$, as an $H(\X)$-module, is given by
\begin{align*}
\widehat{HS} (\An, F_{0,2}) 
&\cong
\bar{x}_1 H(\X) \oplus \bar{x}_{-1} H(\X) \\
&\cong
\bar{x}_1 \frac{\Z_2[\ldots, \bar{x}_{-3}, \bar{x}_{-1}, \bar{x}_1, \bar{x}_3, \ldots] }{ (\ldots, \bar{x}_{-3}^2, \bar{x}_{-1}^2, \bar{x}_1^2, \bar{x}_3^2, \ldots) } \oplus \bar{x}_{-1} \frac{\Z_2[\ldots, \bar{x}_{-3}, \bar{x}_{-1}, \bar{x}_1, \bar{x}_3, \ldots] }{ (\ldots, \bar{x}_{-3}^2, \bar{x}_{-1}^2, \bar{x}_1^2, \bar{x}_3^2, \ldots) }
).
\end{align*}
\end{thm}
More precise statements are given in proposition \ref{prop:AX_injective} and theorem \ref{thm:A+-X_homology}. Thus $\widehat{HS}(\An, F_{0,2})$ is naturally a non-free rank-2 module over the ``fermionic" polynomial ring $H(\X)$, and ``fermionic" behaviour persists in this case. 

As we will see, the two summands correspond to whether the open string in a string diagram on $(\An, F_{0,2})$ runs around the annulus in a positive or negative direction. In this way, to continue with loose physical analogies, ``open strings with different chirality do not interact". These computations occupy sections \ref{sec:two_points_single_boundary} and \ref{sec:homology_A_+_tensor_X} of this paper.

Next we consider annuli with four marked points, two on each boundary component $(\An, F_{2,2})$. This is the most difficult case we consider. While we do not have an explicit generators-and-relations description of homology in this case, we can detail a significant amount of structure. The computations occupy the longest section of the paper, section \ref{sec:annuli_two_marked_points_both_boundaries}.

We summarise some of our description of $\widehat{HS}(\An, F_{2,2})$ as follows.
\begin{thm}
\label{thm:22_annulus_description}
\
\begin{enumerate}
\item
The chain complex $\widehat{CS}(\An,F_{2,2})$ naturally splits into two summands
\[
\widehat{CS}(\An,F_{2,2}) \cong (\A \otimes \X \otimes \B) \oplus ( \C \otimes \X \otimes \D ),
\]
corresponding respectively to ``insular" string diagrams in which the open strings begin and end at the same boundary component, and ``traversing" string diagrams in which the open strings run between distinct boundary components. The first summand is a subcomplex, indeed a differential $\X$-module; the second is not.
\item
The subcomplex $\A \otimes \X \otimes \B$ naturally splits as a direct sum of four further subcomplexes (also differential $\X$-submodules), corresponding to the direction of open strings:
\begin{align*}
\A \otimes \X \otimes \B &\cong \left( \A_+ \oplus \A_- \right) \otimes \X \otimes (\B_+ \oplus \B_-) \\
& \cong (\A_+ \otimes \X \otimes \B_+) \oplus (\A_+ \otimes \X \otimes \B_-) \oplus (\A_- \otimes \X \otimes \B_+) \oplus (\A_- \otimes \X \otimes \B_-).
\end{align*}
The homology of each subcomplex $\A_\pm \otimes \X \otimes \B_\pm$ is naturally an $H(\X)$-module; they are non-free $H(\X)$-modules of rank  $\infty, 1, 1, \infty$ respectively.
\item
The summand $\A \otimes \X \otimes \B$ contains cycles which are not boundaries in $\widehat{CS}(\An, F_{2,2})$, and which are not homologous to elements of $\C \otimes \X \otimes \D$.
\item
The summand $\C \otimes \X \otimes \D$ contains cycles which are not boundaries in $\widehat{CS}(\An, F_{2,2})$, and which are not homologous to elements of $\A \otimes \X \otimes \B$.
\end{enumerate}
\end{thm}
The algebraic objects $\A, \A_\pm, \B, \B_\pm, \C, \D, \X$ will be defined in due course. Broadly, these objects keep track of various topological classes of curves in string diagrams on $(\An, F_{2,2})$. The last two parts of the theorem show that both the ``insular $\A$--$\B$ sector" and ``traversing $\C$--$\D$ sector" of $\widehat{HS}(\An,F_{2,2})$ contribute nontrivially to homology. We shall give such homology classes explicitly.

Note also that, even though $\widehat{HS}(\An, F_{2,2})$ is not an $H(\X)$-module, the ``$\A$--$\B$ sector" is, and the ``fermionic" structure of a module over ``fermionic polynomials" $H(\X)$ appears once more.

Part (i) is established in section \ref{sec:ABCDX_chain_complex}, as is the decomposition of part (ii). The calculations of the $H(\A_\pm \otimes \X \otimes \B_\pm)$ are involved and occupy sections \ref{sec:ABX_homology} to \ref{sec:completing_insular_calculation}, where precise results are obtained in theorems \ref{thm:++_homology} and \ref{thm:+-_homology}. Parts (iii) and (iv) are elaborated in section \ref{sec:towards_full_22_homology} and shown in sections \ref{sec:diagonal_sums} and \ref{sec:homomorphism_from_disc}. In sections \ref{sec:diagonal_sums} to \ref{sec:homomorphism_from_disc} we also develop several tools for analysing the $\C \otimes \X \otimes \D$ part of the complex.

\subsection{Adding more marked points: creation operators and doubling}

As we do not have an explicit description of $\widehat{HS}(\An,F_{2,2})$, we cannot expect an explicit description of $\widehat{HS}(\An, F_{2m,2n})$ for larger $(2m,2n)$. However, we can do the next best thing: we can give a complete description of all $\widehat{HS}(\An, F_{2m,2n})$ either explicitly with generators and relations, or \emph{in terms of} $\widehat{HS}(\An,F_{2,2})$.

Indeed, once we have $2$ marked points on a boundary component, as we add more marked points to that boundary component, we can describe the effect on $\widehat{HS}$ explicitly. To this end we will define \emph{creation and annihilation operators}, as used in \cite{Me09Paper, Me10_Sutured_TQFT, Mathews_Schoenfeld12_string}, in section \ref{sec:creation_annihilation}. 

In section \ref{sec:effect_of_creation} we give explicit isomorphisms on string homology; a precise statement is theorem \ref{thm:HS_doubling}
\begin{thm}
\label{thm:tensor_with_Z_2_squared}
Let $(\Sigma,F)$ be an alternating marked surface and suppose $F'$ is an alternating marking obtained from $F$ by adding two marked points on a boundary component already containing marked points. Then
\[
\widehat{HS} (\Sigma, F') \cong \left( \Z_2 \oplus \Z_2 \right) \otimes_{\Z_2} \widehat{HS}(\Sigma, F)
\]
\end{thm}

Repeated use of the above theorem gives a description of these vector spaces for general alternating $(\An, F_{2m,2n})$, as follows. The precise formulations are propositions \ref{prop:F_0_2n+2_homology} and \ref{prop:F_2m+2_2n+2_homology}.
\begin{thm} \
\label{thm:general_annulus_computations}
\begin{enumerate}
\item
If $n \geq 0$ then
\[
\widehat{HS}(\An, F_{0,2n+2}) \cong \left( \Z_2 \oplus \Z_2 \right)^{\otimes n } \otimes_{\Z_2} \left( \bar{x}_1 H(\X) \oplus \bar{x}_{-1} H(\X) \right).
\]
\item
If $m,n \geq 0$ then
\[
\widehat{HS}(\An, F_{2m+2,2n+2}) \cong \left ( \Z_2 \oplus \Z_2 \right)^{\otimes (m+n)} \otimes_{\Z_2} \widehat{HS}(\An, F_{2,2}).
\]
\end{enumerate}
\end{thm}

\subsection{Relations to Floer-theoretic invariants}
\label{sec:relations_to_floer}

As mentioned in section \ref{sec:alternating_surfaces}, there are reasons to expect a connection between the ``string homology" discussed in this paper, and Floer-theoretic invariants such as sutured Floer homology and embedded contact homology, as well as a contact-topological invariant, namely the Grothendieck group of the contact category, consisting of ``sutures modulo bypass triples".

In \cite{Mathews_Schoenfeld12_string} we showed that for an alternating marked disc $(D^2, F)$, our ``string homology" $\widehat{HS}(D^2, F)$, the Grothendieck group of ``sutures modulo bypasses", and the sutured Floer homology of $(D^2 \times S^1, F \times S^1)$ are naturally isomorphic:
\[ \begin{array}{ccccc}
\widehat{HS}(D^2,F) &\cong& \frac{ \widehat{CS}_{sut}(D^2,F) }{ \widehat{\Byp}(D^2,F) } &\cong& SFH(D^2 \times S^1, F \times S^1) 
\\
\text{String homology} && \text{Sutures mod bypasses} && \text{Sutured Floer homology} 
\end{array} \]

The computations of this paper, outlined in this introduction, show that such a direct isomorphism no longer exists for annuli, and the relationship is more complex. 

For instance, suppose, following \cite{Tian12_Uq, Tian13_UT}, we regard the annulus with no marked points $(\An, \emptyset)$ as corresponding to the sutured 3-manifold $(\An \times S^1, \Gamma)$, where $\Gamma$ consists of a pair of sutures of the form $C \times \{\cdot\}$, oppositely oriented, on each boundary component $C \times S^1$ of $\An \times S^1$. Then in fact the sutured manifold $(\An \times S^1, \Gamma)$ is homeomorphic to $(\An \times S^1, F_{2,2} \times S^1)$ and computations of \cite{HKM09} or \cite{HKM08} give $SFH(\An \times S^1, \Gamma) \cong \Z_2^4$. With $\Z$ or twisted coefficients $SFH(\An \times S^1, \Gamma)$ is a free $\Z$ or $\Z[q,q^{-1}]$-module of rank $4$ \cite{Me14_twisty_itsy_bitsy}. But on the other hand, $\widehat{HS}(\An, \emptyset) = H(\X)$ is the ``fermionic polynomial algebra", which has infinite rank as a vector space or algebra over $\Z_2$.
\[
\widehat{HS}(\An, \emptyset) = H(\X) \cong \frac{\Z_2[\ldots, \bar{x}_{-3}, \bar{x}_{-1}, \bar{x}_1, \bar{x}_3, \ldots] }{ (\ldots, \bar{x}_{-3}^2, \bar{x}_{-1}^2, \bar{x}_1^2, \bar{x}_3^2, \ldots) }
\]
There is a basis of sutured Floer homology in this case consisting of contact elements of contact structures with dividing sets given by the string diagrams $1$, $x_1$, $x_{-1}$, $x_1 x_{-1}$. In some sense, then, $\widehat{HS}(\An, \emptyset)$ gives a vast generalisation of $SFH$, with higher order structure from the ``higher spin" strings $\bar{x}_{\pm 3}$ and above. Setting all $\bar{x}_{\pm 3}, \bar{x}_{\pm 5}, \ldots$ to $0$, $\widehat{HS}(\An, \emptyset)$ reduces to $SFH(\An \times S^1, \Gamma)$ with $\Z_2$ coefficients.

Similarly, if we consider the annulus with two marked points on a single boundary component $(\An, F_{0,2})$, the corresponding sutured manifold $(\An \times S^1, \Gamma)$ is a \emph{basic slice} \cite{Hon00I}, which has $SFH$ isomorphic to $\Z_2^4$ \cite{HKM09}. But again our computation of $\widehat{HS}$ is a far richer structure,
\[
\widehat{HS}(\An, F_{0,2}) \cong 
\bar{x}_1 \frac{\Z_2[\ldots, \bar{x}_{-3}, \bar{x}_{-1}, \bar{x}_1, \bar{x}_3, \ldots] }{ (\ldots, \bar{x}_{-3}^2, \bar{x}_{-1}^2, \bar{x}_1^2, \bar{x}_3^2, \ldots) } \oplus \bar{x}_{-1} \frac{\Z_2[\ldots, \bar{x}_{-3}, \bar{x}_{-1}, \bar{x}_1, \bar{x}_3, \ldots] }{ (\ldots, \bar{x}_{-3}^2, \bar{x}_{-1}^2, \bar{x}_1^2, \bar{x}_3^2, \ldots) }.
\]
Again setting all $\bar{x}_3, \bar{x}_5, \ldots$ to $0$ reduces this homology to $\Z_2^4$ in a manner which appears to be consistent with contact elements generating $SFH$.

Finally, $(\An, F_{2,2})$ has corresponding sutured manifold $(\An \times S^1, F_{2,2} \times S^1)$ which has $SFH$ free of rank $4$ (with $\Z_2$, $\Z$ or twisted $\Z[q,q^{-1}]$ coefficients), as computed in \cite{HKM08, Me11_torsion_tori, Me12_itsy_bitsy}. But $\widehat{HS}(\An, F_{2,2})$ is more complicated. Although we cannot compute the homology explicitly, in section \ref{sec:homomorphism_from_disc} we construct a homomorphism $\Phi$ from $\widehat{HS}(\An, F_{2,2})$ onto $\Z_2^4$, with nonzero kernel.

The constructions of this paper, therefore, give an algebraic object with an elementary definition, based on curves on surfaces and their intersections, but which may contain strictly more information than Floer-theoretic invariants of corresponding 3-manifolds. The above remarks suggest that some subspace, or quotient, of $\widehat{HS}(\Sigma, F)$ may be isomorphic to the corresponding sutured Floer homology, or Grothendieck group of ``sutures modulo bypass triples". They also suggest potential ``higher order structure" which may be found in sutured Floer homology, or the equivalent Floer-theoretic invariant of embedded contact homology, or contact categories.

On a different note, as remarked in our previous work \cite{Mathews_Schoenfeld12_string}, the construction of our differential is strongly reminiscent of the work of Goldman \cite{Goldman86} and Turaev \cite{Turaev91}. Goldman defined a Lie bracket on the abelian group $\Z \widehat{\pi}$ freely generated by homotopy classes of loops on a surface --- the bracket resolves crossings between two loops, much like our differential. From this bracket, now known as the \emph{Goldman bracket}, he defined a Lie algebra homomorphism to the space of smooth real-valued functions on the Teichm\"{u}ller space of the surface (and more generally to spaces of flat $G$-connections). Under this map the Goldman bracket corresponds to the Poisson bracket on real-valued functions with respect to the Weil-Petersson symplectic form on Teichm\"{u}ller space. Turaev went further and defined a cobracket on the same abelian group $\Z \widehat{\pi}$, defined by resolving self-intersections of a loop. This cobracket, now known as the \emph{Turaev cobracket}, gives $\Z \widehat{\pi}$ a Lie bialgebra structure and we obtain a Poisson algebra homomorphism to the space of smooth real-valued functions on Teichm\"{u}ller space. 

Our differential, taking a string diagram and resolving its crossings --- both intersections between distinct strings and self-intersections of each string --- thus incorporates both the Goldman bracket and Turaev cobracket. 

A similar combination of both Lie bialgebra operations also arises in the \emph{symplectic field theory} of Eliashberg--Givental--Hofer \cite{EGH}. Cieliebak--Latschev in \cite{CL} studied symplectic field theory with Lagrangian boundary conditions, and found that the master equation of symplectic field theory 
$e^{{\bf F}} \overleftarrow{ {\bf H}^+ } - \overrightarrow{ {\bf H}^-} e^{{\bf F}} = 0$ is generalised in the Lagrangian boundary case to the equation
\[
e^{{\bf F}} \overleftarrow{ {\bf H}^+ } - \overrightarrow{ {\bf H}^-} e^{{\bf F}} = (\partial + \Delta + \hbar \nabla) e^{{\bf L}},
\]
where $\nabla$ and $\Delta$ are essentially the Goldman bracket and Turaev cobracket respectively. In other words, our differential resembles boundary phenomena in symplectic field theory. As discussed in \cite{Mathews_Schoenfeld12_string}, the relationships known to exist between symplectic field theory, embedded contact homology, and Heegaard Floer homology suggest that none of these similarities is a coincidence.

For now, however, we leave the pursuit of these connections to subsequent work, and for the remainder of this paper concentrate on the calculations of various $\widehat{HS}(\An, F)$, hoping that the above considerations provide sufficient motivation to follow them.

\subsection{Structure of this paper}

We will proceed by first establishing some preliminaries and definitions, in section \ref{sec:preliminaries}. Much of this section recalls \cite{Mathews_Schoenfeld12_string} but there is some significant generalisation of basic notions and re-working of definitions in the present context. For instance, our notion of \emph{marked surface} in section \ref{sec:marked_surfaces} allows boundary components without marked points. Also in section \ref{sec:chain_complex_and_subspaces} we define the chain complex $\widehat{CS}(\Sigma,F)$ as a \emph{quotient} by contractible loops. And in sections  \ref{sec:algebraic_structure} and \ref{sec:differential_algebra} we discuss additional algebraic structure.

In section \ref{sec:annuli_no_marked_points} we consider the annulus $(\An, F_{0,0}=\emptyset)$ with no marked points, and calculate the resulting homology. The computation is rather long and involves a decomposition over odd integers (section \ref{sec:odd_integer_decomposition}), a notion of ``decay" and ``fusion" operators (sections \ref{sec:decay_chains}--\ref{sec:creation_operators}), and Weyl algebra representations (section \ref{sec:Weyl_hierarchy}).

In section \ref{sec:annuli_two_marked_points} we turn to annuli with two marked points. There are two cases to consider. In sections \ref{sec:one_marked_point_each_boundary}--\ref{sec:source_creation} we consider the case of one marked point on each boundary component; this is the remaining case we need, in section \ref{sec:non-alternating_annuli}, to prove homology is zero for non-alternating annuli. Then in sections \ref{sec:two_points_single_boundary}--\ref{sec:homology_A_+_tensor_X} we consider two marked points on a single boundary component.

Section \ref{sec:annuli_two_marked_points_both_boundaries} is by far the longest and most difficult part of the paper; it consists of computations relating to $\widehat{HS}(\An, F_{2,2})$. In sections \ref{sec:ABCDX_chain_complex}--\ref{sec:insular_differential} we describe the chain complex and differential. Then in sections \ref{sec:ABX_homology}--\ref{sec:completing_insular_calculation} we present a long argument to calculate the homology of the ``insular" part of the chain complex. Finally in sections \ref{sec:towards_full_22_homology}--\ref{sec:homomorphism_from_disc} we consider the full homology $\widehat{HS}(\An, F_{2,2})$, make some general statements, and develop some tools to analyse it, including a \emph{diagonal sum sequence} (section \ref{sec:diagonal_sums}) and homomorphisms to and from a simpler chain complex (section \ref{sec:homomorphism_from_disc}).

Finally, in section \ref{sec:adding_marked_points} we consider adding further marked points, defining creation operators (section \ref{sec:creation_annihilation}) and analysing their effect on $\widehat{HS}$ (section \ref{sec:effect_of_creation}), before giving results for $\widehat{HS}(\An, F_{0,2n+2})$ and $\widehat{HS}(\An, F_{2m+2,2n+2})$ (section \ref{sec:results_for_annuli_with_more}).

\section{Preliminaries and definitions}
\label{sec:preliminaries}

\subsection{Marked surfaces and string diagrams}
\label{sec:marked_surfaces}

Here, as elsewhere, $\Sigma$ denotes a compact oriented surface with nonempty boundary.

\begin{defn}\
\begin{enumerate}
\item
A \emph{(weak) marking} $F$ on $\Sigma$ is a set of $2n$ distinct points on $\partial \Sigma$, where $n \geq 0$, with $n$ points $F_{in}$ labelled ``in" and $n$ points $F_{out}$ labelled ``out". The pair $(\Sigma, F)$ is called a \emph{(weakly) marked surface}.
\item
A \emph{strong marking} $F$ on $\Sigma$ is a weak marking with at least one point on each component of $\Sigma$. The pair $(\Sigma,F)$ is called a \emph{strongly marked surface}.
\end{enumerate}
\end{defn}
Note $F = F_{in} \sqcup F_{out}$. The points of $F$ are \emph{marked points}; the \emph{sign} of a marked point is ``in" or ``out" accordingly as it lies in $F_{in}$ or $F_{out}$. Note that a weakly marked surface may have no marked points at all. However we always require $\Sigma$ to have nonempty boundary.

In \cite{Mathews_Schoenfeld12_string}, we only considered strong markings and strongly marked surfaces, referring to them as ``markings" and ``marked surfaces". However in this paper we use the adjectives ``strong" and ``weak" to distinguish the two types of markings. When applied without an adjective, by a \emph{marking} or \emph{marked surface} we mean a weak one.

\begin{defn}
A marking (strong or weak) is called \emph{alternating} if for each component of $\partial \Sigma$, the points of $F$ are alternately labelled (in, out, ..., in, out).
\end{defn}
In the weak case, a boundary component with no marked points is alternating (``the null alternation"). The empty marking $F = \emptyset$ is always an alternating weak marking.

An alternating marking (strong or weak) has an even number of points on each boundary component. An alternating strong marking has at least 2 points on each boundary component. 

An alternating strong marking forms a natural boundary condition for \emph{sutures}, as discussed in \cite{Me10_Sutured_TQFT, Me12_itsy_bitsy}; see also \cite{Zarev09}. Sutures can be regarded as dividing sets on a convex surface in a contact 3-manifold \cite{Gi91}. A boundary component $C$ is regarded as Legendrian, with Thurston-Bennequin number $-\frac{1}{2} |C \cap F|$ \cite{Hon00I}.  However in this paper we only consider sutures in passing.

\begin{defn}
Let $(\Sigma,F)$ be a (weakly) marked surface. A \emph{string diagram} $s$ on $(\Sigma,F)$ is an immersed oriented compact 1-manifold in $\Sigma$ such that $\partial s = F$, as signed 0-manifolds, with all self-intersections in the interior of $\Sigma$.

The components of a string diagram are called \emph{strings}; arc components are called \emph{open strings}, and closed curve components are called \emph{closed strings}.
\end{defn}
By $\partial s = F$ as signed 0-manifolds, we mean that strings point in and out of $\Sigma$ at $F_{in}$ and $F_{out}$ respectively. Clearly each open string runs from a point of $F_{in}$ to a point of $F_{out}$.

The condition that all self-intersections occur in the interior of $\Sigma$ avoids unnecessary technicalities. In particular it ensures that precisely one arc of $s$ begins (resp. ends) at each point of $F_{in}$ (resp. $F_{out}$). Generically a string diagram contains only transverse double self-intersections. 

Note compactness implies that a string diagram only has finitely many strings. (In \cite{Mathews_Schoenfeld12_string} we did not require string diagrams to be compact; but once contractible curves on a disc are ruled out, compactness follows for free. Here we require compactness explicitly.)

In \cite{Mathews_Schoenfeld12_string} we defined several notions of homotopy for string diagrams; in this paper we consider only homotopy per se (not regular homotopy, spin homotopy, or ambient isotopy). Two string diagrams $s_0, s_1$ on a (weakly) marked surface $(\Sigma, F)$ are \emph{homotopic} if there is a homotopy relative to endpoints from $s_0$ to $s_1$. (Such a homotopy may introduce or remove self-intersections in the string diagram, it need not be through immersions, and it may change the writhe of strings.) Although strictly speaking the string diagram is an explicit immersion of a disjoint union of copies of $S^1$ and $[0,1]$ into $(\Sigma,F)$, we abuse notation and identify this immersion with its image in $\Sigma$.

It is useful to consider various sets of string diagrams on a (weakly) marked surface $(\Sigma,F)$.
\begin{defn}\
\label{def:set_of_string_diagrams}
\label{def:more_sets_of_string_diagrams}
\begin{enumerate}
\item
Let $\S(\Sigma,F)$ denote the set of homotopy classes of string diagrams on $(\Sigma,F)$.
\item
Let $\S^O (\Sigma,F)$ denote the set of  homotopy classes of string diagrams containing only open strings.
\item
Let $\S^C (\Sigma,F)$ denote the set of homotopy classes of string diagrams containing a contractible closed curve.
\end{enumerate}
\end{defn}

\subsection{The chain complex}
\label{sec:chain_complex_and_subspaces}

Following \cite{Mathews_Schoenfeld12_string} we define a chain complex $(\widehat{CS}(\Sigma,F), \partial)$ out of homotopy classes of string diagrams. It is described in terms of free vector spaces over the sets of definition \ref{def:more_sets_of_string_diagrams}. For any set $S$, denote by $\Z_2 \langle S \rangle$ the free $\Z_2$ vector space on $S$, i.e. with basis $S$.

\begin{defn}
Let $(\Sigma,F)$ be a (weakly) marked surface. The $\Z_2$-vector space $\widehat{CS}(\Sigma,F)$ is
\[
\widehat{CS}(\Sigma,F) = \frac{ \Z_2 \langle \S (\Sigma, F) \rangle }{ \Z_2 \langle \S^C (\Sigma, F) \rangle }.
\]
\end{defn}
As $\S^C (\Sigma, F) \subseteq \S (\Sigma, F)$, we have a well-defined quotient $\Z_2$-vector space. A basis of $\widehat{CS}(\Sigma,F)$ is given by homotopy classes of string diagrams on $(\Sigma,F)$ which contain no contractible closed curves. (In fact in \cite{Mathews_Schoenfeld12_string} we defined $\widehat{CS}(\Sigma, F)$ this way.) Our definition of $\widehat{CS}(\Sigma,F)$ as a quotient makes clear that any string diagram $s$ on $(\Sigma, F)$ can (by taking its homotopy class) be regarded as an element of $\widehat{CS}(\Sigma, F)$; that element is zero precisely when $s$ contains a contractible closed curve. (In \cite{Mathews_Schoenfeld12_string} we had to enforce this condition by fiat.)

Reasons for setting contractible closed curves to zero are given in \cite{Mathews_Schoenfeld12_string}. Briefly: the triviality of contractible closed curves is consistent with the nature of overtwisted discs in contact topology; moreover we find that we \emph{may} allow contractible curves, but at the cost of restricting to spin homotopy classes, obtaining a different vector space $CS^\infty(\Sigma,F)$. (In this paper, as we do not consider spin homotopy, we do not consider $CS^\infty (\Sigma,F)$ either. We can however note that $CS^\infty$ may be defined for weakly marked surfaces.) 

We next define the differential on $\widehat{CS}(\Sigma,F)$. Let $s$ be a string diagram on a marked surface $(\Sigma,F)$. Assuming $s$ is generic, all intersection points of $s$ are transverse double crossings; let $x$ be such an intersection point. As the curves of $s$ are oriented, there is a string diagram $r_x(s)$ obtained by resolving $s$ at $x$, as shown in figure figure \ref{fig:resolving_crossing}. 
\begin{defn}
Let $s$ be a string diagram. Then
\[
\partial s = \sum_{x \text{ crossing of } s} r_x (s).
\]
\end{defn}
In \cite{Mathews_Schoenfeld12_string} we showed that $\partial$ is well-defined on homotopy classes of string diagrams without contractible loops, so that we obtain an operator
\[
\partial \; : \; \widehat{CS}(\Sigma,F) \To \widehat{CS}(\Sigma,F).
\]
Moreover, we showed that $\partial^2 = 0$. The proofs in \cite{Mathews_Schoenfeld12_string} for strongly marked surfaces extend without change to the weak case. Essentially, $\partial^2 = 0$ holds since $\partial^2 s$ consists of a sum of string diagrams, each obtained by resolving two crossings $x$ and $y$ of a string diagram $s$; but $x$ and $y$ may be resolved in either order, so terms come in pairs and mod $2$ the result is $0$.

Thus for any (weakly) marked surface $(\Sigma, F)$ we have a chain complex $(\widehat{CS}(\Sigma, F), \partial)$.
\begin{defn}
Let $(\Sigma, F)$ be a weakly marked surface. The \emph{string homology} of $(\Sigma, F)$ is
\[
\widehat{HS}(\Sigma, F) = \frac{ \ker \partial }{ \im \partial }.
\]
\end{defn}
Clearly $\widehat{HS}(\Sigma, F)$ is a $\Z_2$-vector space.

\subsection{Algebraic structure}
\label{sec:algebraic_structure}

All of the above essentially appears in \cite{Mathews_Schoenfeld12_string}, apart from the extension to weakly marked surfaces. In this section we introduce some multiplicative structures on the various vector spaces involved, making them into modules and/or algebras. This material has not appeared before, so far as we are aware.

For a (weakly) marked surface with \emph{no} marked points, $\widehat{CS}(\Sigma, \emptyset)$ is naturally a \emph{ring}, which we denote $\X(\Sigma)$. Multiplication in this ring is given by disjoint union of string diagrams, which necessarily only contain closed strings. For a general (weakly) marked surface, $\widehat{CS}(\Sigma,F)$ is a \emph{free $\X(\Sigma)$-module}, with basis given by $\S^O(\Sigma, F)$, string diagrams containing only open strings. We can describe these structures formally in terms of the fundamental group and/or loop space of $\Sigma$. 

Consider the free homotopy classes of closed loops on $\Sigma$; these correspond to the connected components $\pi_0 (\Lambda \Sigma)$ of the free loop space $\Lambda \Sigma$. When $\Sigma$ is connected, these also correspond to conjugacy classes $\Conj \pi_1 (\Sigma)$ of the fundamental group $\pi_1 (\Sigma)$.

We define the algebra $\tilde{\X}(\Sigma)$ to be the symmetric $\Z_2$-algebra on the set $\pi_0 (\Lambda \Sigma)$. That is, $\tilde{X}(\Sigma) = \Z_2[\pi_0 (\Lambda \Sigma)]$. Elements of $\tilde{X}(\Sigma)$ can be regarded as polynomials in ``variables" given by free homotopy classes of closed curves on $\Sigma$. As a $\Z_2$-vector space, $\tilde{\X}(\Sigma)$ has basis given by free homotopy classes of closed multicurves on $\Sigma$. Multiplication in $\tilde{\X}(\Sigma)$ corresponds to disjoint union of (free homotopy classes of) multicurves. When $\Sigma$ is connected, $\tilde{X}(\Sigma) \cong \Z_2 [ \Conj \pi_1 (\Sigma) ]$.

The ring $\tilde{X}(\Sigma)$ has an ideal $\pi_0(\Sigma) \tilde{X}(\Sigma)$ generated by the homotopy classes of closed contractible curves on $\Sigma$. (There is one such homotopy class for each connected component of $\Sigma$.) The ideal $\pi_0 (\Sigma) \tilde{X}(\Sigma)$ is a free $\Z_2$-vector subspace of $\tilde{X}(\Sigma)$ with basis $\S^C(\Sigma, \emptyset)$, the homotopy classes of closed multicurves containing a contractible closed curve. When $\Sigma$ is connected, this ideal is isomorphic to $\{1\} \Z_2[ \Conj \pi_1 (\Sigma) ] \subset \Z_2 [\Conj \pi_1 (\Sigma)]$.

The algebra $\X(\Sigma)$ is then the quotient of $\tilde{\X}(\Sigma)$ by this ideal.
\begin{defn}
The $\Z_2$-algebra $\X(\Sigma)$ is
\[
\X(\Sigma) = \frac{ \tilde{\X}(\Sigma) }{ \pi_0 (\Sigma) \tilde{\X}(\Sigma) }
= \frac{ \Z_2 [ \pi_0 (\Lambda \Sigma) ] }{ \pi_0 (\Sigma) \Z_2 [ \pi_0 (\Lambda \Sigma)] }.
\]
\end{defn}
As a $\Z_2$-algebra, $\X(\Sigma)$ is freely generated by free homotopy classes of non-contractible closed curves on $\Sigma$; contractible curves have been set to zero. As a $\Z_2$-vector space, $\X(\Sigma)$ has basis given by free homotopy classes of closed multicurves on $\Sigma$ with no contractible components. This means that despite our algebraic circumvolutions, $\X(\Sigma)$ is just isomorphic to $\widehat{CS}(\Sigma, \emptyset)$ as a $\Z_2$-vector space; indeed it endows $\widehat{CS}(\Sigma,\emptyset)$ with the structure of a $\Z_2$-algebra.
\[
\widehat{CS}(\Sigma, \emptyset) \cong \X(\Sigma).
\]

Turning to a general (possibly nonempty) marking $F$, it is useful to consider a vector space, like $\widehat{CS}$, but using $\S^O (\Sigma, F)$ (definition \ref{def:more_sets_of_string_diagrams}) rather than $\S(\Sigma,F)$.
\begin{defn}
For a (weakly) marked surface $(\Sigma, F)$, the vector space $\widehat{CS}^O (\Sigma, F)$ is the subspace of $\widehat{CS}(\Sigma, F)$ generated by homotopy classes of string diagrams with no closed curves.
\[
\widehat{CS}^O (\Sigma, F) = \frac{ \Z_2 \langle \S^O (\Sigma, F) \rangle }{ \Z_2 \langle \S^C (\Sigma, F) \cap \S^O (\Sigma, F) \rangle } 
= \Z_2 \langle \S^O (\Sigma, F) \rangle
\]
\end{defn}
(Note $\S^C (\Sigma, F) \cap \S^O (\Sigma, F) = \emptyset$, giving the equalities above; having no closed curves implies no contractible closed curves!)
Thus, $\widehat{CS}^O (\Sigma, F)$ is freely generated a s $\Z_2$ vector space by string diagrams with open strings only.

Now suppose we have a (homotopy class of) string diagram $s$ without contractible curves, so $s$ is a generator of $\widehat{CS}(\Sigma,F)$. We can decompose $s$ into its open and closed string components. Precisely, there exists a (free homotopy class of) closed multicurve without contractible components $x \in \X(\Sigma)$, and a (homotopy class of) open string diagram $s^O \in \widehat{CS}^O (\Sigma,F)$, such that $s$ is the union of the curves of $x$ and $s^O$. Moreover this $x$ and $s^O$ are unique. In keeping with the notion that multiplication corresponds to disjoint union of string diagrams, we may write $s = x s^O$. Moreover, this notation is compatible with the multiplication in $\X(\Sigma)$: if $x = x_1 x_2$ for $x_1, x_2 \in \X(\Sigma)$, we have $x s^O = (x_1 x_2) s^O = x_1 (x_2 s^O)$. Thus, $\widehat{CS}(\Sigma, F)$ naturally has the structure of an $\X(\Sigma)$-module; indeed it is a free $\X(\Sigma)$-module with basis $\S^O(\Sigma,F)$. As $\widehat{CS}^O(\Sigma,F)$ is freely generated by $\S^O(\Sigma,F)$, we immediately have the following result.
\begin{lem}
\label{lem:CS_as_module}
For any (weakly) marked surface $(\Sigma, F)$,
\[
\widehat{CS}(\Sigma, F) \cong \X(\Sigma) \otimes_{\Z_2} \widehat{CS}^O (\Sigma, F).
\]
\qed
\end{lem}

In this paper we only need to take $\Sigma = \An$. In this case $\pi_0 (\Lambda \An) \cong \pi_1 (\An) \cong \Conj \pi_1 (\An) \cong \Z$, as the fundamental group is abelian. For any integer $n$, let us write $x_n$ for the free homotopy class of a closed curve in $\An$ which runs $n$ times around the core of the annulus, i.e. corresponding to $n \in \Z \cong \pi_1 (\An)$. So $\tilde{\X}(\An) = \Z_2 [ \ldots, x_{-2}, x_{-1}, x_0, x_1, x_2, \ldots ]$ is the free polynomial algebra in $x_n$, over all integers $n$. 

The only homotopy class of closed curve on $\An$ which is contractible is $x_0$, corresponding to $0 \in \Z \cong \pi_1 (\An)$. So $\X(\An)$ is the quotient of $\tilde{\X}(\An)$ by the principal ideal generated by $x_0$; this corresponds to setting $x_0 = 0$. As a $\Z_2$-algebra, we thus have
\[
\X(\An) \cong \Z_2 [\ldots, x_{-2}, x_{-1}, x_1, x_2, \ldots],
\]
i.e. the free polynomial algebra in the infinitely many indeterminates $x_n$, over all integers $n \neq 0$. This ring occurs so frequently in the sequel that we simply write $\X$ for $\X(\An)$.

Applying the above to $\An$, we immediately have the following, for integers $m,n \geq 0$.
\begin{align*}
\widehat{CS}(\An, \emptyset) &\cong \X \cong \Z_2[\ldots, x_{-2}, x_{-1}, x_1, x_2, \ldots] \\
\widehat{CS}(\An, F_{2m,2n}) &\cong \X \otimes_{\Z_2} \widehat{CS}^O (\An, F_{2m,2n}) \\
&\cong \Z_2 [\ldots, x_{-2}, x_{-1}, x_1, x_2, \ldots] \otimes_{\Z_2} \widehat{CS}^O (A, F_{2m,2n}).
\end{align*}

\subsection{Goldman bracket, Leibniz rule and differential algebra}
\label{sec:differential_algebra}

For a marked surface with no marked points $(\Sigma, \emptyset)$, we have seen that $\widehat{CS}(\Sigma, \emptyset) \cong \X(\Sigma)$ is naturally both a ring, and a chain complex. Thus it appears to be a \emph{differential algebra} (indeed, a \emph{differential graded algebra} with respect to an intersection grading). However, in general the differential does not obey the Leibniz rule.  We now explain why.

Suppose we have two closed curves on $\Sigma$, with homotopy classes $x, y \in \X(\Sigma) \cong \widehat{CS}(\Sigma, \emptyset)$. Then $xy$ is the disjoint union of these two curves, and $\partial(xy)$ is obtained by resolving the intersection points of $xy$. These intersection points may be of three types: (i) self-intersections of $x$; (ii) self-intersections of $y$; and (iii) intersections of $x$ with $y$. Resolving the intersections of the first and second types gives $(\partial x) y$ and $x (\partial y)$ respectively. So the Leibniz rule $\partial(xy) = (\partial x) y + x (\partial y)$ holds if and only if there are no intersections of the third type, i.e. $x$ and $y$ are disjoint, or at least their resolutions cancel mod $2$.

Thus $\X(\Sigma) \cong \widehat{CS}(\Sigma, \emptyset)$ is a differential algebra if any closed curve in $\Sigma$ can be homotoped to be separate from any other; this occurs precisely when $\Sigma$ is a union of discs and annuli. In this case the homology $H(\X(\Sigma)) \cong \widehat{HS}(\Sigma, \emptyset)$ has the structure of a $\Z_2$-algebra.

The deviation of $\partial$ from obeying the Leibniz rule is thus measured by the intersection of distinct closed curves. To this end we recall the \emph{Goldman bracket} \cite{Goldman86}, which is in fact part of the original motivation for our differential.

\begin{defn}
Let $s,s'$ be two immersed oriented curves on $\Sigma$ intersecting transversely. Their \emph{Goldman bracket} is given by resolving the crossings of $s$ and $s'$:
\[
[s,s'] = \sum_{x \in s \cap s'} r_x (ss').
\]
\end{defn}
The proofs of lemmas 4.1 and 4.2 of \cite{Mathews_Schoenfeld12_string} apply immediately to show that the bracket is well-defined, and the homotopy class of the result depends only on the homotopy class of $s$ and $s'$. Note that the Goldman bracket applies to any immersed oriented \emph{curves}, not just closed curves, and is in fact well-defined on \emph{multicurves}.

Restricting to closed multicurves, the Goldman bracket gives an operation
\[
[\cdot, \cdot] \; : \; \X(\Sigma) \otimes \X(\Sigma) \To \X(\Sigma).
\]

Working over $\Z_2$, the Goldman bracket is both antisymmetric and symmetric. We have the following identity, for any multicurves $x$ and $y$.
\begin{equation}
\label{eqn:general_product_rule}
\partial(xy) = (\partial x) y + x (\partial y) + [x,y]
\end{equation}
This equation expresses precisely that the Goldman bracket is an obstruction to $\partial$ satisfying the Leibniz rule.

Turning to nonempty markings $F$, we have seen that $\widehat{CS}(\Sigma,F) \cong \X(\Sigma) \otimes \widehat{CS}^O (\Sigma,F)$ is a free $\X(\Sigma)$-module. We may write a (homotopy class of) string diagram $s$ on $(\Sigma,F)$ as $x s^O$, where $x \in \X(\Sigma)$ is a (homotopy class of) closed multicurve and $s^O \in \S^O(\Sigma,F)$. We have
\begin{equation}
\label{eqn:product_rule_open_closed}
\partial (x s^O) = (\partial x) s^O + x (\partial s^O) + [x, s^O],
\end{equation}
where $\partial x$ denotes the differential in $\X(\Sigma)$. Thus, if the Leibniz rule is satisfied, then $[x, s^O]=0$ for every closed curve $x$ and every open string diagram $s^O$ on $(\Sigma, F)$. This implies that for connected $(\Sigma,F)$ we have $\Sigma$ a disc, or $\Sigma = \An$ and all points of $F$ are on the same boundary component of $\An$. In fact in these cases $\X(\Sigma)$ is a differential algebra, and the Leibniz rule is satisfied, so that $\widehat{CS}(\Sigma,F)$ is a \emph{differential $\X(\Sigma)$-module}. It follows that the homology $\widehat{HS}(\Sigma,F)$ is an $H(\X(\Sigma))$-module.

When $\Sigma = \An$, we denote the alternating marking points consisting of $2m, 2n \geq 0$ points on the respective boundary components by $F_{2m,2n}$. We then have the following.
\begin{prop}
The $\X$-module $\widehat{CS}(\An, F_{2m,2n})$ is a differential $\X$-module if and only if $m=0$ or $n=0$. In these cases $\widehat{HS}(\An, F_{2m,2n})$ is an $H(\X)$-module.
\qed
\end{prop}
This proves proposition \ref{prop:HX-modules}

In the next section we compute $H(\X)$.

\section{Strings on annuli with no marked points}
\label{sec:annuli_no_marked_points}

We now turn our attention to the marked surface $(\An, \emptyset)$. From section \ref{sec:algebraic_structure} above, we have, as a $\Z_2$-algebra
\[
\widehat{CS}(\An, \emptyset) \cong \X \cong \Z_2[\ldots, x_{-2}, x_{-1}, x_1, x_2, \ldots].
\]

\subsection{Description of the differential}

We first compute the differential. As discussed in section \ref{sec:differential_algebra}, the Goldman bracket vanishes on $(\An, \emptyset)$ so the Leibniz rule is satisfied and $\X$ is a differential algebra.
\begin{lem}
\label{lem:differential_no_marked_pts}
The differential $\partial$ on $\widehat{CS}(A, \emptyset)$ satisfies the following properties.
\begin{enumerate}
\item
For each integer $k \neq 0$,
\[
\partial x_{2k} = x_k^2.
\]
\item
For each integer $k$,
\[
\partial x_{2k+1} = 0.
\]
\end{enumerate}
\end{lem}
These two properties together with the Leibniz rule define $\partial$ completely. As $\partial$ is $\Z_2$-linear, it suffices to define it on (homotopy classes of) string diagrams. Individual strings are dealt with by (i) and (ii) above, and where there are several curves, they are resolved individually via the Leibniz rule.

\begin{proof}
Consider a closed string of homotopy class $n > 0$, represented by $x_n \in \X$. We may draw this loop so that it has $n-1$ self-intersections. Resolving any of these crossings splits the loop into two loops whose homotopy classes are positive and sum to $n$; thus we have
\[
\partial x_n = \sum_{i=1}^{n-1} x_i x_{n-i}.
\]
When $n$ is odd, the number of terms in this sum is even, and terms cancel in pairs mod $2$, so $\partial x_n = 0$. When $n$ is even, $n = 2k$, we have an odd number of terms, and again terms cancel in pairs, except for the ``middle" term $x_k^2$. A similar argument calculates $\partial x_n$ when $n < 0$.
\end{proof}

Thus, computing $\widehat{HS}(A, \emptyset)$ amounts to computing the homology of the polynomial algebra $\X \cong \Z_2 [\ldots, x_{-2}, x_{-1}, x_1, x_2, \ldots]$ with respect to the differential $\partial$ defined by $\partial x_{2k} = x_k^2$, $\partial x_{2k+1} = 0$ and the Leibniz rule.

This description of $\partial$ alone is enough to understand some aspects of the ``fermionic" nature of $\widehat{HS}(A, \emptyset)$. Firstly, as $\partial x_{2k+1} = 0$, but $\partial x_{2k} \neq 0$, only ``odd-spin" strings have homology classes. And secondly, since $\partial x_{2k} = x_j^2$, in homology we have $x_j^2 = 0$; so a loop which appears twice disappears in homology --- reminiscent of the Pauli exclusion principle.

\subsection{Tensor decomposition over odd integers}
\label{sec:odd_integer_decomposition}

Our computation of $\widehat{HS}(A, \emptyset)$ is based on a tensor decomposition of the complex $\widehat{CS}(A, \emptyset)$. We can partition the nonzero integers into subsets $Z_j$, for each odd integer $j$, as follows:
\[
Z_j = \{ j, 2j, 4j, \ldots \} = \{ j \cdot 2^k \; : \; 0 \leq k \in \Z \}.
\]
Obviously any nonzero integer can be written uniquely as $j \cdot 2^k$ where $j \in \Z$ is odd and $0 \leq k \in \Z$, so
\[
\Z \backslash \{0\} = \bigsqcup_{j \text{ odd}} Z_j.
\]

Now $\X$ is a free $\Z_2$-algebra on the $x_n$ over $n \in \Z \backslash \{0\}$. Hence we can define, for each odd $j$,
\[
\X_j = \Z_2 [ \{ x_n \; : \; n \in Z_j \} ] = \Z_2 [x_j, x_{2j}, x_{4j}, \ldots].
\]
That is, $\X_j$ is the polynomial algebra in the $x_n$, over all $n \in Z_j$. Corresponding to the decomposition of $\Z \backslash \{0\}$ into the $Z_j$, we have a tensor decomposition of the $\Z_2$-vector space $\X$:
\[
\X = \bigotimes_{j \text{ odd}} \X_j.
\]
Further, since the differential acts on generators by $x_{2k} \mapsto x_k^2$ and $x_{2k+1} \mapsto 0$, each $\X_j$ is a subcomplex of $\X$; indeed, each $\X_j$ is a differential sub-algebra of $\X$. Since $\partial$ obeys the Leibniz rule, we in fact have a tensor decomposition of chain complexes and differential algebras:
\[
(\X, \partial) \cong \bigotimes_{j \text{ odd}} (\X_j, \partial).
\]

Moreover, all the $\X_j$ are all \emph{isomorphic} differential algebras. We will define a standard differential algebra $\Y$ and show each $\X_j$ is isomorphic to $\Y$.
\begin{defn}
The differential algebra $\Y$ is the polynomial algebra $\Z_2 [y_0, y_1, \ldots]$ in infinitely many indeterminates $y_i$, over all integers $i \geq 0$. The differential on $\Y$ is defined by the Leibniz rule and
\[
\partial y_i = \left\{ \begin{array}{ll} y_{i-1}^2 & i \geq 1 \\ 0 & i = 0. \end{array} \right.
\]
\end{defn}

Thus, the differential on $\Y$ sends each indeterminate $y_i$ ``down" to $y_{i-1}$, squared; and sends the ``lowest" indeterminate $y_0$ to $0$. Similarly, the differential on $\X_j$ sends each indeterminate $x_{j \cdot 2^k}$ ``down" to $x_{j \cdot 2^{k-1}}$, squared, and sends $x_j \mapsto 0$. The following lemma is then immediate.
\begin{lem}
For any odd $j$, $x_{j \cdot 2^k} \mapsto y_k$ induces an isomorphism of differential algebras $\X_j \cong \Y$.
\qed
\end{lem}

The computation of $H(\X) \cong \widehat{HS}(\An, \emptyset)$ is now essentially reduced to computing the homology $H(\Y)$ of $\Y$. In the next sections \ref{sec:decay_chains}--\ref{sec:Weyl_hierarchy} we compute $H(\Y)$ by decomposing $\Y$ in various ways.

\subsection{Homology as decay chain}
\label{sec:decay_chains}

For some intuition in computing $H(\Y)$, we can think of the indeterminates $y_k$ of $\Y$ as describing ``particles", and the differential as describing a ``binary decay cascade"
\[
\cdots \mapsto y_k \mapsto y_{k-1} \mapsto \cdots \mapsto y_1 \mapsto y_0 \mapsto 0,
\]
where each particle decays into two of the subsequent particle, $\partial y_k = y_{k-1}^2$. The differential gives the possible decays of a collection of particles.

As a $\Z_2$-vector space, $\Y$ is free with basis the monomials $y_0^{e_0} y_1^{e_1} \cdots y_{n-1}^{e_{n-1}}$, over all positive integers $n$ and all $n$-tuples $(e_0, e_1, \ldots, e_{n-1})$ of integers $e_1, e_2, \ldots, e_n \geq 0$. As a shorthand we can write $e = (e_0, e_1, \ldots, e_{n-1})$ as the vector of exponents and $y^e = y_0^{e_0} \cdots y_{n-1}^{e_{n-1}}$.

We decompose the differential $\partial$ into ``decay" operators $\alpha_k$, which send $x_k \mapsto x_{k-1}^2$. 
\begin{defn}
For each positive integer $k$, the $k$'th \emph{decay} operator $\alpha_k: \Y \To \Y$ is defined by linearity, the Leibniz rule and 
\[
\alpha_k y_k = y_{k-1}^2, \quad \alpha_k y_i = 0 \text{ for $i \neq k$}.
\]

We define $\alpha_0 = 0$.
\end{defn}

The differential $\partial$ on $\Y$ is then the sum of the $\alpha_k$, $\partial = \sum_{k \geq 0} \alpha_k$.

We next investigate the $\alpha_k$ more closely. For $k \geq 1$, the effect of $\alpha_k$ on a monomial $y_i^{e_i}$ is given by
\[
\alpha_k \left( y_k^{e_k} \right) = e_k y_{k-1}^2 y_k^{e_k - 1},
\quad
\alpha_k \left( y_i^{e_i} \right) = 0 \text{ for } i \neq k,
\]
and on a general monomial $y^e$ is given by
\begin{align*}
\alpha_k y^e &
= \alpha_k \left( \prod_{i=0}^{n-1} y_i^{e_i} \right) 
= \left( \prod_{i \neq k} y_i^{e_i} \right) \alpha_k \left( y_k^{e_k} \right) %\\
%&
= \left( \prod_{i \neq k} y_i^{e_i} \right) e_k y_{k-1}^2 y_k^{e_k - 1} \\
&= e_k y_0^{e_0} \cdots y_{k-1}^{e_{k-1}+2} y_k^{e_k - 1} \cdots y_{n-1}^{e_{n-1}} 
\end{align*}

Abusing notation, we can write this as
\[
\alpha_k y^e = e_k y_{k-1}^2 y_k^{-1} y^e.
\]
Even though $y_k^{-1}$ is obviously not in the polynomial ring $\Y$, we note that if $y_k$ appears with exponent $e_k = 0$ then the factor of $e_k$ results in $0$; and otherwise $e_k \geq 1$, so that $y_k^{-1} y^e$ is a monomial in $\Y$. The exponent of $y_k$ is reduced by one, and the exponent of $y_{k-1}$ is increased by two. We will repeat this abusive, but well-defined, notation throughout the paper as an efficient shorthand.

\begin{lem}
For all $k$, $\alpha_k^2 = 0$.
\end{lem}

\begin{proof}
Obviously $\alpha_0^2 = 0$ since $\alpha_0 = 0$, so assume $k \geq 1$. On a monomial $y^e$ we have $\alpha_k y^e = e_k y_{k-1}^2 y_k^{-1} y^e$ so (with the same abuse of notation)
\[
\alpha_k^2 y^e = \alpha_k \left( e_k y_{k-1}^2 y_k^{-1} y^e \right) = e_k (e_k - 1) y_{k-1}^4 y_k^{-2} x^e.
\]
When $e_k = 0$ or $1$ we have $e_k = 0$ or $e_k - 1 = 0$; if $e_k \geq 2$ then the above monomial makes sense. In any case, mod $2$ we have $e_k (e_k - 1) = 0$ so $\alpha_k^2 = 0$.
\end{proof}

%It follows that each $(\Y, \alpha_k)$ is a chain complex.

We can also consider two distinct $\alpha_i, \alpha_j$. It is not difficult to show they commute.

\begin{lem}
Let $0 \leq i < j$ be integers. Then $\alpha_i \alpha_j = \alpha_j \alpha_i$.
\end{lem}

\begin{proof}
When $i=0$ we have $\alpha_i = 0$ so the result trivially holds; we thus assume $i>0$. We can check directly on a monomial $y^e = y_0^{e_0} \cdots y_{n-1}^{e_{n-1}}$ that 
\[
\alpha_i \alpha_j y^e = \alpha_j \alpha_i y^e = e_i e_j y_{i-1}^2 y_i^{-1} y_{j-1}^2 y_j^{-1} y^e.
\]
\end{proof}

Although we already knew that $\partial^2 = 0$, we can now see it alternatively as follows:
\[
\partial^2 = \left( \sum_{i \geq 0} \alpha_i \right)^2 = \sum_{i \geq 0} \alpha_i^2 + \sum_{0 \leq i < j} \alpha_i \alpha_j + \alpha_j \alpha_i.
\]
The first sum is zero since each $\alpha_i^2=0$, and the second sum is zero since $\alpha_i, \alpha_j$ commute.

\subsection{Fusion operators}
\label{sec:creation_operators}
\label{sec:fusion_operators}

In order to compute the homology of $\Y$ and each $\Y_n$, we will define some ``fusion operators" $\alpha_k^*$, which partly undo, or are ``adjoint" to, the decay operators $\alpha_k$. A decay operator sends $y_k \mapsto y_{k-1}^2$; a fusion puts the two $y_{k-1}$'s back together into a $y_k$, sending $y_{k-1}^2 \mapsto y_k$.

\begin{defn}
For an integer $k \geq 1$ the $k$'th \emph{fusion operator} $\alpha_k^* : \Y \To \Y$ is defined on a monomial $y^e = y_0^{e_0} \cdots y_{n-1}^{e_{n-1}}$ by
\[
\alpha_k^* y^e = \left\{ \begin{array}{ll} 
0 & e_{k-1} \leq 1 \\
y_0^{e_0} \cdots y_{k-1}^{e_{k-1} - 2} y_i^{e_k + 1} \ldots y_{n-1}^{e_{n-1}} = y_{k-1}^{-2} y_k y^e  & e_{k-1} \geq 2
\end{array} \right.
\]
and extended $\Z_2$-linearly over $\Y$. We define $\alpha^*_0 = 0$.
\end{defn}
That is, $\alpha_k^*$ is multiplication by $y_{k-1}^{-2} y_k$, when this gives a polynomial, and is zero otherwise.

We now check how $\alpha_i^*$ interacts with the decay operators.

First, $\alpha_i^* \alpha_i$ sends $y_i^{e_i}$ via $\alpha_i$ to $e_i y_i^{e_i-1} y_{i-1}^2$, and then to $e_i y_i^{e_i}$; this holds even when we multiply by other powers of other variables $y_j$, which gives
\[
\alpha_i^* \alpha_i = e_i.
\]
Next, $\alpha_i \alpha_i^*$ will map to zero, if $e_{i-1} \leq 1$; if $e_{i-1} \geq 2$ then it sends $y_{i-1}^{e_{i-1}} y_i^{e_i}$ via $\alpha_i^*$ to $y_{i-1}^{e_{i-1}-2} y_i^{e_i + 1}$, and then via $\alpha_i$ to $(e_i + 1) y_{i-1}^{e_{i-1}} y_i^{e_i}$. This holds even when we multiply by other variables, giving
\[
\alpha_i \alpha_i^* = \left\{ \begin{array}{ll}
0 & e_{i-1} \leq 1 \\
e_i + 1 & e_{i-1} \geq 2
\end{array} \right.
\]
Hence we have
\[
[\alpha_i, \alpha_i^*] y^e = (\alpha_i \alpha_i^* - \alpha_i^* \alpha_i) y^e = \left\{ \begin{array}{ll}
e_i y^e & e_{i-1} \leq 1 \\
y^e & e_{i-1} \geq 2
\end{array} \right.
= \left\{ \begin{array}{ll}
y^e & e_{i-1} \geq 2 \text{ or } e_i \text{ odd} \\
0 & e_{i-1} \leq 1 \text{ and } e_i \text{ even}
\end{array} \right.
\]
This last distinction between ``$e_{i-1} \geq 2$ or $e_i$ odd", and ``$e_{i-1} \leq 1$ and $e_i$ even" is clearly mutually exclusive and covers all possibilities. Though it may seem to be an obscure distinction, it will be crucial to our computation of homology.

Next we consider the commutativity of $\alpha_i^*$ and $\alpha_j$, where $i \neq j$. When either of $i$ or $j$ is zero, we have $\alpha_i^*$ or $\alpha_j = 0$, so we assume $i,j \geq 1$.

Note $\alpha_i^*$ only affects the variables $y_i$ and $y_{i-1}$, while $\alpha_j$ affects the variables $y_j$ and $y_{j-1}$. If $|i-j| \geq 2$ then these four variables are disjoint. So we have
\[
\alpha_i^* \alpha_j y^e = \alpha_j \alpha_i^* y^e = \left\{ \begin{array}{ll} 0 & e_{i-1} \leq 1 \\ 
e_j  y_{i-1}^{-2} y_i^1 y_{j-1}^2 y_j^{-1} y^e & e_{i-1} \geq 2
\end{array} \right.
\]
In particular,
\[
[\alpha_i^*, \alpha_j] = \alpha_i^* \alpha_j - \alpha_j \alpha_i^* = 0.
\]

This only leaves the case where $i,j$ differ by $1$, and we consider the two cases separately: (i) $\alpha_i^*$ and $\alpha_{i-1}$; and (ii) $\alpha_i^*$ and $\alpha_{i+1}$.
\begin{enumerate}
\item
Consider the effect of $\alpha_i^* \alpha_{i-1}$ on a monomial $y^e$. We first compute 
$\alpha_{i-1} y^e = e_{i-1} y_{i-2}^2 y_{i-1}^{-1} y^e $, which has exponent of $y_{i-1}$ equal to $e_{i-1}-1$. If $e_{i-1} \leq 2$ then applying $\alpha_i^*$ to this gives $0$; if $e_{i-1} \geq 3$ then it gives $e_{i-1} y_{i-2}^2 y_{i-1}^{-3} y_i y^e$. That is,
\[
\alpha_i^* \alpha_{i-1} y^e = \left\{ \begin{array}{ll}
0 & e_{i-1} \leq 2 \\
e_{i-1} y_{i-2}^2 y_{i-1}^{-3} y_i y^e  & e_{i-1} \geq 3
\end{array} \right.
\]
Now consider the effect of $\alpha_{i-1} \alpha_i^*$ on $y^e$. We compute $\alpha_i^* y^e$ is zero if $e_{i-1} \leq 1$, and if $e_{i-1} \geq 2$ then $\alpha_i^* y^e = y_{i-1}^{-2} y_i y^e $. Applying $\alpha_{i-1}$ to this then gives $(e_{i-1} - 2) y_{i-2}^2 y_{i-1}^{-3} y_i y^e $. Since $e_{i-1} - 2 = e_{i-1}$ mod $2$, we have
\[
\alpha_{i-1} \alpha_i^* y^e = \left\{ \begin{array}{ll}
0 & e_{i-1} \leq 1 \\
e_{i-1} y_{i-2}^2 y_{i-1}^{-3} y_i y^e  & e_{i-1} \geq 2
\end{array} \right.
\]
Putting these together gives
\[
[\alpha_i^*, \alpha_{i-1}] y^e = 
(\alpha_i^* \alpha_{i-1} - \alpha_{i-1} \alpha_i^*) y^e = \left\{ \begin{array}{ll}
0 & e_{i-1} \leq 1 \\
2  y_{i-2}^2 y_{i-1}^{-3} y_i y^e = 0 & e_{i-1} = 2 \\
0 & e_{i-1} \geq 3 \end{array} \right\}
= 0
\]
\item
Consider the effect of $\alpha_i^* \alpha_{i+1}$ on $y^e$. We compute $\alpha_{i+1} y^e = e_{i+1}  y_i^2 y_{i+1}^{-1} y^e$. Applying $\alpha_i^*$ to this, when $e_{i-1} \leq 1$ we obtain $0$; when $e_{i-1} \geq 2$ we obtain $e_{i+1} y_{i-1}^{-2} y_i^3 y_{i+1}^{-1} y^e$, so
\[
\alpha_i^* \alpha_{i+1} y^e = \left\{ \begin{array}{ll}
0 & e_{i-1} \leq 1 \\
e_{i+1} y_{i-1}^{-2} y_i^3 y_{i+1}^{-1} y^e & e_{i-1} \geq 2
\end{array} \right.
\]
No consider the effect of $\alpha_{i+1} \alpha_i^*$ on $y^e$. We have $\alpha_i^* y^e$ is zero when $e_{i-1} \leq 1$, and when $e_{i-1} \geq 2$ it is $y_{i-1}^{-2} y_i y^e$. Applying $\alpha_{i+1}$ to this gives $e_{i+1} y_{i-1}^{-2} y_i^3 y_{i+1}^{-1} y^e$. Thus
\[
\alpha_{i+1} \alpha_i^* y^e = \left\{ \begin{array}{ll}
0 & e_{i-1} \leq 1 \\
e_{i+1} y_{i-1}^{-2} y_i^3 y_{i+1}^{-1} y^e & e_{i-1} \geq 2
\end{array} \right.
\]
\end{enumerate}
The results for the operations are the same in any order, and so we conclude
\[
[\alpha_i^*, \alpha_{i+1}] = 0.
\]

We summarise the above discussion with the following statement.
\begin{prop}
\label{prop:decay_creation_commutation}
If $i \neq j$ then $\alpha_i^*$ and $\alpha_j$ commute, i.e. $[\alpha_i^*, \alpha_j] = 0$.

On the other hand, for any $i \geq 0$, the commutator of $\alpha_i^*$ and $\alpha_i$ is
\[
[\alpha_i, \alpha_i^*] = \left\{ \begin{array}{ll}
1 & e_{i-1} \geq 2 \text{ or } e_i \text{ odd} \\
0 & e_{i-1} \leq 1 \text{ and } e_i \text{ even}
\end{array} \right.
\]
\qed
\end{prop}

\subsection{A hierarchy of Weyl algebra representations}
\label{sec:Weyl_hierarchy}

We now extend the distinction in the calculation of $[\alpha_i, \alpha_i^*]$ to define a \emph{hierarchy} of \emph{levels} of monomials $y^e$ in $\Y$, depending on the values of the exponents $e_0, e_1, \ldots$. The idea is as follows:
\begin{itemize}
\item
\emph{Level $0$} consists of monomials such that $[\alpha_1, \alpha_1^*] = 1$. This means $e_0 \geq 2$ or $e_1$ is odd.
\item
\emph{Level $1$} consists of monomials such that $[\alpha_1, \alpha_1^*] = 0$ and $[\alpha_2, \alpha_2^*] = 1$. This means that $e_0 \leq 1$ and $e_1$ is even and ($e_1 \geq 2$ or $e_2$ is odd).
\item
Continuing in this fashion, \emph{level $i$} of the hierarchy, for positive integer $i$, consists of monomials such that $[\alpha_1, \alpha_1^*] = \cdots = [\alpha_i, \alpha_i^*] = 0$ but $[\alpha_{i+1}, \alpha_{i+1}^*] = 1$. This means that $e_0 \leq 1$, $e_1 = e_2 = \cdots = e_{i-1} = 0$, $e_i$ is even, and ($e_i \geq 2$ or $e_{i+1}$ is odd).
\item
\emph{Level $\infty$} of the hierarchy consists of the remaining monomials, namely those such that $e_0 \leq 1$ and $e_1 = e_2 = \cdots = 0$. These are just the monomials $1$ and $y_0$.
\end{itemize}

More formally, we define \emph{levels} $\Y^i$ of $\Y$ as follows. They are $\Z_2$-vector subspaces but not sub-algebras; they are not closed under multiplication.
\begin{defn}
The \emph{level $i$ subspace} $\Y^i$ of $\Y \cong \Z_2[y_0, y_1, \ldots]$, for an integer $i \geq 0$ or $i = \infty$, is defined as follows.
\begin{itemize}
\item
$\Y^0$ is the subspace generated by monomials $y^e$ such that $e_0 \geq 2$ or $e_1$ is odd.
\item
For an integer $i \geq 1$, $\Y^i$ is the subspace generated by monomials $y^e$ such that $e_0 \leq 1$, $e_j = 0$ for all integers $j$ with $1 \leq j \leq i-1$, $e_i$ is even, and ($e_i \geq 2$ or $e_{i+1}$ is odd).
\item
$\Y^\infty$ is the subspace generated by $\{1, y_0\}$.
\end{itemize}
\end{defn}

(Note that in the definition of $\Y^1$ the condition on $j$ is vacuous.)

\begin{lem}
Every monomial in $\Y$ lies in precisely one $\Y^i$, so
\[
\Y = \bigoplus_{i=0}^\infty \Y^i
\]
(where the direct sum includes $i = \infty$).
\end{lem}

\begin{proof}
If $y^e$ is not in $\Y^0$, then $e_0 \leq 1$ and $e_1$ is even. We will show that such an $y^e$ lies in precisely one $\Y^i$ with $i>0$. 

First, if all $e_1 = e_2 = \cdots = 0$, then $y^e = y_0^{e_0}$; and since $e_0 \leq 1$, we have $y^e = 1$ or $y_0$. In this case, $y^e \in \Y^\infty$ and  clearly $y^e$ lies in no other $\Y^i$.

We may then assume that there is some smallest positive integer $j$ such that $e_j > 0$, so $e_1 = \cdots = e_{j-1} = 0$ but $e_j \neq 0$. 

Then clearly $y^e \notin \Y^i$ for any $i$ such that $1 \leq i \leq j-2$, since such a $\Y^i$ requires either $e_i$ or $e_{i+1}$ to be nonzero, but $i+1 \leq j-1$ and all of $e_1 = \cdots = e_{j-1} = 0$. Conversely, $y^e \notin \Y^i$ for any $i \geq j+1$, since such a $\Y^i$ requires $e_j = 0$.

If $e_j$ is even, then $y^e \in \Y^j$, since $e_j \geq 2$. But $y^e \notin \Y^{j-1}$, since $e_{j-1} < 2$ and $e_j$ is not odd. So $y^e$ is in precisely one of the subspaces $\Y^i$, namely $\Y^j$.

If $e_j$ is odd, then $y^e \in \Y^{j-1}$, since $e_{j-1} = 0$ is even, and $e_j$ is odd. But $y^e \notin \Y^j$, since $e_j$ is not even. So $y^e$ lies in precisely one of the subspaces $\Y^i$, namely $\Y^{j-1}$.
\end{proof}

We note that each $\Y^i$ can be regarded as a \emph{Weyl algebra} representation. Recall that the \emph{Weyl algebra} (over $\Z_2$) on $n$ variables is the $\Z_2$-algebra freely generated by commuting variables $x_1, \ldots, x_n$ and partial derivatives $\frac{\partial}{\partial x_1}, \ldots, \frac{\partial}{\partial x_n}$, so that they obey commutation relations $[x_i, x_j] = [\frac{\partial}{\partial x_i}, \frac{\partial}{\partial x_j}] = 0$ and $[x_i, \frac{\partial}{\partial x_j}] = \delta_{ij}$, the Kronecker delta. The operators $\alpha_i, \alpha_i^*$ we have defined satisfy $[\alpha_i, \alpha_j] = [\alpha_i^*, \alpha_j^*] = 0$ and $[\alpha_i, \alpha_j^*] = 0$ when $i \neq j$. The algebra of these operators is thus very close to the Weyl algebra. We do not have $[\alpha_i, \alpha_i^*] = 1$ on $\Y$ in general, but on $\Y^i$ we have $[\alpha_i, \alpha_i^*] = 1$. (When $j>i$, on $\Y^i$ the commutator $[\alpha_j, \alpha_j^*]$ is sometimes $1$ and sometimes $0$.)

Thus each $\Y^i$ carries a representation of the $1$-variable Weyl algebra generated by $\alpha_i, \alpha_i^*$. As we will see, this relation will provide a chain homotopy demonstrating trivial homology.

This hierarchy of subspaces behave in a ``triangular" way with respect to the decay operators $\alpha_j$.
\begin{lem}
If $j \leq i$ then
\[
\alpha_j \Y^i = 0.
\]
This includes the case $i = \infty$; any integer is less than $\infty$.
\end{lem}

\begin{proof}
When $j = 0$ the statement is clear, since $\alpha_0 = 0$; we thus only need consider $\alpha_j$ with $j \geq 1$. The result holds when $i = 0$, since then $j = 0$. So we may assume $i \geq 1$.

Consider $\Y^i$ for $i \geq 1$. If $1 \leq j < i$, then the statement follows immediately from the fact that each monomial generating $\Y^i$ has $e_j = 0$. For $j=i$, the statement follows from the fact that each monomial generating $\Y^i$ has $e_i$ even.

Finally consider $i = \infty$, i.e. $\Y^\infty$. For any $j \geq 1$, the statement follows since each generating monomial has $e_j = 0$.
\end{proof}

In fact, in the case  $i = \infty$, there are only the two monomials $1, y_0$ to consider, and it is clear that every $\alpha_j$ annihilates these.

We have just seen that $\alpha_j \Y^i = 0$ when $j \leq i$. When $j > i$ the result of applying $\alpha_j$ to a monomial in $\Y^i$ will not usually be zero, but it \emph{does} lie in $\Y^i$. We will then be able to show that the subspaces $\Y^i$ are \emph{subcomplexes} of $\Y$.

\begin{lem}
Applying any decay operator to a monomial results in zero, or a monomial of the same level. In particular, for any $i \geq 0$ or $i=\infty$ and any integer $j \geq 0$,
\[
\alpha_j \Y^i \subseteq \Y^i.
\]
\end{lem}

\begin{proof}
Since we know that $\alpha_j \Y^i = 0$ when $j \leq i$, it suffices to check the cases $j > i$. When $i = \infty$ there are no $j > i$ to check, so we only need check integers $i \geq 0$. We consider the cases $i = 0$ and $i \geq 1$ separately.

First consider $\Y^0$, and a monomial generator $y^e$ satisfying $e_0 \geq 2$ or $e_1$ odd. Then $\alpha_1 y^e$ is zero, if $e_1$ is even; and if $e_1$ is odd then $\alpha_1 y^e = y_1^{-1} y_0^2 y^e$. Thus the result will either be zero or have $e_0 \geq 2$, and hence will be a level $0$ monomial. Next, $\alpha_2 y^e$, if nonzero, will be equal to $e_2 y_2^{-1} y_1^2 y^e$. If the exponent $e_0 \geq 2$, then after applying $\alpha_2$, the exponent of $y_0$ remains greater than or equal to $2$. If $e_1$ is odd, then after applying $\alpha_2$, the exponent of $y_1$ is increased by $2$ and remains odd. Either way, $\alpha_2 y^e$ is a level $0$ monomial. For $j \geq 3$, $\alpha_j y^e$ either produces zero or only changes the exponents of $y_j$ and $y_{j-1}$, neither of which is $e_0$ or $e_1$; hence the result remains a level $0$ monomial.

Now consider $\Y^i$, for $i \geq 1$, generated by $y^e$ satisfying $e_0 \leq 1, e_1 = \cdots = e_{i-1} = 0$, $e_i$ even, and ($e_i \geq 2$ or $e_{i+1}$ odd); we consider the effect of $\alpha_j$ for $j \geq i+1$. The value of $\alpha_{i+1} y^e$, if nonzero, is $e_{i+1} y_{i+1}^{-1} y_i^2 y^e$. The only exponents that change are those of $y_i$ and $y_{i+1}$; $e_i$ is increased by $2$, so remains even, and becomes $\geq 2$. Thus the result is a level $i$ monomial. The effect of $\alpha_{i+2}$ on $y^e$, if nonzero, is $e_{i+2} y_{i+1}^2 y_{i+2}^{-1} y^e$, so the only exponents that change are those of $y_{i+1}$ and $y_{i+2}$; $e_{i+1}$ is increased by $2$, so if it was odd it remains odd; hence the result is a level $i$ monomial. The effect of $\alpha_j$, for $j \geq i+3$, on $y^e$, if nonzero, only affects exponents of $y_k$ with $k \geq i+2$, hence not those in the defining condition for $\Y^i$, and so results in a level $i$ monomial.
\end{proof}

As the differential $\partial = \sum_{k \geq 0} \alpha_k$ is the sum of the decay operators, it follows that $\partial$ preserves each level in the hierarchy of $\Y$. That is, for each $i \in \{0, 1, \ldots, \infty\}$, $(\Y^i, \alpha)$ is a chain complex, a subcomplex of $\Y$. Slightly abusing notation, we will also write $\partial$ for the restriction of the differential to each $\Y^i$.

\begin{cor}
As chain complexes,
\[
(\Y, \partial) = (\Y^0, \partial ) \oplus (\Y^1, \partial) \oplus \cdots \oplus (\Y^\infty, \partial).
\qed
\]
\end{cor}

Referring to proposition \ref{prop:decay_creation_commutation}, we now note that the commutation relations between the fusion and decay operators $\alpha_i, \alpha_j^*$ can be used to provide \emph{chain homotopies} on these chain complexes.

On $\Y^0$, for each generating monomial $y^e$ we have $e_0 \geq 2$ or $e_1$ is odd, so $[\alpha_1, \alpha_1^*] y^e = y^e$ and $[\alpha_j, \alpha_1^*] y^e = 0$ for $j \geq 2$. That is, on $\Y^0$, we have $[\alpha_1, \alpha_1^*] = 1$ and $\alpha_j, \alpha_1^*] = 0$ for $j \geq 2$. It follows that, on $\Y^0$,
\[
[\partial, \alpha_1^*] = \sum_{j=0}^\infty [\alpha_j, \alpha_1^*] = [\alpha_1, \alpha_1^*] = 1.
\]
The resulting equation $\partial \alpha_1^* + \alpha_1^* \partial = 1$ says that $\alpha_1^*$ is a chain homotopy from $1$ to $0$, so that the homology of the complex is zero, $H_* ( \Y^0, \partial ) = 0$.

Similarly, for each monomial generating $\Y^1$ we have $e_0 \leq 1$, $e_1$ even and ($e_1 \geq 2$ or $e_2$ is odd). Thus $[\alpha_2, \alpha_2^*] = 1$ and
\[
[\partial, \alpha_2^*] = \sum_{j=0}^\infty [\alpha_j, \alpha_2^*] = [\alpha_2, \alpha_2^*] = 1.
\]
Thus $\partial \alpha_2^* + \alpha_2^* \partial = 1$, so that $\alpha_2^*$ is a chain homotopy from $1$ to $0$ and the homology $H_* (\Y^1, \partial ) = 0$.

Similarly, on each $\Y^i$, for $i \geq 1$, we consider monomials where $e_i \geq 2$ or $e_{i+1}$ is odd, so that $[\alpha_{i+1}, \alpha_{i+1}^*] = 1$, and hence
\[
[\partial, \alpha_{i+1}^*] = \sum_{j=0}^{n-1} [\alpha_j, \alpha_{i+1}^*] = [\alpha_{i+1}, \alpha_{i+1}^*] = 1,
\]
so that $\partial \alpha_{i+1}^* + \alpha_{i+1}^* \partial = 1$, giving a chain homotopy from $1$ to $0$ and demonstrating that $H_* (\Y^i, \alpha) = 0$.

On $\Y^\infty$, we have that every $\alpha_j = 0$, so $\partial = 0$ and hence $H_* (\Y^\infty, \partial) = \Y^\infty$. This $\Y^\infty$ is rather small, generated by $1$ and $y_0$.

We have now computed the homology of each summand of $(\Y, \partial)$, hence of $\Y$. As $\Y$ is a differential $\Z_2$-algebra (even though the $\Y^i$ are not), the homology $H(\Y)$ is a $\Z_2$-algebra. Denoting the homology class of $y \in \Y$ by $\bar{y}$, we have the following.
\begin{thm}
The homology of $(\Y, \partial)$, as a $\Z_2$-algebra, is
\[
H(\Y) = \frac{ \Z_2 [ \bar{y}_0 ] }{ ( \bar{y}_0^2 ) }.
\]
\end{thm}

\begin{proof}
The only summand $\Y^i$ of $\Y$ with nontrivial homology is $\Y^\infty$, which is 2-dimensional over $\Z_2$ with basis $\{1, y_0\}$ and trivial differential. So as a $\Z_2$-vector space, $H(\Y)$ has basis $\{ \bar{1}, \bar{y}_0 \}$. As an algebra, $H(\Y)$ must be generated by $\bar{y}_0$, inheriting multiplication from $\Y$. Since $y_0^2 = \partial y_1$, in homology we have $\bar{y}_0^2 = 0$, and as an algebra the homology is as claimed.
\end{proof}

\subsection{Putting the chain complexes back together}

We now return to the original chain complex $\widehat{CS}(A, \emptyset) \cong \X$, and reassemble it from the various $\X_j$. We have the tensor decomposition of differential algebras $\X \cong \bigotimes_j \X_j$, and each $\X_j \cong \Y$ under an isomorphism which takes the ``infinite decay chain" of $\Y$
\[
y_k \mapsto y_{k-1} \mapsto \cdots \mapsto y_1 \mapsto y_0 \mapsto 0
\]
to the corresponding ``decay chain" of $\X_j$
\[
x_{j \cdot 2^k} \mapsto x_{j \cdot 2^{k-1}} \mapsto \cdots \mapsto x_{2j} \mapsto x_j \mapsto 0.
\]

We will keep track of how some structure in the $\X_j$ translates back into the original $\X$. In particular, each operator $\alpha_k$ on each $\X_j$ translates to a corresponding operator $\alpha_{(j,k)}$ on $\X$. As $\alpha_k$ has the effect of sending $y_k \mapsto y_{k-1}^2$, the corresponding operator $\alpha_{(j,k)}$ has the effect of sending $y_{j \cdot 2^k} \mapsto y_{j \cdot 2^{k-1}}^2$, for each $j$ and $k$. Similarly we can define $\alpha_{(j,k)}^*$ to be the operator on $\X$ corresponding to the operator $\alpha_k^*$ on $\X_j$.
\begin{defn}
\label{def:alpha_j_k}
For each odd integer $j$ and each positive integer $k$, the following operators $\X \To \X$ are defined on monomials as follows, and extended over $\X$ by linearity.
\begin{enumerate}
\item
The operator $\alpha_{(j,k)}$ is defined by $\Z_2$-linearity, the Leibniz rule, and
\[
\alpha_{(j,k)} x_{j \cdot 2^k} = x_{j \cdot 2^{k-1}}^2,
\quad
\alpha_{(j,k)} x_i = 0 \text{ for $i \neq j \cdot 2^k$}.
\]
\item
The operator $\alpha_{(j,k)}^*$ is defined on a monomial $x^e = \prod_{i \in \Z \backslash \{0\}} x_i^{e_i}$ by
\[
\alpha_{(j,k)}^* x^e = \left\{ \begin{array}{ll}
0 & e_{j \cdot 2^{k-1}} \leq 1 \\
x_{j \cdot 2^{k-1}}^{-2} x_{j \cdot 2^k} & e_{j \cdot 2^{k-1}} \geq 2
\end{array} \right.
\]
\end{enumerate}
\end{defn}
We then have
\[
\partial = \sum_{j \text{ odd}} \sum_{k=1}^\infty \alpha_{(j,k)}.
\]

Since $\alpha_{(j,k)}$ and $\alpha_{(j,k)}^*$ mainly affect $x_i$ where $i = j \cdot 2^k$ or $j \cdot 2^{k-1}$, we can also write
\[
\alpha_{(j,k)} = \alpha_{j \cdot 2^k}, \quad
\alpha_{(j,k)}^* = \alpha_{j \cdot 2^k}^*.
\]
With this notation, we have an $\alpha_i$ defined on $\X$ for each even integer $i$. (In section \ref{sec:one_marked_point_each_boundary}, we will see a natural definition for odd $i$, when $F$ is nonempty.)

The commutation relations on the operators $\alpha_k, \alpha_k^*$ on each $\X_j$ also translate directly into commutation relations on the $\alpha_{(j,k)}$ and $\alpha_{(j,k)}^*$. From proposition \ref{prop:decay_creation_commutation} we have $[\alpha_k^*, \alpha_l] = 0$ when $k \neq l$, and $[\alpha_k, \alpha_k^*] = 1$ or $0$ accordingly as ($e_{k-1} \geq 2$ or $e_k$ is odd) or ($e_{k-1} \leq 1$ and $e_k$ is even). In $\X$, we see that $\alpha_{(j,k)}$ and $\alpha_{(j,k)}^*$ only affect those $x_i$ where $i = j \cdot 2^k$ or $j \cdot 2^{k-1}$; thus we obtain the following lemma.

\begin{lem}
\label{lem:alpha_j_k_commutations}
Let $j,j'$ be odd integers and $k,k'$ positive integers.
\begin{enumerate}
\item
For $j \neq j'$ or $k \neq k'$ we have
\[
[\alpha_{(j,k)}, \alpha_{(j',k')}^*] = 0.
\]
\item
For a monomial $x^e \in \X$,
\[
[\alpha_{(j,k)}, \alpha_{(j,k)}^*] = 
\left\{ \begin{array}{ll}
1 & e_{j \cdot 2^{k-1}} \geq 2 \text{ or } e_{j \cdot 2^k} \text{ odd} \\
0 & e_{j \cdot 2^{k-1}} \leq 1 \text{ and } e_{j \cdot 2^k} \text{ even}
\end{array} \right.
\]
\end{enumerate}
\qed
\end{lem}

\subsection{Computing the total homology}
\label{sec:computing_total_homology}

We have found that the homology $\Y$ is generated as a $\Z_2$-algebra by $\bar{y}_0$, the homology class of $y_0$, with the relation $\bar{y}_0^2 = 0$. Correspondingly, for each odd integer $j$, the homology of $\X_j$ is generated as a $\Z_2$-algebra by the homology class $\bar{x}_j$ of $x_j$, whose square is zero.
\[
H (\X_j) = \frac{ \Z_2 [ \bar{x}_j ] }{ ( \bar{x}_j^2) }
\]

Calculating the homology of $\widehat{CS}(\An, \emptyset) \cong \X$ now amounts to an application of the K\"{u}nneth theorem. As usual, write $\bar{x}_j$ for the homology class of $x_j$.

\begin{thm}
\label{thm:homology_of_X}
As a $\Z_2$-algebra, $H(\X)$ is generated by the homology classes $\bar{x}_j$ of $x_j$ over all odd $j$, where each $\bar{x}_j^2 = 0$. 
\[
\widehat{HS}(\An, \emptyset) = H(\X) = \frac{ \Z_2 [ \ldots, \bar{x}_{-3}, \bar{x}_{-1}, \bar{x}_1, \bar{x}_3, \ldots ] }{ ( \ldots, \bar{x}_{-3}^2, \bar{x}_{-1}^2, \bar{x}_1^2, \bar{x}_3^2, \ldots ) }
\]
\end{thm}

Note that the theorem just asserts that the $H (\X) = H(\bigotimes \X_j)$ is the tensor product of the $H ( \X_j)$:
\[
\bigotimes_{j \text{ odd}} H ( \X_j ) = \bigotimes_{j \text{ odd}} \frac{ \Z_2 [ \bar{x}_j ] }{ (\bar{x}_j^2) } = \frac{ \Z_2 [ \ldots, \bar{x}_{-3}, \bar{x}_{-1}, \bar{x}_1, \bar{x}_3, \ldots ] }{ ( \ldots, \bar{x}_{-3}^2, \bar{x}_{-1}^2, \bar{x}_1^2, \bar{x}_3^2, \ldots ) }.
\]

For the purposes of the proof, we will define a differential algebra $\X^{\leq N}$, for any odd positive integer $N$: it is the tensor product of the $\X_j$ with $|j| \leq N$,
\[
\X^{\leq N} = \X_{-N} \otimes \cdots \otimes \X_{-3} \otimes \X_{-1} \otimes \X_1 \otimes \X_3 \otimes \cdots \otimes \X_{N} = \bigotimes_{j=-N}^N \X_j.
\]
Note that $\X$ is the direct limit of the $\X^{\leq N}$:
\[
\X^{\leq 1} \subset \X^{\leq 3} \subset \cdots \subset \X
\quad \text{with} \quad
\bigcup_{N=1}^\infty \X^{\leq N} = \X.
\]

\begin{proof}
We repeatedly apply the K\"{u}nneth theorem (for instance as in theorem V.2.1 of \cite{HS}), which implies the following statement: if $A,B$ are chain complexes over $\Z_2$, then $H(A \otimes B) \cong H(A) \otimes H(B)$. (Any chain complex over $\Z_2$ is a $\Z_2$-module, hence $\Z_2$-vector space, hence free, hence projective, hence flat; and its homology is also a $\Z_2$-vector space, so $\Tor(H(A), H(B)) = 0$, and hence the map $H(A) \otimes H(B) \To H(A \otimes B)$ is an isomorphism.)

Applying the K\"{u}nneth theorem to $\X_{-N}, \ldots, \X_{-1}, \X_1, \ldots, \X_N$ gives immediately
\[
H \left( \X^{\leq N} \right) = H \left( \bigotimes_{j=-N}^N \X_j \right) \cong \bigotimes_{j=-N}^N H \left( \X_j \right)
\cong \frac{ \Z_2 [ \bar{x}_{-N}, \ldots, \bar{x}_{-1}, \bar{x}_1, \ldots, \bar{x}_N ] }{ ( \bar{x}_{-N}^2, \ldots, \bar{x}_{-1}^2, \bar{x}_1^2, \ldots, \bar{x}_N^2 ) }.
\]

Now homology commutes with direct limits, and so the homology of $\X$ is the direct limit of the homologies of the $\X^{\leq N}$, hence is as claimed.
\end{proof}

Thus, the string homology of $(\An, \emptyset)$ is generated by the homology classes of closed curves $x_j$ which run an odd number $j$ of times around the core of $\An$. As a $\Z_2$-module, this homology is generated by collections of closed curves with distinct odd homotopy classes. The multiplicative structure on $\widehat{HS}(A, \emptyset)$ can be regarded as arising from disjoint union of curves; or, equivalently, by gluing two string diagrams in distinct annuli together by gluing the annuli together along their boundaries. 

In any case, we have proved theorem \ref{thm:annulus_no_marked_points_calculation}.

It will be useful in the sequel to make various definitions based on various types of monomials and polynomials in the $x_i$ and $\bar{x}_i$ that arise in $\X$ and $H(\X)$. To this end we make the following definitions.
\begin{defn} \
\label{def:clean}
\begin{enumerate}
\item
A \emph{fermionic monomial} is an element of $\X$ of of the form $x_{j_1} \cdots x_{j_k}$ where $j_1, \ldots, j_k$ are all odd and pairwise distinct. 
\item
A \emph{(positively) clean monomial} is a fermionic monomial not containing $x_1$.
\item
A \emph{negatively clean monomial} is a fermionic monomial not containing $x_{-1}$.
\item
A \emph{totally clean monomial} is a fermionic monomial not containing $x_1$ or $x_{-1}$.
\item
A \emph{fermionic polynomial} (resp. positively clean, negatively clean, totally clean polynomial) is a finite sum of fermionic monomials (resp. positively clean, negatively clean, totally clean monomials).
\end{enumerate}
\end{defn}
Thus, for instance, a positively clean polynomial is a polynomial in $\{x_j \; : \; j \text{ odd}, \; j \neq 1\}$, linear in each $x_j$. When we refer to a clean monomial or polynomial, by default we mean a positively clean one. The polynomials are ``fermionic" in the sense that only odd variables (``particles") appear, and that each variable (``particle") can only appear once in each monomial, so obeys a ``Pauli exclusion principle".

Note that any fermionic polynomial $p \in \X$ satisfies $\partial q = 0$ (hence also any positively or negatively or totally clean polynomial). Thus we may speak of fermionic, positive and negatively and totally clean polynomials in $H(\X)$ as those represented by such polynomials in $\X$. Here the variables have the further ``fermionic" property that $\bar{x}_j^2 = 0$.

Our computations show that $H(\X)$ is a free $\Z_2$-module with basis the fermionic monomials, and every homology class has a unique fermionic polynomial representative in $\X$. The result in $H(\X)$ of multiplying two fermionic monomials is their usual product, if that product is another fermionic monomial; otherwise some variable appears twice, and since $\bar{x}_j^2=0$, the product in $H(\X)$ is zero.

The $\Z_2$-submodule of $H(\X)$ generated by (positively) clean monomials is in fact a subring; its elements are precisely the (positively) clean polynomials. Similarly there are subrings of negatively clean and totally clean polynomials.
\begin{defn}
\label{defn:polynomial_subrings}
The subring of $\left\{ \begin{tabular}{c} (positively) clean \\ negatively clean \\ totally clean \end{tabular} \right \}$ polynomials in $H(\X)$ is denoted $\left\{ \begin{tabular}{c} $H(\X)_{\neq 1}$ \\ $H(\X)_{\neq -1}$ \\ $H(\X)_{\neq -1,1}$ \end{tabular} \right\}$.
\end{defn}

Explicitly,
\[
H(\X)_{\neq 1} = \frac{ \Z[\ldots, \bar{x}_{-3}, \bar{x}_{-1}, \bar{x}_3, \bar{x}_5, \ldots] }{ (\ldots, \bar{x}_{-3}^2, \bar{x}_{-1}^2, \bar{x}_3^2, \bar{x}_5^2, \ldots) },
\quad
H(\X)_{\neq -1} = \frac{ \Z[\ldots, \bar{x}_{-5}, \bar{x}_{-3}, \bar{x}_{1}, \bar{x}_3, \ldots] }{ (\ldots, \bar{x}_{-5}^2, \bar{x}_{-3}^2, \bar{x}_{1}^2, \bar{x}_3^2, \ldots) },
\]
\[
H(\X)_{\neq -1,1} = \frac{ \Z[\ldots, \bar{x}_{-5}, \bar{x}_{-3}, \bar{x}_3, \bar{x}_5, \ldots] }{ (\ldots, \bar{x}_{-5}^2, \bar{x}_{-3}^2, \bar{x}_3^2, \bar{x}_5^2, \ldots) }.
\]

\section{Strings on annuli with two marked points}
\label{sec:annuli_two_marked_points}

We next turn to marked annuli $(\An, F)$ where $|F|=2$, so that $|F_{in}|=|F_{out}|=1$. A nonzero string diagram $s$ in $\widehat{CS}(\An,F)$ consists of a single open string from $F_{in}$ to $F_{out}$, and some number (possibly zero) of closed curves, each running some nonzero number of times around the annulus. We again write $x_n$ for the (homotopy class of the) closed curve which runs $n$ times around the annulus.

There are two cases: $F$ either consists of one marked point on each boundary component; or both marked points are on a single boundary component. We write $F=F_{1,1}$ and $F=F_{0,2}$ accordingly. We consider the first case in sections \ref{sec:one_marked_point_each_boundary} to \ref{sec:source_creation}; and the second case in sections \ref{sec:two_points_single_boundary} to \ref{sec:homology_A_+_tensor_X}.

\subsection{Annuli with one marked point on each boundary}
\label{sec:one_marked_point_each_boundary}

When $F=F_{1,1}$ consists of one point on each boundary component, the arc connecting these two points runs from one boundary component to the other. There are infinitely many such homotopy classes (relative to endpoints) of such arcs, each corresponding to running some number of times around the annulus. We denote (the homotopy classes of) these curves $c_n$, for $n \in \Z$, as shown in figure \ref{fig:one_open_string_diagrams}; they form a $\Z_2$-basis for $\widehat{CS}^O (\An,F_{1,1})$.

\begin{figure}[ht]
\begin{center}
\begin{tikzpicture}[
scale=1.2, 
string/.style={thick, draw=red, postaction={nomorepostaction, decorate, decoration={markings, mark=at position 0.5 with {\arrow{>}}}}}]

\draw (3,0) circle (1 cm); 	
\draw (3,0) circle (0.2 cm);
\draw (0,0) circle (1 cm);
\draw (0,0) circle (0.2 cm);
\draw (-3,0) circle (1 cm);
\draw (-3,0) circle (0.2 cm);

\draw [xshift=-3 cm, string] (0,0.2) -- (0,0.3) .. controls (0,0.5) and (0.6,0.5) .. (0.6,0) arc (0:-180:0.6) .. controls (-0.6,0.5) and (0,0.7) .. (0,0.9) -- (0,1);

\draw [string] (0,0.2) -- (0,1);

\draw [xshift=3cm, string] (0,0.2) -- (0,0.3) .. controls (0,0.5) and (-0.6,0.5) .. (-0.6,0) arc (-180:0:0.6) .. controls (0.6,0.5) and (0,0.7) .. (0,0.9) -- (0,1);

\draw (-3,-1.5) node {$c_{-1}$};
\draw (0,-1.5) node {$c_0$};
\draw (3,-1.5) node {$c_1$};

\end{tikzpicture}
\caption{Open strings on $(\An, F_{1,1})$.}
\label{fig:one_open_string_diagrams}
\end{center}
\end{figure}

It will be useful later to write $\C$ for the free $\Z_2$-module on $\{c_n \; : \; n \in \Z\}$. So $\C = \widehat{CS}^O(\An, F_{1,1})$. 

From lemma \ref{lem:CS_as_module} we have $\widehat{CS}(\An,F_{1,1}) \cong \X \otimes_{\Z_2} \widehat{CS}^O (\An,F_{1,1}) = \X \otimes_{\Z_2} \C$. Thus so $\widehat{CS}(\An,F_{1,1})$ is a free $\X$-module with basis the $c_n$. As a $\Z_2$-module, $\widehat{CS}(\An,F_{1,1})$ is free with basis given by elements
\[
c_n x_{k_1} x_{k_2} \cdots x_{k_m},
\]
where $n \in \Z$ and $k_1, \ldots, k_m \in \Z \backslash \{0\}$. We can also use the exponential notation
\[
c_n x^e = c_n \prod_{i \in \Z \backslash \{0\}} x_i^{e_i}
\]
where each $e_i \geq 0$ and only finitely many $e_i$ are nonzero.

We now describe the differential on $\widehat{CS}(\An,F_{1,1}) \cong \X \otimes \C$; as the Goldman bracket is nonzero, $\partial$ does not obey the Leibniz rule and we have (equation \ref{eqn:product_rule_open_closed})
\[
\partial (c_n x^e) = (\partial c_n) x^e + c_n (\partial x^e) + [c_n, x^e].
\]

Here $\partial x^e$ is as in $\X$. Just as in the case $F = \emptyset$, after a homotopy the closed strings can be made pairwise disjoint, with each $x_i$ having $|i|-1$ self-intersection points. The arc $c_n$ can be drawn without self-intersections, intersecting each closed curve $x_i$ precisely $|i|$ times. Thus $\partial c_n = 0$. When we resolve a crossing between a $c_n$ and $x_i$, we obtain the open string $c_{n+i}$; there are $|i|$ such crossings, so (mod 2) $[c_n, x_i] = i c_{n+i}$.

In general, resolving all the crossings between $c_n$ and a closed multicurve $x^e$, we obtain (using our standard abusive notation) the Goldman bracket as
\[
[c_n, x^e] = \sum_{i \in \Z \backslash \{0\}} e_i [c_n, x_i] x_i^{-1} x^e
= \sum_{i \in \Z \backslash \{0\}} i e_i c_{n+i} x_i^{-1} x^e,
\]
since each $x_i$ occurs $e_i$ times, and resolving a crossing between $c_n$ and an $x_i$ leaves the closed multicurve $x_i^{-1} x^e$ remaining.

Thus, on a string diagram $s = c_n x^e$, the differential is given as follows.
\begin{equation}
\label{eqn:partial_s_description}
\partial s = \partial \left( c_n x^e \right) = c_n \left( \partial x^e \right) + \sum_{i \in \Z \backslash \{0\}} i e_i c_{n + i} x_i^{-1} x^e
\end{equation}

Recall definition \ref{def:alpha_j_k} of the operators $\alpha_{(j,k)}$ on $\X$, for each odd $j$ and each positive integer $k$, which sends $x_{j \cdot 2^k} \mapsto x_{j \cdot 2^{k-1}}^2$; recall that sum of all the $\alpha_{(j,k)}$ is the differential on $\X$. These operators extend naturally over $\X \otimes \C$, sending each $c_n x^e \mapsto c_n \alpha_{(j,k)} x^e$. Applied to $s$, they give the first term of equation \ref{eqn:partial_s_description}.

We alternatively wrote $\alpha_{(j,k)} = \alpha_{j \cdot 2^k}$; with this notation, there is an $\alpha_i$ for every nonzero \emph{even} integer $i$. On the other hand, the terms in the second sum of equation \ref{eqn:partial_s_description} are only nonzero when $i$ is \emph{odd}, as there is a factor of $i$ (mod $2$), corresponding to the number of intersections between the arc $c_n$ and a closed curve in homotopy class $i$. Hence it is natural to extend the definition of the $\alpha_i$ to odd $i$ to produce the second term of equation \ref{eqn:partial_s_description}.

\begin{defn}
For each odd $j$, define $\alpha_j = \alpha_{(j,0)}$ on monomials (extended  by linearity) as
\[
\alpha_{(j,0)} \left( c_n x^e \right) = e_j c_{n+j} x_j^{-1} x^e.
\]
\end{defn}

Equation \ref{eqn:partial_s_description} then becomes
\[
\partial s = \sum_{j \text{ odd}} \sum_{k=1}^\infty \alpha_{(j,k)} s + \sum_{j \text{ odd}} \alpha_{(j,0)} s
= \sum_{j \text{ odd}} \sum_{k=0}^\infty \alpha_{(i,j)} s
= \sum_{i \in \Z \backslash \{0\}} \alpha_i s,
\]
so we have a compact description of the chain complex as
\[
\widehat{CS}(\An, F_{1,1}) = \X \otimes \C, \quad
\partial = \sum_{i \in \Z \backslash \{0\}} \alpha_i.
\]
We compute its homology we will compute in the next section.

\subsection{Source operators}
\label{sec:source_creation}

Returning to our particle analogy, the operation of $\alpha_{(j,k)}$ for positive $k$, which sends $x_{j \cdot 2^{k}} \mapsto x_{j \cdot 2^{k-1}}^2$, can be seen as a ``decay". The new operator $\alpha_{(j,0)}$ corresponds to a ``decay" $x_j \mapsto c_{\cdot + j}$. We can regard this as the ``end of a decay chain"
\[
x_{j \cdot 2^k} \mapsto x_{j \cdot 2^{k-1}} \mapsto \cdots \mapsto x_2j \mapsto x_j \mapsto c_{\cdot},
\]
where the $c_n$ terms are ``sinks" or ``ground states" into which the $x_i$ are finally absorbed.

The fusion operators $\alpha_{(j,k)}^* = \alpha_{j \cdot 2^k}^*$ on $\X$, for positive $k$, reverse this decay process and send $x_{j \cdot 2^{k-1}}^2 \mapsto x_{j \cdot 2^k}$. These operators extend to $\X \otimes \C$, sending $c_n x^e \mapsto c_n \alpha_{(j,k)}^* x^e$. Note an $\alpha_{j \cdot 2^k}^* = \alpha_i^*$ is defined for all nonzero even $i$.

We can now extend this definition to $\alpha_i^*$ for odd $i$, reversing the ``sinking" of an $x_i$ particle into a $c_n$. Instead, $\alpha_i^*$ will ``create" an $x_i$ out of a $c_n$ ``source" by sending $c_n \mapsto c_{n - i} x_{i}$. 
\begin{defn}
For each odd $j$, define $\alpha_j^* = \alpha_{(j,0)}^*$ on monomials (extended by linearity) as
\[
\alpha_j \left( c_n x^e \right) = c_{n-j} x_j x^e.
\]
\end{defn}

We now investigate the commutativity of the various $\alpha_{(j,k)}$ and $\alpha_{(j,k)}^*$. For $k>0$ these were given in lemma \ref{lem:alpha_j_k_commutations} and it remains to find commutators involving $k=0$; we give these now.
\begin{lem}
For any odd integer $i$ and any non-negative integers $j,k$,
\[
[ \alpha_{(i,0)}, \alpha_{(j,k)}^* ] = \delta_{i,j} \delta_{k,0}
\]
\[
[ \alpha_{(j,k)}, \alpha_{(i,0)}^* ] = \delta_{i,j} \delta_{k,0}
\]
\end{lem}
(Each $\delta$ here is a Kronecker delta.)

\begin{proof}
We first show that $\alpha_{(i,0)}, \alpha_{(j,k)}^*$ commute when $i \neq j$.
\begin{align*}
\alpha_{(i,0)} \alpha_{(j,k)}^* c_n x^e &= \alpha_{(i,0)} \left\{
\begin{array}{ll}
0 & e_{j \cdot 2^{k-1}} \leq 1 \\
c_n x_{j \cdot 2^{k-1}}^{-2} x_{j \cdot 2^k} x^e & e_{j \cdot 2^{k-1}} \geq 2
\end{array} \right. \\
&= \left\{ \begin{array}{ll}
0 & e_{j \cdot 2^{k-1}} \leq 1 \\
e_{i} c_{n+i} x_{i}^{-1} x_{j \cdot 2^{k-1}}^{-2} x_{j \cdot 2^k} x^e & e_{j \cdot 2^{k-1}} \geq 2.
\end{array} \right.
\end{align*}
\begin{align*}
\alpha_{(j,k)}^* \alpha_{(i,0)} c_n x^e &= \alpha_{(j,k)}^* e_{i} c_{n+i} x_{i}^{-1} x^e \\
&= \left\{ \begin{array}{ll}
0 & e_{j \cdot 2^{k-1}} \leq 1 \\
e_{i} c_{n+i} x_{j \cdot 2^{k-1}}^{-2} x_{j \cdot 2^k} x_{i}^{-1} x^e & e_{j \cdot 2^{k-1}} \geq 2
\end{array} \right.
\end{align*}
It is similar to check that $\alpha_{(i,0)}^*$ and $\alpha_{(j,k)}$ commute when $i \neq j$.

We next consider the commutativity of $\alpha_{(i,0)}$ and $\alpha_{(i,k)}^*$where $k \neq 0$.
\begin{align*}
\alpha_{(i,0)} \alpha_{(i,k)}^* c_n x^e &= \alpha_{(i,0)}
\left\{ \begin{array}{ll}
0 & e_{i \cdot 2^{k-1}} \leq 1 \\
c_n x_{i \cdot 2^{k-1}}^{-2} x_{i \cdot 2^k} x^e & e_{i \cdot 2^{k-1}} \geq 2
\end{array} \right. \\
&= \left\{ \begin{array}{ll}
0 & e_{i \cdot 2^{k-1}} \leq 1 \\
e_{i} c_{n+i} x_{i}^{-1} x_{i \cdot 2^{k-1}}^{-2} x_{i \cdot 2^k} x^e & e_{i \cdot 2^{k-1}} \geq 2
\end{array} \right.
\end{align*}
\begin{align*}
\alpha_{(i,k)}^* \alpha_{(i,0)} c_n x^e &= \alpha_{(i,k)}^* e_{i} c_{n+i} x_{i}^{-1} x^e \\
&= \left\{ \begin{array}{ll}
0 & e_{i \cdot 2^{k-1}} \leq 1 \\
e_{i} c_{n+i} x_{i \cdot 2^{k-1}}^{-2} x_{i \cdot 2^k} x_{i}^{-1} x^e & e_{i \cdot 2^{k-1}} \geq 2
\end{array} \right. %\quad \text{(Note this applies even if $j=1$)} 
\end{align*}
Hence $\alpha_{(i,0)}$ and $\alpha_{(i,k)}^*$ commute for $k \neq 0$. It is similar to check that $\alpha_{(i,0)}^*$ and $\alpha_{(i,k)}$ commute for $k \neq 0$.

Finally we consider the commutativity of $\alpha_{(i,0)}$ and $\alpha_{(i,0)}^*$.
\begin{align*}
\alpha_{(i,0)} \alpha_{(i,0)}^* c_n x^e &= \alpha_{(i,0)} \left( c_{n-i} x_{i} x^e \right) \\
&= \left( e_{i} + 1 \right) c_n x^e
\end{align*}
\begin{align*}
\alpha_{(i,0)}^* \alpha_{(i,0)} c_n x^e &= \alpha_{(i,0)}^* \left( e_{i} c_{n+i} x_{i}^{-1} x^e \right) \\
&= e_{i} c_n x^e
\end{align*}
\end{proof}

As the differential is the sum of the $\alpha_{(j,k)}$, the commutator of $\alpha_{(i,0)}^*$ with $\partial$ is now easily given.
\begin{prop}
For any odd integer $i$,
\[
[\partial, \alpha_{(i,0)}^*] = 1.
\]
\end{prop}

\begin{proof}
%We have $\partial = \sum_{(i,j)} \alpha_{(i,j)}$, over all pairs $(i,j)$ where $i$ is odd and $j \geq 0$. Hence
\[
\left[ \partial, \alpha_{(i,0)}^* \right]
= \left[ \sum_{(j,k)} \alpha_{(j,k)}, \alpha_{(i,0)}^* \right] 
= \sum_{(j,k)} \left[ \alpha_{(j,k)}, \alpha_{(i,0)}^* \right] 
= \sum_{(j,k)} \delta_{j,i} \delta_{k,0} \\
= 1.
\]
\end{proof}

Thus any $\alpha_{(i,0)}^*$ satisfies
\[
\partial \alpha_{(i,0)}^* + \alpha_{(i,0)}^* \partial = 1,
\]
and hence is a chain homotopy from $1$ to $0$. We immediately obtain the homology of the complex.
\begin{thm}
\label{thm:one_one_homology}
\[
\widehat{HS}(\An,F_{1,1}) = 0.
\]
\qed
\end{thm}

This is as it should be, since such $(\An,F_{1,1})$ does not correspond to any sutures.

\subsection{Annuli with two marked points on single boundary}
\label{sec:two_points_single_boundary}

We now turn to $(\An,F)$ where both marked points lie on the same boundary component, $F=F_{0,2}$. An open string runs between the two marked points. With the marked points drawn as in figure \ref{fig:a_curves}, the arc runs an integer and a half $n + \frac{1}{2}$ times around the annulus and we denote its homotopy class by $a_{n+\frac{1}{2}}$, for $n \in \Z$, as shown.

\begin{figure}[ht]
\begin{center}
\begin{tikzpicture}[
scale=1.2, 
string/.style={thick, draw=red, postaction={nomorepostaction, decorate, decoration={markings, mark=at position 0.5 with {\arrow{>}}}}}]

\draw (-4.5,0) circle (1 cm); 	
\draw (-4.5,0) circle (0.2 cm);
\draw (-1.5,0) circle (1 cm);
\draw (-1.5,0) circle (0.2 cm);
\draw (1.5,0) circle (1 cm);
\draw (1.5,0) circle (0.2 cm);
\draw (4.5,0) circle (1 cm);
\draw (4.5,0) circle (0.2 cm);

\draw [xshift = -4.5 cm, string] (0,-1) .. controls (0,-0.8) and (0,-0.8) .. (-95:0.8) arc (265:90:0.8) arc (90:-90:0.7) arc (270:90:0.6) .. controls (0,0.6) and (0,0.8) .. (0,1);

\draw [xshift = -1.5 cm, string] (0,-1) .. controls (0,-0.8) and (0,-0.8) .. (-95:0.8) arc (265:95:0.8) .. controls (0,0.8) and (0,0.8) .. (0,1);

\draw [xshift = 1.5 cm, string] (0,-1) .. controls (0,-0.8) and (0,-0.8) .. (-85:0.8) arc (-85:85:0.8) .. controls (0,0.8) and (0,0.8) .. (0,1);

\draw [xshift= 4.5 cm, string] (0,-1) .. controls (0,-0.8) and (0,-0.8) .. (-85:0.8) arc (-85:90:0.8) arc (90:270:0.7) arc (-90:85:0.6) .. controls (0,0.6) and (0,0.8) .. (0,1);

\draw (-4.5,-1.5) node {$a_{-\frac{3}{2}}$};
\draw (-1.5,-1.5) node {$a_{-\frac{1}{2}}$};
\draw (1.5,-1.5) node {$a_{\frac{1}{2}}$};
\draw (4.5,-1.5) node {$a_{\frac{3}{2}}$};

\end{tikzpicture}
\caption{Open strings on $(\An, F_{0,2})$.}
\label{fig:a_curves}
\end{center}
\end{figure}

It will be useful later to make definitions as follows.
\begin{defn}
Let $\A = \widehat{CS}^O(\An,F_{0,2})$ be the free $\Z_2$-module with basis $\{a_{n+\frac{1}{2}} \; : \; n \in \Z  \}$.

Let $\A_+$ be the free $\Z_2$-module on $\{a_{n+\frac{1}{2}} \; : \; n \in \Z_{\geq 0} \}$, and $\A_-$ the free $\Z_2$-module on $\{a_{n-\frac{1}{2}} \; : \; n \in \Z_{\leq 0} \}$.
\end{defn}

Lemma \ref{lem:CS_as_module} gives $\widehat{CS}(\An,F_{0,2}) \cong \X \otimes_{\Z_2} \A$, a free $\X$-module with basis the $a_n$. A generator $s$ of $\widehat{CS}(\An,F_{0,2})$, i.e. string diagram (up to homotopy) without contractible closed curves, consists of a single $a_n$, together with some (possibly none) closed curves $x_i$, so can be written as $a_n x^e$.

We have $\A = \A_+ \oplus \A_-$, and moreover
\begin{equation}
\label{eqn:F02_decomposition}
\widehat{CS}(\An,F_{0,2}) \cong \left( \X \otimes_{\Z_2} \A_+ \right) \oplus \left( \X \otimes_{\Z_2} \A_- \right).
\end{equation}

As in previous cases, a string diagram on $(\An,F_{0,2})$ can be drawn so that all the closed curves $x_i$ are disjoint from each other. We can also draw any $a_n$ disjoint from all closed strings. This being so, the Goldman bracket vanishes, the differential obeys the Leibniz rule, $\widehat{CS}(\An, F_{0,2}) \cong \X \otimes \A$ is a differential $\X$-module, and $\widehat{HS}(\An,F_{0,2})$ is a $H(\X)$-module, as discussed in section \ref{sec:differential_algebra}. As such, $\partial$ is determined by its action on $\X$, $\A$ and the Leibniz rule. 

We thus consider $\partial a_n$. We can draw $a_n$ as in figure \ref{fig:a_curves} so that it has $|n|- \frac{1}{2}$ self-intersections, and resolving each such intersection splits $a_n$ into an arc $a_i$, where $i$ has the same sign as $n$ and $|i| < |n|$, and a closed curve $x_j$ where $i+j=n$.

Thus 
\[
\ldots, \partial a_{-\frac{3}{2}} = a_{-\frac{1}{2}} x_{-1}, \quad
a_{-\frac{1}{2}} = a_\frac{1}{2} = 0, \quad
\partial a_\frac{3}{2} = a_\frac{1}{2} x_1, \quad
\partial a_\frac{5}{2} = a_\frac{3}{2} x_1 + a_\frac{1}{2} x_2, \quad \ldots
\]
and in general we have, for a positive integer $n$,
\[
\partial a_{n-\frac{1}{2}} = \sum_{i=1}^{n-1} a_{i-\frac{1}{2}} x_{n-i}, \quad
\partial a_{-n+\frac{1}{2}} = \sum_{i=1}^{n-1} a_{-i+\frac{1}{2}} x_{-n+i},
\]
which can also be written, for any $n$, as
\[
\partial a_n = \sum_{i+j=n, \; ij>0} a_i x_j.
\]

For a general string diagram $s = a_n x^e $ we have, by the Leibniz rule,
\[
\partial s = \left( \partial a_n \right) x^e + a_n \left( \partial x^e \right)
= \sum_{i+j = n, \; ij>0}  a_i x_j  x^e + a_n \left( \partial x^e \right),
\]
with $\partial x^e$ given by the differential on $\X$. 

A general element of $\widehat{CS}(A,F) \cong \X \otimes \A$ can be written as a sum
\[
f = \sum_{n \in \Z + \frac{1}{2} } a_n p_n,
\]
with finitely many nonzero terms, where $p_n \in \X$ is a polynomial in $\{x_n \; : \; n \in \Z \backslash \{0\} \}$.

We may split $f$ into terms involving $a_n$ with positive and negative $n$ respectively, i.e. according to the decomposition (\ref{eqn:F02_decomposition}). Then $f=f_+ + f_-$ where
\[
f_+ = \sum_{n = 1}^\infty a_{n-\frac{1}{2}} p_{n-\frac{1}{2}}, \quad f_- = \sum_{n=1}^\infty a_{-n + \frac{1}{2}} p_{-n+\frac{1}{2}}
\]

Note that each $\partial a_n$ only includes terms with $a_m$, where $m$ has the same sign as $n$. Thus the decomposition (\ref{eqn:F02_decomposition}) is in fact a direct sum of chain complexes and differential $\X$-modules, and we have
\[
\widehat{HS}(\An, F_{0,2}) = H( \X \otimes \A_+ ) \oplus H(\X \otimes \A_-).
\]
Moreover the $\X$-linear maps defined by $a_n \leftrightarrow a_{-n}$ give mutually inverse chain maps $\X \otimes \A_+ \leftrightarrow \X \otimes \A_- $, so the two summands are isomorphic. In the next section we compute $H(\X \otimes \A_+ )$ and hence $\widehat{HS}(\An, F_{0,2})$.

\subsection{Computing homology by simplifying cycles}
\label{sec:homology_A_+_tensor_X}

The main idea of our computation of $H(\X \otimes \A_+)$ is to successively simplify cycles of $\X \otimes \A_+$, showing that a cycle is homologous to an element of smaller ``degree" in the $a_{n+\frac{1}{2}}$. This technique will also be employed subsequently, in a more involved way, to compute homology in $(\An, F_{2,2})$.

Throughout this section we use the notions of fermionic and (positively) clean polynomials from definition \ref{def:clean}.

\begin{defn}
A general element $f$ of $\A_+ \otimes \X$ has the form
\[
f = a_\frac{1}{2} p_\frac{1}{2} + a_\frac{3}{2} p_\frac{3}{2} + \cdots + a_{n+\frac{1}{2}} p_{n+\frac{1}{2}},
\]
where each $p_{i+\frac{1}{2}} \in \X$, for some positive integer $n$ such that $p_{n+\frac{1}{2}} \neq 0$. This $n$ is the \emph{degree} of $f$.

We write $O(a_m)$ to denote an element of $\A_+ \otimes \X$ of degree $\leq m$. 
\end{defn}
Note that as $\partial a_n$ only involves terms with $a_j$ with $\frac{1}{2} \leq j \leq n-1$, $\partial$ lowers degree by at least $1$.

\begin{lem}
\label{lem:a_simplification}
If $f \in \A_+ \otimes \X$ satisfies $f = O(a_n)$, $\partial f = 0$ and $n \geq \frac{3}{2}$, then $f = \partial g + O(a_{n-1})$ for some $g \in \A_+ \otimes \X$.
\end{lem}

\begin{proof}
Let the $a_n$ and $a_{n-1}$ terms of $f$ be $a_n p_n$ and $a_{n-1} p_{n-1}$ respectively, so
\[
f = a_n p_n + a_{n-1} p_{n-1} + O(a_{n-2}).
\]
(If $n=3/2$ then $O(a_{n-2})=0$.)  We examine the highest order terms of $\partial f$, namely the $a_n$ and $a_{n-1}$ terms.
\begin{align*}
\partial f &= (\partial a_n) p_n + a_n (\partial p_n ) + ( \partial a_{n-1} ) p_{n-1} + a_{n-1} ( \partial p_{n-1} ) + O(a_{n-2}) \\
&= \left( a_{n-1} x_1 + O(a_{n-2}) \right) p_n + a_n \partial p_n + O(a_{n-2}) p_{n-1} + a_{n-1} \partial p_{n-1} + O(a_{n-2}) \\
&= a_n \partial p_n + a_{n-1} \left( x_1 p_n + \partial p_{n-1} \right) + O(a_{n-2})
\end{align*}
Now as $\partial f = 0$, then the polynomials which are coefficients of each $a_i$ must be zero. Considering the coefficients of $a_n$ and $a_{n-1}$ then gives
\[
\partial p_n = 0, \quad x_1 p_n = \partial p_{n-1}.
\]
Thus $p_n$ is a cycle in $\X$ but $x_1 p_n$ is a boundary. From our computation of $H(\X)$, we know $p_n = r + \partial u$, where $r,u \in \X$ and $r$ is a fermionic polynomial. We can further decompose $r$ as $s + x_1 t$, where $x_1$ does not occur in $s$ or $t$, i.e. $s,t$ are clean polynomials. (Note $\partial r = \partial s = \partial t = 0$.) We then have $p_n = s + x_1 t + \partial u$ so
\[
\partial p_{n-1} = x_1 p_n = x_1 s + x_1^2 t + x_1 \partial u 
= x_1 s + \partial (x_2 t + x_1 u)
\]
Thus $x_1 s$ is a boundary in $\X$. But as $s$ is clean, $x_1 s$ is fermionic; so if $s \neq 0$ then $x_1 s$ is nonzero in $H(\X)$. Thus $s=0$ and we have
\[
p_n = x_1 t + \partial u
\quad \text{so}
\quad
f = a_n \left( x_1 t + \partial u \right) + a_{n-1} p_{n-1} + O(a_{n-2}).
\]
Now we note that there is an element whose differential has the same $a_n$ term:
\[
\partial ( a_{n+1} t + a_n u ) = a_{n} \left( x_1 t + \partial u \right) + a_{n-1} \left( x_2 t + x_1 u \right) + O(a_{n-2}).
\]
(Here we used $\partial t =0$.) It follows that
\begin{align*}
f &= a_n \left( x_1 t + \partial u \right) + a_{n-1} p_{n-1} + O(a_{n-2}) \\
&= \partial (a_{n+1} t + a_n u ) + a_{n-1} (p_{n-1} + x_2 t + x_1 u) + O(a_{n-2})
\end{align*}
so $f = \partial g + O(a_{n-1})$ where $g = a_{n+1} t + a_n u$, as desired.
\end{proof}

Successively applying this lemma, we may reduce any cycle $f \in \A_+ \otimes \X$ to a homologous cycle of degree $1/2$, so $f = \partial g + a_\frac{1}{2} p$ where $g \in \A_+ \otimes \X$ and $p \in \X$. And we may use our knowledge of $H(\X)$ to say a little more.

\begin{prop}
\label{prop:homology_with_a1}
\label{prop:strong_homology_with_a1}
Suppose $f \in \A_+ \otimes \X$ satisfies $\partial f = 0$. Then $f = \partial g + a_{\frac{1}{2}} p$ where $g \in \A_+ \otimes \X$ and $p \in \X$ is a clean polynomial.
\end{prop}

\begin{proof}
From above we have $f = \partial g_0 + a_{\frac{1}{2}} p_0$ with $g_0 \in \A_+ \otimes \X$ and $p_0 \in \X$. Differentiating gives $0 = \partial f = a_{\frac{1}{2}} \partial p_0$. Thus $\partial p_0 = 0$, so $p_0$ represents a homology class in $\X$, and hence, up to a boundary, is a fermionic polynomial. Thus $p_0 = q + \partial r$ where $q,r \in \X$ and $q$ is fermionic. We can further split $q$ into terms which contain $x_1$ and those which do not: $q = p + x_1 u$, where $p,u$ are clean polynomials. (Note $\partial q = \partial p = \partial u = 0$.) We then have
\[
f + \partial g_0 = a_{\frac{1}{2}} p_0 = a_\frac{1}{2} (q + \partial r) = a_{\frac{1}{2}} (p + x_1 u + \partial r) = a_{\frac{1}{2}} p + \partial (a_{\frac{3}{2}} u + a_{\frac{1}{2}} r).
\]
(Here we have used the fact that $\partial (a_{\frac{3}{2}} u) = (\partial a_{\frac{3}{2}}) u = a_{\frac{1}{2}} x_1 u$.) Thus $f$ has the desired form, with $g = g_0 + a_{\frac{3}{2}} u + a_{\frac{1}{2}} r$.
\end{proof}

Roughly then, the homology of $A_+ \otimes \X$ behaves something like $\{a_\frac{1}{2} \} \otimes \X$, and indeed rather like $\X$. We will in fact give a chain map $\A_+ \otimes \X \To \X$.

\begin{defn}
The map $\Phi: \A_+ \otimes \X \To \X$ sends $a_{n-\frac{1}{2}} p \mapsto x_n p$, for each positive integer $n$ and $p \in \X$, and extends linearly over $\A_+ \otimes \X$.
\end{defn}

The map $\Phi$ is in fact an $\X$-module homomorphism; we now show it is a chain map, hence a homomorphism of differential $\X$-modules and descends to homology as an $H(\X)$-module homomorphism.
\begin{lem}
The map $\Phi$ commutes with $\partial$.
\end{lem}
Geometrically, $\Phi$ ``closes off" the endpoints of $a_{n-\frac{1}{2}}$ by gluing an annulus to the boundary of $A$, joining the two marked points by an arc which turns $a_{n-\frac{1}{2}}$ into $x_n$. This lemma is an instance of a more general result about gluing string diagrams together.

Essentially this lemma holds because the rule $\partial a_{n-\frac{1}{2}} = \sum_{i=1}^{n-1} a_{i-\frac{1}{2}} x_{n-i}$ becomes, after applying $\Phi$, the rule  $\partial x_n = \sum_{i=1}^{n-1} x_i x_{n-i}$. 
\begin{proof}
Using linearity and the Leibniz rule, it is sufficient to check that $\Phi \partial a_{n-\frac{1}{2}} = \partial \Phi a_{n-\frac{1}{2}}$, for $n \in \Z_{>0}$:
\[
\Phi \partial a_{n-\frac{1}{2}} = \Phi \sum_{i=1}^{n-1} a_{i-\frac{1}{2}} x_{n-i} = \sum_{i=1}^{n-1} x_i x_{n-i} = \partial x_n = \partial \Phi a_{n-\frac{1}{2}}.
\]
\end{proof}

As in section \ref{sec:computing_total_homology}, we write $\bar{f}$ for the homology class of $f \in \A_+ \otimes \X$. We have seen in proposition \ref{prop:homology_with_a1} that a cycle $f \in \A_+ \otimes \X$ is homologous to some $a_\frac{1}{2} p$ where $p$ is a clean polynomial. Thus $\Phi \bar{f} = \Phi \overline{a_\frac{1}{2} p} = \bar{x}_1 \bar{p}$. It follows that on homology, $\Phi$ has image in $\bar{x}_1 H(\X)$.
\begin{prop}
\label{prop:AX_injective}
The map $\Phi: H(\A_+ \otimes \X) \To H(\X)$ is an isomorphism onto $\bar{x}_1 H(\X)$.
\end{prop}

\begin{proof}
We have shown $\Im \Phi \subseteq \bar{x}_1 H(\X)$. For the reverse inclusion, take $\bar{x}_1 \bar{p} \in \bar{x}_1 H(\X)$, where $p$ is a fermionic polynomial. Splitting $p$ into terms with and without $x_1$, we may write $p = x_1 q + r$, where $q,r \in \X$ are both clean polynomials. Then $\bar{x}_1 \bar{p}= \bar{x}_1^2 \bar{q} + \bar{x}_1 \bar{r} = \bar{x}_1 \bar{r}$ as $\bar{x}_1^2 = 0$ in homology. 

Thus, any element of $\bar{x}_1 H(\X)$ is of the form $\bar{x}_1 \bar{r}$ where $r$ is a clean  polynomial. Such an element is certainly in the image of $\Phi$, as $\Phi( a_\frac{1}{2} r ) = x_1 r$. Thus $\Im \Phi = \bar{x}_1 H(\X)$.

To see $\Phi$ is injective, take a homology class $\bar{f} \in \ker \Phi$. By proposition \ref{prop:homology_with_a1} we have $\bar{f} = \bar{a}_{\frac{1}{2}} \bar{p}$ where $p \in \X$ is a clean polynomial. Then $0 = \Phi \bar{f} = \bar{x}_1 \bar{p}$. But as $p$ is clean, $x_1 p$ is fermionic, hence nonzero in homology unless $p=0$. Thus $\bar{f} = 0$ and $\Phi$ is injective.
\end{proof}

This result allows us to strengthen proposition \ref{prop:strong_homology_with_a1}. There we showed any cycle in $\A_+ \otimes \X$ is homologous to an element $a_{\frac{1}{2}} p$, where $p \in \X$ is a clean polynomial. We can now show this $p$ is unique: if $a_\frac{1}{2} p$ and $a_\frac{1}{2} p'$ are homologous, then $x_1 p$ and $x_1 p'$ are homologous in $\X$, so $\bar{x}_1 (\bar{p} - \bar{p}') = 0$ in $H(\X)$. It follows that $\bar{p} = \bar{p}'$ and hence $p=p'$.

\begin{prop}
\label{lem:stronger_homology_with_a1}
Suppose $f \in \A_+ \otimes \X$ satisfies $\partial f = 0$. Then $f = a_{\frac{1}{2}} p + \partial g$, where $g \in \A_+ \otimes \X$ and $p \in \X$ is a \emph{unique} clean polynomial.
\qed
\end{prop}

By proposition \ref{prop:AX_injective} we now have $H(\A_+ \otimes \X) \cong \bar{x}_1 H(\X)$; as $\Phi$ is an $H(\X)$-module homomorphism, this is an isomorphism of $H(\X)$-modules. Replacing $\A_+, a_n, x_n$ with $\A_-, a_{-n}, x_{-n}$ we similarly obtain an $H(\X)$-module isomorphism $H(\A_- \otimes \X) \cong \bar{x}_{-1} H(\X)$.
\begin{thm}
\label{thm:A+-X_homology}
There are isomorphisms of $H(\X)$-modules
\begin{align*}
H(\A_+ \otimes \X) &\cong \bar{x}_1 H(\X) = \bar{x}_1 \frac{\Z[\ldots, \bar{x}_{-3}, \bar{x}_{-1}, \bar{x}_1, \bar{x}_3, \ldots]}{( \ldots, \bar{x}_{-3}^2, \bar{x}_{-1}^2, \bar{x}_1^2, \bar{x}_3^2, \ldots ),} \\
H(\A_- \otimes \X) &\cong \bar{x}_{-1} H(\X) = \bar{x}_{-1} \frac{\Z[\ldots, \bar{x}_{-3}, \bar{x}_{-1}, \bar{x}_1, \bar{x}_3, \ldots]}{( \ldots, \bar{x}_{-3}^2, \bar{x}_{-1}^2, \bar{x}_1^2, \bar{x}_3^2, \ldots )}.
\end{align*}
An explicit isomorphism is given by $\bar{a}_{\pm \frac{1}{2}} \bar{p} \mapsto \bar{x}_{\pm 1} \bar{p}$.
\qed
\end{thm}

Since $\widehat{CS}(\Sigma,F)$ is the direct sum of the $\A_\pm \otimes \X$, we have now proved theorem \ref{thm:annulus_two_marked_points_same_boundary_calculation} and have an $H(\X)$-module isomorphism
\[
\widehat{HS}(\Sigma,F) \cong \bar{x}_1 H(\X) \oplus \bar{x}_{-1} H(\X) = (\bar{x}_1, \bar{x}_{-1})  \frac{\Z[\ldots, \bar{x}_{-3}, \bar{x}_{-1}, \bar{x}_1, \bar{x}_3, \ldots]}{( \ldots, \bar{x}_{-3}^2, \bar{x}_{-1}^2, \bar{x}_1^2, \bar{x}_3^2, \ldots )}.
\]

\section{Non-alternating annuli}
\label{sec:non-alternating_annuli}

In this section we prove theorem \ref{thm:non-sutured_zero_iso} that for any non-alternating weakly marked annulus $(\An,F)$, its string homology is zero. In \cite{Mathews_Schoenfeld12_string} we proved such a result for discs; and as we will see, the methods used there apply immediately here, for all cases except one, with which we have already dealt.

In section 5 of \cite{Mathews_Schoenfeld12_string} we introduced a \emph{switching operation} $W$ on a string diagram $s$ on a disc; we now define it more generally. Let $(\Sigma,F)$ be a weakly marked surface. Suppose that there are two distinct points $p, q$ of $F$, on the same boundary component of $\Sigma$, of the same sign; suppose $p,q \in F_{in}$ (resp. $F_{out}$). The switching operation $W$ alters $s$ near $p$ and $q$, so that the strand which began (resp. ended) at $p$ now begins (resp. ends) at $q$; and the strand which began (resp. ended) at $q$ now begins (resp. ends) at $p$. In the process we introduce precisely one new crossing. The switching operation extends linearly to a $\Z_2$-linear map $W: \widehat{CS}(\Sigma, F) \To \widehat{CS}(\Sigma, F)$. See figure \ref{fig:switching}.

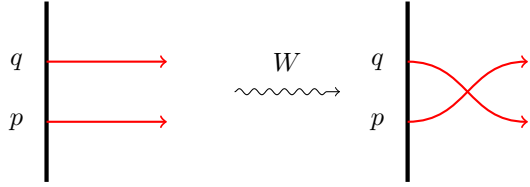
\begin{figure}[ht]
\begin{center}
\begin{tikzpicture}[
scale=0.8, 
string/.style={thick, draw=red, -to},
boundary/.style={ultra thick}]

\draw [boundary] (0,0) -- (0,3);
\draw [string] (0,1) -- (2,1);
\draw [string] (0,2) -- (2,2);
\draw (-0.5,1) node {$p$};
\draw (-0.5,2) node {$q$};

\draw [shorten >=1mm, -to, decorate, decoration={snake,amplitude=.4mm, segment length = 2mm, pre=moveto, pre length = 1mm, post length = 2mm}]
(3,1.5) -- (5,1.5);
\draw (4,2) node {$W$};

\draw [boundary] (6,0) -- (6,3);
\draw [string] (6,1) .. controls (7,1) and (7,2) .. (8,2);
\draw [string] (6,2) .. controls (7,2) and (7,1) .. (8,1);
\draw (5.5,1) node {$p$};
\draw (5.5,2) node {$q$};

\end{tikzpicture}
\caption{Switching operation.} \label{fig:switching}
\end{center}
\end{figure}

Resolving crossings in $Ws$ we obtain $\partial W s = s + w \partial s$, as shown in figure \ref{fig:chain_homotopy}. So $\partial W + W \partial = 1$, and $W$ is a chain homotopy between the chain maps $1$ and $0$ on $\widehat{CS}(\Sigma, F)$.

\begin{figure}[h]
\begin{center}
\begin{tikzpicture}[
scale=0.8, 
string/.style={thick, draw=red, -to},
boundary/.style={ultra thick}]

\draw (-1,1.5) node {$\partial$};
\draw [boundary] (0,0) -- (0,3);
\draw [rounded corners=8pt] (0,3) -- (3,3) -- (3,0) -- (0,0);
\draw [string] (0,1) .. controls (0.5,1) and (0.5,2) .. (1,2);
\draw [string] (0,2) .. controls (0.5,2) and (0.5,1) .. (1,1);
\draw (2,1.5) node {$s$};

\draw (3.5,1.5) node {$=$};

\draw [xshift=4 cm, boundary] (0,0) -- (0,3);
\draw [xshift=4 cm, rounded corners=8pt] (0,3) -- (3,3) -- (3,0) -- (0,0);
\draw [xshift=4 cm, string] (0,1) -- (1,1);
\draw [xshift=4 cm, string] (0,2) -- (1,2);
\draw [xshift=4 cm] (2,1.5) node {$s$};

\draw (7.5,1.5) node {$+$};

\draw [xshift=8 cm, boundary] (0,0) -- (0,3);
\draw [xshift=8 cm, rounded corners=8pt] (0,3) -- (3,3) -- (3,0) -- (0,0);
\draw [xshift=8 cm, string] (0,1) .. controls (0.5,1) and (0.5,2) .. (1,2);
\draw [xshift=8 cm, string] (0,2) .. controls (0.5,2) and (0.5,1) .. (1,1);
\draw [xshift=8 cm] (2,1.5) node {$\partial s$};

\draw (1.5,-1) node {$\partial Ws$};
\draw (3.5,-1) node {$=$};
\draw (5.5,-1) node {$s$};
\draw (7.5,-1) node {$+$};
\draw (9.5,-1) node {$W \partial s$};

\end{tikzpicture}
\caption{The switching operation $W$ is a chain homotopy.} \label{fig:chain_homotopy}
\end{center}
\end{figure}

Clearly a similar argument applies if the two adjacent points $p,q$ lie in $F_{out}$ rather than $F_{in}$. We immediately obtain the following result.
\begin{prop}[\cite{Mathews_Schoenfeld12_string}]
\label{prop:alternating_distinct_zero}
If $(\Sigma, F)$ is a weakly marked surface such that some boundary component of $\Sigma$ contains two adjacent distinct points of $F$ of the same sign, then $\widehat{HS}(\Sigma,F) = 0$.
\qed
\end{prop}

\begin{proof}[Proof of theorem \ref{thm:non-sutured_zero_iso}]
Let $(\An,F)$ be a non-alternating weakly marked annulus. If $F$ has two distinct consecutive points of the same sign on some boundary component of $\An$, then by proposition \ref{prop:alternating_distinct_zero} $\widehat{HS}(\An,F)=0$; so now assume this is not the case. Then each boundary component either contains an even number of points of $F$ alternating in sign, or contains a single point of $F$. As $F$ is not alternating, the only possibility is $F=F_{1,1}$, and each boundary component of $\An$ contains precisely one point of $F$. But by theorem \ref{thm:one_one_homology} $\widehat{HS}(\An,F_{1,1})=0$.
\end{proof}

\section{Annuli with two marked points on both boundaries}
\label{sec:annuli_two_marked_points_both_boundaries}

We now turn to $(\An, F_{2,2})$, the annulus with two alternating marked points on each boundary component. As noted in the introduction, this is the most difficult case, and our results are partial.

\subsection{Description and decomposition of the chain complex}
\label{sec:ABCDX_chain_complex}

Denote the two components of $\partial \Sigma$ by $C_0$ and $C_1$. Let $F_i = F \cap C_i$ and $F_{i,in} = F_{in} \cap C_i$, $F_{i,out} = F_{out} \cap C_i$ so $|F_0| = |F_1| = 2$ and $|F_{0, in}| = |F_{0, out}| = |F_{1, in}| = |F_{1, out}| = 1$. We will draw annuli so that $C_0$ is the ``outside" and $C_1$ the ``inside". We will draw marked points with $F_{0,in}, F_{1,out}$ at the bottom and $F_{0,out}, F_{1,in}$ at the top. See figure \ref{fig:F22_schematic}

%We will draw annuli so that $C_0$ is the ``inside" and $\partial \Sigma_1$ the ``outside"'. We will draw marked points with $F_{0,in}$ and $F_{1,out}$ are at the top and $F_{0,out}$ and $F_{1,in}$ are at the bottom. See figure [***].

\begin{figure}[ht]
\begin{center}
\begin{tikzpicture}[
scale=1.5, 
string/.style={thick, draw=red, postaction={nomorepostaction, decorate, decoration={markings, mark=at position 0.5 with {\arrow{>}}}}}]

\draw (0,0) circle (1.2 cm);
\draw (0,0) circle (0.5 cm);

\fill [draw=red, fill=red] (0,-1.2) circle (2pt);
\fill [draw=red, fill=red] (0,-0.5) circle (2pt);
\fill [draw=red, fill=red] (0,0.5) circle (2pt);
\fill [draw=red, fill=red] (0,1.2) circle (2pt);

\draw (0,-1.4) node {$F_{0,in}$};
\draw (0,-0.25) node {$F_{1,out}$};
\draw (0,0.25) node {$F_{1,in}$};
\draw (0,1.4) node {$F_{0,out}$};

\draw (0.7,0) node {$C_1$};
\draw (1.4,0) node {$C_0$};

\end{tikzpicture}
\caption{Boundary components and marked points on $(\An, F_{2,2})$.}
\label{fig:F22_schematic}
\end{center}
\end{figure}
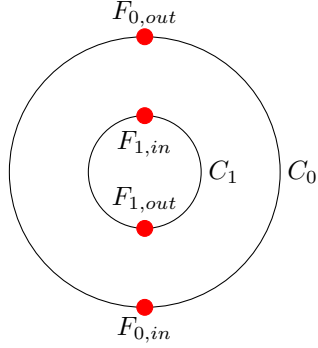

Homotopy classes of closed curves on $\An$ are again denoted by $x_n$. Homotopy classes of open strings on $(\An,F_{2,2})$ can be classified as follows.
\begin{defn}\
\begin{enumerate}
\item
An open string which begins and ends on the same $F_i$ is called \emph{insular}.
\begin{enumerate}
\item
An insular string joining the two points of $F_0$ runs $n + \frac{1}{2}$ times around the core of the annulus, for some $n \in \Z$. We denote the homotopy class of this curve by $a_{n + \frac{1}{2}}$. 

Let $\A, \A_+, \A_-$ be free $\Z_2$-modules on $\{a_n \; : \; n \in \Z + \frac{1}{2}\}$, and the positive and negative subsets thereof respectively.
\item
An insular string which joins the two points of $F_1$ also runs $n + \frac{1}{2}$ times around the core of the annulus for some $n \in \Z$, and we denote its homotopy class by $b_{n + \frac{1}{2}}$. 

Let $\B, \B_+, \B_-$ be free $\Z_2$-modules on $\{b_n \; : \; n \in \Z + \frac{1}{2}\}$, and the positive and negative subsets thereof respectively.
\end{enumerate}
\item
An open string which begins on $F_i$ and ends on $F_j$ for $i \neq j$ is called \emph{traversing}.
\begin{enumerate}
\item
A traversing string which joins $F_{1,in}$ to $F_{0,out}$ runs $n$ times around the core of the annulus, for some $n \in \Z$, and we denote its homotopy class by $c_n$.

Let $\C$ be the free $\Z_2$-module on $\{c_n \; : \; n \in \Z\}$.
\item
A traversing string which joins $F_{0,in}$ to $F_{1,out}$ runs $n$ times around the core of the annulus, for some $n \in \Z$, and we denote its homotopy class by $d_n$.

Let $\D$ be the free $\Z_2$-module on $\{d_n \; : \; n \in \Z\}$.
\end{enumerate}
\end{enumerate}
\end{defn}
Note the definition of $a_{n + \frac{1}{2}}$ and $\A, \A_\pm$ follows the notation of section \ref{sec:two_points_single_boundary}, and the $a_n$ are as shown in figure \ref{fig:a_curves}. The $\B, \B_\pm$ are defined in a similar vein. The $c_n$ follow the notation of section \ref{sec:one_marked_point_each_boundary}, and are as shown in figure \ref{fig:one_open_string_diagrams}. The $d_n$ are defined similarly. Some further examples are shown in figure \ref{fig:22_string_diagrams}.

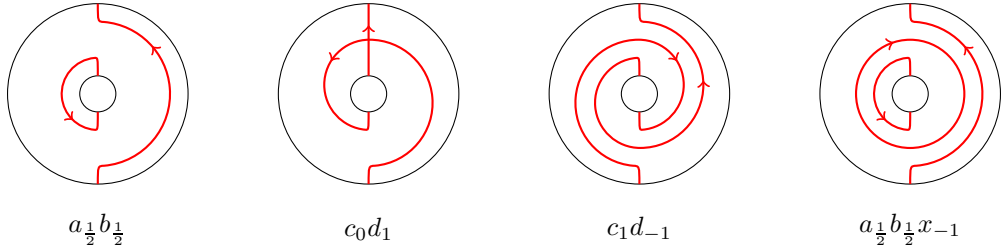
\begin{figure}[ht]
\begin{center}
\begin{tikzpicture}[
scale=1.2, 
string/.style={thick, draw=red, postaction={nomorepostaction, decorate, decoration={markings, mark=at position 0.7 with {\arrow{>}}}}}]

\draw (-4.5,0) circle (1 cm); 	
\draw (-4.5,0) circle (0.2 cm);
\draw (-1.5,0) circle (1 cm);
\draw (-1.5,0) circle (0.2 cm);
\draw (1.5,0) circle (1 cm);
\draw (1.5,0) circle (0.2 cm);
\draw (4.5,0) circle (1 cm);
\draw (4.5,0) circle (0.2 cm);

%a_{1/2} b_{1/2}
\draw [xshift = -4.5 cm, string] (0,-1) .. controls (0,-0.8) and (0,-0.8) .. (-85:0.8) arc (-85:85:0.8) .. controls (0,0.8) and (0,0.8) .. (0,1);
\draw [xshift = -4.5 cm, string] (0,0.2) .. controls (0,0.4) and (0,0.4) .. (95:0.4) arc (95:265:0.4) .. controls (0,-0.4) and (0,-0.4) .. (0,-0.2);

%c_0 d_1
\draw [xshift = -1.5 cm, string] (0,0.2) -- (0,1);
\draw [xshift = -1.5 cm, string] (0,-1) .. controls (0,-0.8) and (0,-0.8) .. (-85:0.8) arc (-85:90:0.7) arc (90:265:0.5) .. controls (0,-0.4) and (0,-0.4) .. (0,-0.2);

%c_1 d_{-1}
\draw [xshift= 1.5 cm, string] (0,0.2) .. controls (0,0.4) and (0,0.4) .. (95:0.4) arc (95:270:0.5) arc (-90:85:0.7) .. controls (0,0.8) and (0,0.8) .. (0,1);
\draw [xshift = 1.5 cm, string] (0,-1) .. controls (0,-0.8) and (0,-0.8) .. (-95:0.8) arc (265:90:0.7) arc (90:-85:0.5) .. controls (0,-0.4) and (0,-0.4) .. (0,-0.2);

%a_{1/2} b_{1/2} x_{-1}
\draw [xshift = 4.5 cm, string] (0,-1) .. controls (0,-0.8) and (0,-0.8) .. (-85:0.8) arc (-85:85:0.8) .. controls (0,0.8) and (0,0.8) .. (0,1);
\draw [xshift = 4.5 cm, string] (0,0.2) .. controls (0,0.4) and (0,0.4) .. (95:0.4) arc (95:265:0.4) .. controls (0,-0.4) and (0,-0.4) .. (0,-0.2);
\draw [xshift = 4.5 cm, string] (0.6,0) arc (0:-360:0.6);

\draw (-4.5,-1.5) node {$a_{\frac{1}{2}} b_{\frac{1}{2}}$};
\draw (-1.5,-1.5) node {$c_0 d_1$};
\draw (1.5,-1.5) node {$c_1 d_{-1}$};
\draw (4.5,-1.5) node {$a_{\frac{1}{2}} b_{\frac{1}{2}} x_{-1}$};

\end{tikzpicture}
\caption{Some string diagrams on $(\An, F_{2,2})$.}
\label{fig:22_string_diagrams}
\end{center}
\end{figure}

The two open strings in a string diagram on $(\An, F_{2,2})$ are either both insular or both traversing; we call the string diagram \emph{insular} or \emph{traversing} accordingly.

The (homotopy classes of) purely open string diagrams on $(\An, F_{2,2})$ are thus precisely given by $a_i b_j$ and $c_m d_n$, over all $i,j \in \Z + \frac{1}{2}$ and $m,n \in \Z$. Thus $\widehat{CS}^O(\An, F_{2,2}) \cong (\A \otimes_{\Z_2} \B) \oplus (\C \otimes_{\Z_2} \D)$. Note that monomials from $\A, \B, \C$ or $\D$ alone do not describe string diagrams: we require an $a_m$ together with a $b_n$, or a $c_m$ together with a $d_n$, to provide the required open strings. By lemma \ref{lem:CS_as_module} then
\begin{align*}
\widehat{CS}(\An,F_{2,2}) &\cong \X \otimes \left( \left( \A \otimes \B \right) \oplus \left( \C \otimes \D \right) \right) \\
&
\cong \left( \A \otimes \X \otimes \B \right) \oplus \left( \C \otimes \X \otimes \D \right)
\end{align*}
where all tensor products are over $\Z_2$.  In particular, an insular string diagram on $(\An,F_{2,2})$ can be written as $a_i b_j x^e$, and a traversing string diagram on $(\An,F_{2,2})$ as $c_m d_n x^e$, up to homotopy.

For the purposes of algebraic computations, it can be useful to think of elements of $\C \otimes \X \otimes \D$ as elements of $\X$ spaced on a $\Z \times \Z$ lattice: a general element of $\C \otimes \X \otimes \D$ can be given in the form $\sum_{m,n} c_m d_n p_{m,n}$, where each $p_{i,j} \in \X$, and we can think of the $p_{m,n}$ at the point $(i,j) \in \Z \times \Z$. Similarly, we can think of elements of $\A \otimes \X \otimes \B$, given in the form $\sum_{i,j} a_{i} b_{j} p_{i,j}$, as elements of $\X$ spaced on a $(\Z + \frac{1}{2}) \times (\Z + \frac{1}{2})$ lattice, placing $p_{i,j}$ at $(i,j)$.

\subsection{Description of differential}
\label{sec:insular_differential}
\label{sec:traversing_differential}

Consider an insular string diagram, of homotopy class $s = a_i b_j x^e$. By a homotopy relative to endpoints, we can separate the strings, i.e. make the strings pairwise disjoint. Thus on $\A \otimes \X \otimes \B$, the Goldman bracket vanishes and the differential obeys the Leibniz rule.

Since the open strings $a_i$ arose previously in considering $(\An, F_{0,2})$, $\partial a_i$ is as described in section \ref{sec:two_points_single_boundary}: for any $n \in \Z + \frac{1}{2}$ we have
\[
\partial a_i = \sum_{k+l=i, \; kl>0} a_k x_l.
\]
The calculation of $\partial b_j$ is similar; indeed $b_j$ is obtained from $a_j$ by a symmetry of the annulus, and there is an isomorphism of $\X$-modules, $\A \otimes \X \cong \B \otimes \X$ induced by $a_j \leftrightarrow b_j$, which commutes with $\partial$.
\[
\partial b_n = \sum_{i+j=n, \; ij>0} b_i x_j.
\]

By the Leibniz rule then we have, for a general insular string diagram $s = a_i b_j x^e)$,
\begin{align*}
\partial s &= \left( \partial a_i \right) b_j x^e + a_i \left( \partial b_j \right) x^e + a_i b_j \left( \partial x^e \right) \\
&= \sum_{k+l=i, \; kl>0} a_k b_j x_l x^e + \sum_{k+l=j, \; kl>0} a_i b_k x_l x^e + a_i b_j (\partial x^e),
\end{align*}
where $\partial x^e$ is given by the differential on $\X$. 

Note that the differential maps $\A \otimes \X \otimes \B$ into itself, so it is a subcomplex of $\widehat{CS}(\An, F_{2,2})$; as the Leibniz rule is satisfied, $\A \otimes \X \otimes \B$ is a differential $\X$-module and the homology $\widehat{HS}(\A \otimes \X \otimes \B)$ is an $H(\X)$-module.

Now consider a traversing string diagram, of homotopy class $s = c_i d_j x^e = c_i d_j \prod_{k \in \Z \backslash \{0\}} x_k^{e_k}$. The situation here is more complicated than the insular case. The $c_i$ and $d_j$ sometimes intersect each other and always intersect each $x_k$, so the Goldman bracket does not vanish and the Leibniz rule is not obeyed.

After a homotopy relative to endpoints, we may draw the string diagram so that all the curves intersect minimally. Then $c_i$ has no self-intersections; nor does $d_j$; and the intersection points are precisely as follows.
\begin{enumerate}
\item
Each $x_k$ has self-intersections; resolving them gives $\partial x_k$.
\item
Each $x_k$ intersects $c_i$ in $|k|$ points; resolving them gives $[c_i, x_k]$.
\item
Each $x_k$ intersects $d_j$ in $|k|$ points; resolving them gives $[d_j, x_k]$.
\item
The open strings $c_i$ and $d_j$ intersect in $|i+j|$ points; ;resolving them gives $[c_i, d_j]$.
\end{enumerate}
To see that $|c_i \cap d_j| = |i+j|$, first note that $|c_0 \cap d_j| = |j|$ for any integer $j$. Then note that under a Dehn twist on the annulus, $c_i \mapsto c_{i \pm 1}$ and $d_j \mapsto d_{j \mp 1}$; hence some number of Dehn twists takes $c_i d_j \mapsto c_0 d_{j+i}$, so $|c_i \cap d_j| = |c_0 \cap d_{i+j}| = |i+j|$.

For a general string diagram with homotopy class $c_i d_j x^e$, we can then write
\[
\partial  \left( c_i d_j x^e \right) = c_i d_j (\partial x^e) + \sum_{k \in \Z} e_k [c_i, x_k] d_j x_k^{-1} x^e + \sum_{k \in \Z} e_k c_i [d_j, x_k] x_k^{-1} x^e + [c_i, d_j] x^e ,
\]
where the four terms correspond to the four types of intersections listed above. It remains to compute the Goldman brackets in the above. 

Resolving an intersection point of $c_i$ and $x_k$, we obtain an open string running from $F_{1,in}$ to $F_{0,out}$, hence one of the $c_n$. As $c_i, x_k$ respectively run $i,k$ times around the annulus, we obtain $c_{i+k}$. Resolving all $|k|$ such crossings (mod 2) gives
\[
[c_i, x_k] = k c_{i+k}.
\]
Similarly, resolving the $|k|$ intersections of $x_k$ and $d_j$ gives
\[
[d_j, x_k] = k d_{j+k}.
\]

As for $[c_i, d_j]$, we have the following lemma.
\begin{lem}
\label{lem:bracket_ci_d_j}
Resolving an intersection between $c_i$ and $d_j$ produces a string diagram of the form $a_{i'} b_{j'}$, where $i', j'$ both have the same sign as $i+j$, and $i'+j' = i+j$. There are $|i+j|$ diagrams of this form and each appears precisely once as we resolve the $|i+j|$ crossings between $c_i$ and $d_j$. That is,
\[
[c_i, d_j] = \left\{ \begin{array}{ll}
a_{\frac{1}{2}} b_{i+j - \frac{1}{2}} + a_{\frac{3}{2}} b_{i+j - \frac{3}{2}} + \cdots + a_{i+j - \frac{1}{2}} b_{\frac{1}{2}} & i+j > 0 \\
0 & i+j = 0 \\
a_{-\frac{1}{2}} b_{i+j + \frac{1}{2}} + a_{-\frac{3}{2}} b_{i+j + \frac{3}{2}} + \cdots + a_{i+j + \frac{1}{2}} b_{-\frac{1}{2}} & i+j < 0
\end{array} \right\}
= \sum_{k+l=i+j, \; kl>0} a_k b_l
\]
\end{lem}

The expressions arising as $[c_i, d_j]$ appear frequently in the sequel; we call them $s_n$.
\begin{defn}
\label{defn:s_n_defn}
For an integer $n$, let $s_n \in \A \otimes \X \otimes \B$ be 
\[
s_n = \left\{ \begin{array}{ll}
a_{\frac{1}{2}} b_{n - \frac{1}{2}} +  a_{\frac{3}{2}} b_{n - \frac{3}{2}} + \cdots + a_{n - \frac{1}{2}} b_{\frac{1}{2}} & n > 0 \\
0 & n = 0 \\
a_{-\frac{1}{2}} b_{n + \frac{1}{2}} + a_{-\frac{3}{2}} b_{n + \frac{3}{2}} + \cdots + a_{n + \frac{1}{2}} b_{-\frac{1}{2}} & n < 0
\end{array} \right\}
= \sum_{k+l=n, \; kl>0} a_k b_l
\]
\end{defn}

Note $s_n$ contains precisely $|n|$ nonzero terms in the sum; if we regard elements $\sum_{i,j} a_i b_j p_{i,j}$ of $\A \otimes \X \otimes \B$ as polynomials $p_{i,j} \in \X$ placed at $(i,j)$ on the lattice $(\Z + \frac{1}{2}) \times (\Z + \frac{1}{2})$, then $s_n$ consists of $1$s placed along the ``diagonal" $i+j=n$, at points $(i,j)$ where $i,j$ have the same sign as $n$. That is,
\[
\ldots, \quad
s_{-1} = a_{-\frac{1}{2}} b_{-\frac{1}{2}}, \quad
s_0 = 0, \quad
s_1= a_{\frac{1}{2}} b_{\frac{1}{2}}, \quad
s_2= a_{\frac{1}{2}} b_{\frac{3}{2}} + a_{\frac{3}{2}} b_{\frac{1}{2}}, \quad \ldots
\]

Thus lemma \ref{lem:bracket_ci_d_j}, which we now prove, states that
\[
[c_i, d_j] = s_{i+j}.
\]
\begin{proof}
First consider resolving one of the $|j|$ crossings of $c_0$ and $d_j$. This gives two open strings: one starting along $c_0$ at $F_{1,in}$ and ending along $d_j$ at $F_{1,out}$; and one starting along $d_j$ at $F_{0,in}$ and ending along $c_0$ at $F_{0,out}$; the resulting open strings are therefore of the form $a_m b_n$. Now we note that $m,n$ both have the same sign as $j$, and $m+n = j$. As we resolve the $|j|$ crossings we obtain all such $a_m b_n$. This gives $[c_0, d_j] = s_j$ as desired.

Now a general pair of open strings $c_i, d_j$ can be taken by some Dehn twists to the pair $c_0, d_{i+j}$. We know $[c_0, d_{i+j}] = s_{i+j}$, which has every term of the form $a_m, b_n$. But the boundary-parallel open strings $a_m, b_n$ are invariant under Dehn twists; thus $[c_i, d_j] = s_{i+j}$ as well.
\end{proof}

We have now explicitly computed the differential on $\widehat{CS}(\An,F_{2,2}) \cong (\A \otimes \X \otimes \B) \oplus (\C \otimes \X \otimes \D)$.
\begin{prop}
\label{prop:differential_on_ABCDX_complex}
The differential $\partial$ on $\widehat{CS}(A,F)$ is given by
\begin{align*}
\partial \left( a_i b_j x^e \right) &= (\partial a_i) b_j x^e + a_i (\partial b_j) x^e + a_i b_j (\partial x^e) \\
&= \sum_{i'+k'=i, \; i' k'>0} a_{i'} b_j x_{k'} x^e + \sum_{j'+k'=j, \; j' k' >0} a_i b_{j'} x_{k'} x^e + a_i b_j (\partial x^e) \\
\partial \left( c_i d_j x^e \right) &= c_i d_j (\partial x^e) + s_{i+j} x^e + \sum_{k \in \Z} k e_k \left( c_{i+k} d_j + c_i d_{j+k} \right) x_k^{-1} x^e.
\end{align*}
\qed
\end{prop}

We next turn to the homology of $\widehat{CS}(\Sigma,F)$. In sections \ref{sec:ABX_homology}--\ref{sec:completing_insular_calculation} we compute the homology of the subcomplex $\A \otimes \X \otimes \B$; in sections \ref{sec:towards_full_22_homology}--\ref{sec:homomorphism_from_disc} we consider the homology of the entire complex.

\subsection{Homology of insular string diagrams I: simplifying cycles}
\label{sec:ABX_homology}

We focus on the subcomplex $\A \otimes \X \otimes \B$ of $\widehat{CS}(\Sigma,F)$, generated by insular string diagrams. 

We have defined $\A_\pm, \B_\pm$ so that $\A = \A_+ \oplus \A_-$ and $\B = \B_+ \oplus \B_-$. From proposition \ref{prop:differential_on_ABCDX_complex}, $\partial (a_i b_j x^e)$ is a sum of terms $a_{i'} b_{j'} x^{e'}$ where $i,i'$ have the same sign, and $j,j'$ have the same sign.  Thus $\A \otimes \X \otimes \B$ splits as a direct sum of four differential $\X$-submodules:
\[
\A \otimes \X \otimes \B = \left( \A_+ \otimes \X \otimes \B_+ \right) \oplus \left( \A_+ \otimes \X \otimes \B_- \right) \oplus \left( \A_- \otimes \X \otimes \B_+ \right) \oplus \left( \A_- \otimes \X \otimes \B_- \right).
\]
We will deal with these submodules separately. We shall find that the modules $\A_\pm \otimes \X \otimes \B_\pm$ behave rather differently from the modules $\A_\pm \otimes \X \otimes \B_\mp$.

We will first compute the homology of $\A_+ \otimes \X \otimes \B_+$. The argument is quite long and takes up to the end of section \ref{sec:homology_of_++_computed}. The method is similar to section \ref{sec:homology_A_+_tensor_X}, simplifying cycles to have lower ``degree". 

\begin{defn}
\label{def:a-degree}
A general element $f$ of $\A_+ \otimes \X \otimes \B$ has the form
\[
f = a_{\frac{1}{2}} p_{\frac{1}{2}} + a_{\frac{3}{2}} p_{\frac{3}{2}} + \cdots + a_{n+\frac{1}{2}} p_{n+\frac{1}{2}},
\]
where each $p_{i+\frac{1}{2}} \in \X \otimes \B$, and $p_{n+\frac{1}{2}} \neq 0$. The \emph{$a$-degree} of $f$ is $n+\frac{1}{2}$.

We write $O(a_m)$ to denote an element of $\A_+ \otimes \X \otimes \B$ of $a$-degree $\leq m$.
\end{defn}
(We could equally define a $b$-degree; however we will not need it.) Note that this definition applies to both $\A_+ \otimes \X \otimes \B_+$ and $\A_+ \otimes \X \otimes \B_-$.

Each $p_i \in \X \otimes \B_+$ is a ``polynomial" in the $b_j$ and $x_k$, with each term containing precisely one $b_j$ factor. By the Leibniz rule, for an $f$ of degree $n-\frac{1}{2}$ we can write
\[
f = \sum_{i=1}^n a_{i-\frac{1}{2}} p_{i-\frac{1}{2}}
\quad \text{so} \quad
\partial f = \sum_{i=1}^n \left( \partial a_{i-\frac{1}{2}} \right) p_{i-\frac{1}{2}} + a_{i-\frac{1}{2}} \left( \partial p_{i-\frac{1}{2}} \right).
\]

The key to the computation of $H(\A_+ \otimes \X \otimes \B_+)$ is the following lemma. Analogously to lemma \ref{lem:a_simplification}, it shows how to simplify a given cycle $f \in \A_+ \otimes \X \otimes \B_+$ to a homologous one of smaller $a$-degree. Recall definition \ref{defn:s_n_defn} of $s_n$.
\begin{lem}
\label{lem:ABX_standard_form}
Suppose $f \in \A_+ \otimes \X \otimes \B_+$ has $a$-degree $n-\frac{1}{2}$ and satisfies $\partial f = 0$. Then 
\begin{align*}
f &= \left( a_{\frac{1}{2}} b_{\frac{1}{2}} \right) q_1 + \left( a_{\frac{1}{2}} b_{\frac{3}{2}} + a_{\frac{3}{2}} b_{\frac{1}{2}} \right) q_2 + \cdots +
\left( a_{\frac{1}{2}} b_{n - \frac{1}{2}} + a_{\frac{3}{2}} b_{n - 1} + \cdots + a_{n - \frac{1}{2}} b_{\frac{1}{2}} \right) q_n + \partial g \\
&= s_1 q_1 + s_2 q_2 + \cdots + s_n q_n + \partial g,
\end{align*}
for some $g = O(a_{n-\frac{1}{2}}) \in \A_+ \otimes \X \otimes \B_+$ and $q_j \in \X$, for $1 \leq j \leq n$, where each $q_j$ is a clean polynomial.
\end{lem}

Before commencing the proof, recall our computation of $H(\A_+ \otimes \X)$. We have
\[
H(\A_+ \otimes \X) \cong \bar{x}_1 H(\X) = \bar{x}_1 \frac{ \Z[\ldots, \bar{x}_{-3}, \bar{x}_{-1}, \bar{x}_1, \bar{x}_3, \ldots] }{ ( \ldots, \bar{x}_{-3}^2, \bar{x}_{-1}^2, \bar{x}_1^2, \bar{x}_3^2, \ldots )},
\]
with the isomorphism induced at the chain level by $a_{n - \frac{1}{2}} x^e \mapsto x_n x^e$. In particular (proposition \ref{prop:strong_homology_with_a1}), if $f \in \A_+ \otimes \X$ has $\partial f = 0$, then $f = a_{\frac{1}{2}} p + \partial g$, where $g \in \A_+ \otimes \X$ and $p$ is a clean polynomial.

The chain complex $\B_+ \otimes \X$ is isomorphic to $\A_+ \otimes \X$ via $b_i \mapsto a_i$. Then, similarly, $H(\B_+ \otimes \X) \cong \bar{x}_1 H(\X)$. Moreover, if $f \in B_+ \otimes \X$ and $\partial f = 0$ then $f = b_{\frac{1}{2}} p + \partial g$, where $g \in \B_+ \otimes \X$ and $p \in \X$ is a clean polynomial.

As a preliminary, we demonstrate lemma \ref{lem:ABX_standard_form} when $n - \frac{1}{2} = \frac{1}{2}$, i.e. $n=1$. In this case $f = a_{\frac{1}{2}} p_\frac{1}{2}$ with $p_{\frac{1}{2}} \in \X \otimes \B_+$, so $0 = \partial f = a_{\frac{1}{2}} \partial p_\frac{1}{2}$. Hence $p_\frac{1}{2} \in \B_+ \otimes \X$ satisfies $\partial p_\frac{1}{2} = 0$, so the previous paragraph gives $p_\frac{1}{2} = b_{\frac{1}{2}} q_1 + \partial r$, where $q_1$ is a clean polynomial. We then have
\[
f = a_{\frac{1}{2}} p_{\frac{1}{2}} = a_{\frac{1}{2}} b_{\frac{1}{2}} q_1 + a_{\frac{1}{2}} \partial r = s_1 q_1 + \partial \left( a_{\frac{1}{2}} r \right)
\]
as desired. The result for general $n$, though technically complicated, is based on a repetition of this argument.
\begin{proof}
As $f$ has $a$-degree $n-\frac{1}{2}$,
\[
f = a_{n - \frac{1}{2}} p_{n-\frac{1}{2}} + a_{n - \frac{3}{2}} p_{n-\frac{3}{2}} + \cdots + a_{\frac{1}{2}} p_{\frac{1}{2}}
\quad \text{where each } p_i \in \X \otimes \B_+.
\]
Differentiating $f$ gives
\begin{align*}
0 = \partial f &= a_{n - \frac{1}{2}} \left( \partial p_{n-\frac{1}{2}} \right) + \left( \partial a_{n - \frac{1}{2}} \right) p_{n-\frac{1}{2}} + a_{n - \frac{3}{2}} \left( \partial p_{n-\frac{3}{2}} \right) + + O(a_{n - \frac{5}{2}}) \\
&= a_{n - \frac{1}{2}} \left( \partial p_{n-\frac{1}{2}} \right) + a_{n - \frac{3}{2}} \left( x_1 p_{n-\frac{1}{2}} + \partial p_{n-\frac{3}{2}} \right) + O(a_{n - \frac{5}{2}}).
\end{align*}
Equating coefficients of $a_{n - \frac{1}{2}}$ and $a_{n - \frac{3}{2}}$ gives
\[
\partial p_{n-\frac{1}{2}} = 0 \quad \text{and} \quad x_1 p_{n-\frac{1}{2}} = \partial p_{n-\frac{3}{2}}.
\]
Using proposition \ref{prop:strong_homology_with_a1} on $p_{n-\frac{1}{2}} \in \B_+ \otimes \X$, we have $p_{n-\frac{1}{2}} = b_{\frac{1}{2}} q_n + \partial r$, where $q_n$ is a clean polynomial, and $r \in \B_+ \otimes \X$. We then have
\[
f = a_{n - \frac{1}{2}} b_{\frac{1}{2}} q_n + a_{n-\frac{1}{2}} \partial r + a_{n-\frac{3}{2}} p_{n-\frac{3}{2}} + \cdots + a_\frac{1}{2} p_\frac{1}{2}.
\]
Noting that $\partial \left( a_{n - \frac{1}{2}} r \right) = a_{n-\frac{1}{2}} \partial r + O(a_{n-\frac{3}{2}})$ produces a term $a_{n-\frac{1}{2}} \partial r$ as in the above, we have
\begin{align*}
f &= a_{n-\frac{1}{2}} b_{\frac{1}{2}} q_n + \partial \left( a_{n-\frac{1}{2}} r \right) + O(a_{n-\frac{3}{2}}) \\
&= a_{n - \frac{1}{2}} b_{\frac{1}{2}} q_n + f_{n-1} + \partial g_{n-1},
\end{align*}
where $f_{n-1}, g_{n-1} \in \A_+ \otimes \X \otimes \B_+$, $f_{n-1} = O(a_{n-\frac{3}{2}})$ and $g_{n-1} = a_{n-\frac{1}{2}} r = O(a_{n-\frac{1}{2}})$. (When $n=1$, this is just the preliminary demonstration given above.)

We claim now that, for each integer $i$ with $0 \leq i \leq n-1$, we can write
\begin{align*}
f &= a_{n - \frac{1}{2}} b_{\frac{1}{2}} q_n + a_{n - \frac{3}{2}} \left( b_{\frac{3}{2}} q_n + b_{\frac{1}{2}} q_{n-1} \right) + a_{n- \frac{5}{2}} \left( b_{\frac{5}{2}} q_n + b_{\frac{3}{2}} q_{n-1} + b_{\frac{1}{2}} q_{n-2} \right) \\
& \quad \quad + \cdots + a_{n - i - \frac{1}{2}} \left( b_{i+\frac{1}{2}} q_n + b_{i - \frac{1}{2}} q_{n-1} + \cdots + b_{\frac{1}{2}} q_{n-i} \right) + f_{n-i-1} + \partial g_{n-i-1} \\
&= \sum_{k=0}^{i} a_{n - k - \frac{1}{2}} \sum_{l=0}^{k} b_{l+\frac{1}{2}} q_{n-k+l} + f_{n-i-1} + \partial g_{n-i-1},
\end{align*}
where $q_n, q_{n-1}, \ldots q_{n-i} \in \X$ are clean polynomials, and $f_{n-i-1}, g_{n-i-1} \in \A_+ \otimes \X \otimes \B_+$, where $f_{n-i-1} = O(a_{n-i-\frac{3}{2}})$ and $g_{n-i-1} = O(a_{n-\frac{1}{2}})$. We have just shown this claim for $i=0$. So suppose that the claim holds for a particular value of $i$, where $0 \leq i \leq n-2$; we shall show it holds for $i+1$.

Let then $f$ be given as claimed. Consider differentiating $f$; we have $\partial f = 0$; moreover, differentiating the last two terms gives $\partial \left( f_{n-i-1} + \partial g_{n-i-1} \right) = O \left( a_{n-i-\frac{3}{2}} \right)$. From the other terms of $f$, we then obtain
\[
\partial \left( \begin{array}{c} a_{n - \frac{1}{2}} b_{\frac{1}{2}} q_n + a_{n - \frac{3}{2}} \left( b_{\frac{3}{2}} q_n + b_{\frac{1}{2}} q_{n-1} \right) + a_{n- \frac{5}{2}} \left( b_{\frac{5}{2}} q_n + b_{\frac{3}{2}} q_{n-1} + b_{\frac{1}{2}} q_{n-2} \right) \\
\quad \quad + \cdots + a_{n - i - \frac{1}{2}} \left( b_{i+\frac{1}{2}} q_n + b_{i - \frac{1}{2}} q_{n-1} + \cdots + b_{\frac{1}{2}} q_{n-i} \right)
\end{array} \right) = O(a_{n-i-\frac{3}{2}}).
\]
We thus consider the terms of $a$-degree $n-i-\frac{3}{2}$ in $\partial f$. Note that in any differential $\partial ( a_i b_j q_k)$, there is a unique term with $a$-degree $a_{l-\frac{1}{2}}$, for a positive integer $l < i$, namely $a_l b_j q_k x_{i-l}$. Further, since by assumption $f_{n-i-1} = O \left( a_{n-i - \frac{3}{2}} \right)$, let $f_{n-i-1} = a_{n-i-\frac{3}{2}} p_{n-i-\frac{3}{2}} + O(a_{n-i-\frac{5}{2}})$ where $p_{n-i-\frac{3}{2}} \in \B_+ \otimes \X$.  We obtain
\begin{align*}
0 = \partial f &= a_{n-i-\frac{3}{2}} \Big[ x_{i+1} b_{\frac{1}{2}} q_n  + x_i \left( b_{\frac{3}{2}} q_n + b_{\frac{1}{2}} q_{n-1} \right) + x_{i-1} \left( b_{\frac{5}{2}} q_n + b_{\frac{3}{2}} q_{n-1} + b_{\frac{1}{2}} q_{n-2} \right) \\
& \quad \quad + \cdots + x_1 \left( b_{i+\frac{1}{2}} q_n + b_{i-\frac{1}{2}} q_{n-1} + \cdots + b_{\frac{1}{2}} q_{n-i} \right) + \partial p_{n-i-\frac{3}{2}} \Big] +  O(a_{n-i-\frac{5}{2}}).
\end{align*}
The coefficient of $a_{n-i-\frac{3}{2}}$ must be zero, hence
\begin{align*}
\partial p_{n-i-\frac{3}{2}} &= 
x_{i+1} b_{\frac{1}{2}}   q_n  + x_i \left( b_{\frac{3}{2}} q_n + b_{\frac{1}{2}} q_{n-1} \right) + x_{i-1} \left( b_{\frac{5}{2}} q_n + b_{\frac{3}{2}} q_{n-1} + b_{\frac{1}{2}} q_{n-2} \right) \\
& \quad \quad + \cdots + x_1 \left( b_{i+\frac{1}{2}} q_n + b_{i-\frac{1}{2}} q_{n-1} + \cdots + b_{\frac{1}{2}} q_{n-i} \right) \\
&= \left( b_{\frac{1}{2}} x_{i+1} + b_{\frac{3}{2}} x_i  + \cdots + b_{i+\frac{1}{2}} x_1 \right) q_n
+ \left( b_{\frac{1}{2}} x_i + b_{\frac{3}{2}} x_{i-1} + \cdots + b_{i-\frac{1}{2}} x_1 \right) q_{n-1}
+ \cdots
+ \left( b_{\frac{1}{2}} x_1 \right) q_{n-i} \\
&= \left( \partial b_{i+\frac{3}{2}} \right) q_n + \left( \partial b_{i+\frac{1}{2}} \right) q_{n-1} + \cdots + \left( \partial b_{\frac{3}{2}} \right) q_{n-i} \\
&= \partial \left( b_{i+\frac{3}{2}} q_n + b_{i+\frac{1}{2}} q_{n-1} + \cdots + b_{\frac{3}{2}} q_{n-i} \right)
\end{align*}
In the second line we regrouped; in the third line used the formula for $\partial b_j$; and in the last line used $\partial q_j = 0$, which follows from our assumptions on the $q_j$.

Thus we have a cycle $p_{n-i-\frac{3}{2}} + b_{i+\frac{3}{2}} q_n + b_{i+\frac{1}{2}} q_{n-1} + \cdots + b_{\frac{3}{2}} q_{n-i}$ in $\B_+ \otimes \X$, and by proposition \ref{prop:strong_homology_with_a1} this is homologous to $b_{\frac{1}{2}} q_{n-i-1}$, where $q_{n-i-1}$ is a clean polynomial. This gives
\[
p_{n-i-\frac{3}{2}} = b_{i+\frac{3}{2}} q_n + b_{i+\frac{1}{2}} q_{n-1} + \cdots + b_{\frac{3}{2}} q_{n-i} + b_{\frac{1}{2}} q_{n-i-1} + \partial r,
\]
for some $r \in \B_+ \otimes \X$.

Returning to $f$ and substituting this expression for $p_{n-i-\frac{3}{2}}$, we have
\begin{align*}
f &=
a_{n - \frac{1}{2}} b_{\frac{1}{2}} q_n + a_{n - \frac{3}{2}} \left( b_{\frac{3}{2}} q_n + b_{\frac{1}{2}} q_{n-1} \right) + a_{n- \frac{5}{2}} \left( b_{\frac{5}{2}} q_n + b_{\frac{3}{2}} q_{n-1} + b_{\frac{1}{2}} q_{n-2} \right) \\
& \quad \quad + \cdots + a_{n - i - \frac{1}{2}} \left( b_{i+\frac{1}{2}} q_n + b_{i - \frac{1}{2}} q_{n-1} + \cdots + b_{\frac{1}{2}} q_{n-i} \right) + \\
& \quad \quad a_{n-i-\frac{3}{2}} \left( b_{i+\frac{3}{2}} q_n + b_{i+\frac{1}{2}} q_{n-1} + \cdots + b_{\frac{3}{2}} q_{n-i} + b_{\frac{1}{2}} q_{n-i-1} + \partial r \right) 
+ O(a_{n-i-\frac{5}{2}})
 + \partial g_{n-i-1}
\end{align*}
Finally, using $a_{n-i-\frac{3}{2}} \partial r = \partial \left( a_{n-i-\frac{3}{2}} r \right) + O(a_{n - i - \frac{5}{2}})$, we have
\begin{align*}
f &=
a_{n - \frac{1}{2}} b_{\frac{1}{2}} q_n + a_{n - \frac{3}{2}} \left( b_{\frac{3}{2}} q_n + b_{\frac{1}{2}} q_{n-1} \right) + a_{n- \frac{5}{2}} \left( b_{\frac{5}{2}} q_n + b_{\frac{3}{2}} q_{n-1} + b_{\frac{1}{2}} q_{n-2} \right) \\
& \quad \quad + \cdots + a_{n - i - \frac{1}{2}} \left( b_{i+\frac{1}{2}} q_n + b_{i - \frac{1}{2}} q_{n-1} + \cdots + b_{\frac{1}{2}} q_{n-i} \right)  \\
& \quad \quad + a_{n-i-\frac{3}{2}} \left( b_{i+\frac{3}{2}} q_n + b_{i+\frac{1}{2}} q_{n-1} + \cdots + b_{\frac{3}{2}} q_{n-i} + b_{\frac{1}{2}} q_{n-i-1} \right)
+ f_{n-i-2} + \partial g_{n-i-2}
\end{align*}
where $f_{n-i-2} = O(a_{n-i-\frac{5}{2}})$, and $g_{n-i-2} = a_{n-i-\frac{3}{2}} r + g_{n-i-1} = O(a_{n-\frac{1}{2}})$. This puts $f$ in the desired form for $i+1$, proving the claim. 

Now consider the claim with $i=n-1$. It says that 
\begin{align*}
f &= a_{n - \frac{1}{2}} b_{\frac{1}{2}} q_n + a_{n - \frac{3}{2}} \left( b_{\frac{3}{2}} q_n + b_{\frac{1}{2}} q_{n-1} \right) + a_{n- \frac{5}{2}} \left( b_{\frac{5}{2}} q_n + b_{\frac{3}{2}} q_{n-1} + b_{\frac{1}{2}} q_{n-2} \right) \\
& \quad \quad + \cdots + a_{\frac{1}{2}} \left( b_{n-\frac{1}{2}} q_n + b_{n - \frac{3}{2}} q_{n-1} + \cdots + b_{\frac{1}{2}} q_{1} \right) + f_{0} + \partial g_{0},
\end{align*}
where $f_0 = O(a_{-\frac{1}{2}})$, hence $f_0 = 0$, and $g_0 = O(a_{n-\frac{1}{2}})$. Writing $g=g_0$ this rearranges as
\[
f = s_n q_n + s_{n-1} q_{n-1} + \cdots + s_1 q_1 + \partial g
\]
as desired.
\end{proof}

This technical lemma shows that any cycle in $\A_+ \otimes \X \otimes \B_+$ is homologous to an element in the standard form $s_1 q_1 + \cdots + s_n q_n$. It is clear that, since $\partial s_i = \partial q_i = 0$, any element of this form is a cycle. We will next show that such representatives are unique, proving an analogy of proposition \ref{lem:stronger_homology_with_a1}. This will give us an explicit description of $H(\A_+ \otimes \X \otimes \B_+)$.

For this uniqueness result, however, we work in a truncated complex, and then take a direct limit.

\subsection{Homology of insular string diagrams II: truncated complex}

We now restrict to those elements of $\A_+ \otimes \X \otimes \B_+$ which have bounded $a$-degree. (We could equally well truncate with respect to $b$-degree, but we do not need it.)

\begin{defn}
Let $N$ be a positive integer. The $\Z_2$-module $\A_+^{< N}$ is the submodule of $\A$ generated by $a_n$ with $0 < n < N$.
\end{defn}

Note $\A_+^{<N}$ has $\Z_2$-rank $N$, with basis $\{a_\frac{1}{2}, a_\frac{3}{2}, \ldots, a_{N-\frac{1}{2}} \}$. Our strategy is to consider the homology in the ascending sequence 
\[
\A_+^{< 1} \otimes \X \otimes \B_+ \subset \A_+^{< 2} \otimes \X \otimes \B_+ \subset \cdots,
\]
whose direct limit is $\A_+ \otimes \X \otimes \B_+$. As the differential on $\A_+ \otimes \X \otimes \B_+$ lowers $a$-degree (or keeps it constant, proposition \ref{prop:differential_on_ABCDX_complex}), this is an ascending sequence of subcomplexes. As homology computes with direct limits, we will the homology of $\A_+ \otimes \X \otimes \B_+$ is the direct limit of the $H(\A_+^{<N} \otimes \X \otimes \B_+)$.

We first restate lemma \ref{lem:ABX_standard_form} in the truncated case.
\begin{lem}
\label{lem:standard_form_in_truncated_complex}
Suppose $f \in \A_+^{< n} \otimes \X \otimes \B_+$ and satisfies $\partial f = 0$. Then 
\[
f = s_1 q_1 + s_2 q_2 + \cdots + s_n q_n + \partial g,
\]
for some $g \in \A_+^{<n} \otimes \X \otimes \B_+$ and $q_j \in \X$, for $1 \leq j \leq n$, where each $q_j$ is a clean polynomial.
\end{lem}

(Note that $s_1, s_2, \ldots, s_n$ are precisely the $s_i$ which lie in $\A_+^{< n} \otimes \X \otimes \B_+$.)

\begin{proof}
Let the given $f$ have $a$-degree $m-\frac{1}{2}$, for some integer $m$, $1 \leq m \leq n$. Lemma \ref{lem:ABX_standard_form} shows how to write $f$ in the form $s_1 q_1 + \cdots + s_m q_m + \partial g$, where the $s_i$ and $q_i$ have the desired form, and $g \in \A_+ \otimes \X \otimes \B_+$ has $a$-degree $\leq m-\frac{1}{2} \leq n - \frac{1}{2}$, hence $g \in \A_+^{<n} \otimes \X \otimes \B_+$.
\end{proof}

Thus every homology class of $\A_+^{<n} \otimes \X \otimes \B_+$ has a representative in the ``standard form" $s_1 q_1 + \cdots + s_n q_n$. Since $\partial s_i = \partial q_i = 0$, every such element is a cycle and represents some homology class. We will show that each homology class has precisely one representative in this standard form; equivalently, as the next proposition shows, the only ``standard form" element which is a boundary is $0$.
\begin{prop}
\label{prop:truncated_injectivity}
Let $n \in \Z_{>0}$, and suppose $q_1, \ldots, q_n \in \X$ are clean polynomials such that
\[
s_1 q_1 + s_2 q_2 + \cdots + s_n q_n = \partial r
\]
for some $r \in \A_+^{< n} \otimes \X \otimes \B_+$. Then $q_1 = q_2 = \cdots = q_n = 0$.
\end{prop}

\begin{proof}
Let $f = s_1 q_1 + \cdots + s_n q_n$. As $r \in \A_+^{<n} \otimes \X \otimes \B_+$, we can write
\[
r = a_{n-\frac{1}{2}} r_n + a_{n-\frac{3}{2}} r_{n-1} + a_{n - \frac{5}{2}} r_{n-2} + \cdots + a_{\frac{3}{2}} r_2 + a_{\frac{1}{2}} r_1,
\]
where each $r_1, r_2, \ldots, r_n \in \B_+ \otimes \X$. Differentiating gives
\begin{align*}
\partial r &= a_{n - \frac{1}{2}} \left( \partial r_n \right) + a_{n - \frac{3}{2}} \left( x_1 r_n + \partial r_{n-1} \right)  + a_{n - \frac{5}{2}} \left( x_2 r_n + x_1 r_{n-1} + \partial r_{n-2} \right) \\
& \quad \quad
+ \cdots +
a_{n-i-\frac{1}{2}} \left( x_i r_n + x_{i-1} r_{n-1} + \cdots + x_1 r_{n-i+1} + \partial r_{n-i} \right)
+ \cdots \\
& \quad \quad 
+
a_{\frac{1}{2}} \left( x_{n-1} r_n + x_{n-2} r_{n-1} + \cdots + x_1 r_2 + \partial r_1 \right)
\end{align*}
On the other hand, we can write out the terms of $f$ by $a$-degree as follows. 
\begin{align*}
f &= a_{n - \frac{1}{2}} \left( b_{\frac{1}{2}} q_n \right) 
+ a_{n - \frac{3}{2}} \left( b_{\frac{3}{2}} q_n + b_{\frac{1}{2}} q_{n-1} \right) 
+ a_{n - \frac{5}{2}} \left( b_{\frac{5}{2}} q_n + b_{\frac{3}{2}} q_{n-1} + b_{\frac{1}{2}} q_{n-2} \right) \\
& \quad \quad
+ \cdots
+ a_{n - i - \frac{1}{2}} \left( b_{i + \frac{1}{2}} q_n + b_{i - \frac{1}{2}} q_{n-1} + \cdots + b_{\frac{3}{2}} q_{n-i+1} + b_{\frac{1}{2}} q_{n-i} \right)
+ \cdots \\
& \quad \quad
+ a_{\frac{1}{2}} \left( b_{n-\frac{1}{2}} q_n + b_{n-\frac{3}{2}} q_{n-1} + \cdots + b_{\frac{3}{2}} q_2 + b_{\frac{1}{2}} q_1 \right)
\end{align*}
Equating coefficients of $a_{n-\frac{1}{2}}$ in $f = \partial r$ gives
\[
b_{\frac{1}{2}} q_n = \partial r_n,
\]
so that $b_{\frac{1}{2}} q_n$ is a boundary in $\B_+ \otimes \X$. But by our computation of $H(\A_+ \otimes \X) \cong H(\B_+ \otimes \X)$ of section \ref{sec:homology_A_+_tensor_X} (specifically proposition \ref{lem:stronger_homology_with_a1}) then $q_n = 0$. Equating coefficients of $a_{n - \frac{3}{2}}$ then gives
\begin{equation}
\label{eqn:q_n-3/2_coeffs}
b_{\frac{1}{2}} q_{n-1} = x_1 r_n + \partial r_{n-1}.
\end{equation}

We note that $x_1 (\partial r_n) = \partial (x_1 r_n) = \partial ( \partial r_{n-1} + b_\frac{1}{2} q_{n-1}) = 0$ (as $\partial b_\frac{1}{2} = \partial q_{n-1} = 0$), so $r_n$ is a cycle in $B_+ \otimes \X$. Writing $r_n$ homologous to its standard form, we have
\[
r_n = b_{\frac{1}{2}} t_n + \partial u_n,
\]
where $t_n \in \X$ is a clean polynomial, and $u_n \in \B_+ \otimes \X$. Substituting this expression into (\ref{eqn:q_n-3/2_coeffs}) gives
\[
b_{\frac{1}{2}} q_{n-1} = b_{\frac{1}{2}} x_1 t_n + x_1 \partial u_n + \partial r_{n-1}
= \partial \left( b_{\frac{3}{2}} t_n + x_1 u_n + r_{n-1} \right),
\]
so $b_{\frac{1}{2}} q_{n-1}$ is a boundary. Applying proposition \ref{lem:stronger_homology_with_a1} again we have $q_{n-1} = 0$.

Returning to equation (\ref{eqn:q_n-3/2_coeffs}), we now have
\[
\partial r_{n-1} = x_1 r_n = x_1 \left( b_{\frac{1}{2}} t_n + \partial u_n \right) = \partial \left( b_{\frac{3}{2}} t_n + x_1 u_n \right)
\]
so that $r_{n-1} + b_{\frac{3}{2}} t_n + x_1 u_n$ is a cycle, hence homologous to a standard form element
\[
r_{n-1} + b_{\frac{3}{2}} t_n + x_1 u_n = b_{\frac{1}{2}} t_{n-1} + \partial u_{n-1},
\]
where $t_{n-1} \in \X$ is a clean polynomial and $u_n \in \B_+ \otimes \X$.

We claim now inductively that $q_n = q_{n-1} = \cdots = q_{n-i+1} = 0$, for all $1 \leq i \leq n$. We also claim that each of $r_n, r_{n-1}, \ldots, r_{n-i+1}$ satisfies
\begin{align*}
r_n &= b_{\frac{1}{2}} t_n + \partial u_n \\
r_{n-1} &= b_{\frac{3}{2}} t_n + b_{\frac{1}{2}} t_{n-1} + x_1 u_n + \partial u_{n-1} \\
r_{n-2} &= b_{\frac{5}{2}} t_n + b_{\frac{3}{2}} t_{n-1} + b_{\frac{1}{2}} t_{n-2} + x_2 u_n + x_1 u_{n-1} + \partial u_{n-2} \\
\cdots \\
%r_{n-i+2} &= b_{i - \frac{3}{2}} t_n + b_{i - \frac{5}{2}} t_{n-1} + \cdots + b_{\frac{3}{2}} t_{n-i+3} + b_{\frac{1}{2}} t_{n-i+2} \\
%& \quad \quad + x_{i-2} u_n + x_{i-3} u_{n-1} + \cdots + x_2 u_{n-i+4} + x_1 u_{n-i+3} + \partial u_{n-i+2} \\
r_{n-i+1} &= b_{i - \frac{1}{2}} t_n + b_{i-\frac{3}{2}} t_{n-1} +\cdots + b_{\frac{3}{2}} t_{n-i+2} + b_{\frac{1}{2}} t_{n-i+1} \\
& \quad \quad + x_{i-1} u_n + x_{i-2} u_{n-1} + \cdots + x_2 u_{n-i+3} + x_1 u_{n-i+2} + \partial u_{n-i+1},
\end{align*}
where each of $t_n, t_{n-1}, \ldots, t_{n-i+1} \in \X$ is a clean polynomial, and each of $u_n, u_{n-1}, \ldots, u_{n-i+1} \in \B_+ \otimes \X$.
We have these claims for for $i=1,2$; now suppose they are true for $i$ with $2 \leq i \leq n$, and we show they are true for $i+1$. 

Equating coefficients of $a_{n-i-\frac{1}{2}}$ in $f = \partial r$, and noting $q_n = \cdots = q_{n-i+1} = 0$, we obtain
\[
b_{\frac{1}{2}} q_{n-i} = x_i r_n + x_{i-1} r_{n-1} + \cdots + x_2 r_{n-i+2} + x_1 r_{n-i+1} + \partial r_{n-i}.
\]
Then, as we have each of $r_n, \ldots, r_{n-i+1}$ in terms of $t$'s and $u$'s, we have
\begin{align*}
b_{\frac{1}{2}} q_{n-i} &=
x_i \left( b_{\frac{1}{2}} t_n + \partial u_n \right)
+ x_{i-1} \left( b_{\frac{3}{2}} t_n + b_{\frac{1}{2}} t_{n-1} + x_1 u_n + \partial u_{n-1} \right) \\
& \quad \quad + \cdots
%+ x_2 \Big( b_{i - \frac{3}{2}} t_n + b_{i - \frac{5}{2}} t_{n-1} + \cdots + b_{\frac{3}{2}} t_{n-i+3} + b_{\frac{1}{2}} t_{n-i+2} \\
%& \quad \quad + x_{i-2} u_n + x_{i-3} u_{n-1} + \cdots + x_2 u_{n-i+4} + x_1 u_{n-i+3} + \partial u_{n-i+2} \Big) \\
%& \quad \quad 
+ x_1 \Big( b_{i - \frac{1}{2}} t_n + b_{i-\frac{3}{2}} t_{n-1} +\cdots + b_{\frac{3}{2}} t_{n-i+2} + b_{\frac{1}{2}} t_{n-i+1} \\
& \quad \quad + x_{i-1} u_n + x_{i-2} u_{n-1} + \cdots + x_2 u_{n-i+3} + x_1 u_{n-i+2} + \partial u_{n-i+1} \Big) + \partial r_{n-i}.
 %\\
\end{align*}
We may regroup according to the $t_n$ terms:
\begin{align*}
b_{\frac{1}{2}} q_{n-i} &= \left( b_{\frac{1}{2}} x_i + b_{\frac{3}{2}} x_{i-1} + \cdots + b_{i-\frac{3}{2}} x_2 + b_{i-\frac{1}{2}} x_1 \right) t_n
+ \left( b_{\frac{1}{2}} x_{i-1} + b_{\frac{3}{2}} x_{i-2} + \cdots + b_{i-\frac{5}{2}} x_2 + b_{i-\frac{3}{2}} x_1 \right) t_{n-1} \\
& \quad \quad + \cdots
+ \left( b_{\frac{1}{2}} x_2 + b_{\frac{3}{2}} x_1 \right) t_{n-i+2}
+ \left( b_{\frac{1}{2}} x_1 \right) t_{n-i+1} \\
& \quad \quad
+ \left( x_{i-1} x_1 + x_{i-2} x_2 + \cdots + x_2 x_{i-2} + x_1 x_{i-1} \right) u_n
+ \left( x_{i-2} x_1 + \cdots + x_1 x_{i-2} \right) u_{n-1} \\
& \quad \quad + \cdots
+ \left( x_2 x_1 + x_1 x_2 \right) u_{n-i+3}
+ \left( x_1 x_1 \right) u_{n-i+2} \\
& \quad \quad + x_i \partial u_n + x_{i-1} \partial u_{n-1} + \cdots + x_2 \partial u_{n-i+2} + x_1 \partial u_{n-i+1} + \partial r_{n-i} %\\
\end{align*}
We now recognise this as a boundary:
\begin{align*}
b_{\frac{1}{2}} q_{n-i}
%&= \partial \left( b_{i+\frac{1}{2}} t_n + b_{i-\frac{1}{2}} t_{n-1} + \cdots + b_{\frac{5}{2}} t_{n-i+2} + b_{\frac{3}{2}} t_{n-i+1} \right) + \partial x_i u_n + \partial x_{i-1} u_{n-1} + \cdots + \partial x_3 u_{n-i+3} + \partial x_2 u_{n-i+2} \\
%& \quad \quad + x_i \partial u_n + x_{i-1} \partial u_{n-1} + \cdots + x_2 \partial u_{n-i+2} + x_1 \partial u_{n-i+1} + \partial r_{n-i} \\
&= \partial \left( b_{i+ \frac{1}{2}} t_n + b_{i-\frac{1}{2}} t_{n-1} + \cdots + b_{\frac{3}{2}} t_{n-i+1} \right. \\
&\quad \quad \left. + x_i u_n + x_{i-1} u_{n-1} + \cdots + x_3 u_{n-i+3} + x_2 u_{n-i+2} + x_1 u_{n-i+1} + r_{n-i} \right).
\end{align*}
Thus $q_{n-i} = 0$. Moreover we obtain a cycle, whose homology class has a standard form:
\[
b_{i+\frac{1}{2}} t_n + b_{i-\frac{1}{2}} t_{n-1} + \cdots + b_{\frac{3}{2}} t_{n-i+1} + x_i u_n + x_{i-1} u_{n-1} + \cdots + x_2 u_{n-i+2} + x_1 u_{n-i+1} + r_{n-i} = b_{\frac{1}{2}} t_{n-i} + \partial u_{n-i},
\]
for some clean polynomial $t_{n-i} \in \X$ and some $u_{n-i} \in \B_+ \otimes \X$. Rearranging this gives
\[
r_{n-i} = b_{i+\frac{1}{2}} t_n + \cdots + + b_{\frac{3}{2}} t_{n-i+1} + b_{\frac{1}{2}} t_{n-i} + x_i u_n + \cdots + x_1 u_{n-i+1} + \partial u_{n-i}.
\]
Thus the claims are proved for $i+1$.; by induction it is true for all $i$ up to $n$. With $i=n$ then we have $q_n = q_{n-1} = \cdots = q_1 =0$ as desired. 
\end{proof}

Now we can write down $H(\A_+^{<n} \otimes \X \otimes \B_+)$. By lemma \ref{lem:standard_form_in_truncated_complex}, any cycle $f$ in $\A_+^{< n} \otimes \X \otimes \B_+$ is homologous to one of the form $s_1 q_1 + \cdots s_n q_n$, where $q_i$ are clean polynomials. Conversely, as $\partial s_i = \partial q_i = 0$ every element of the form $s_1 q_1 + \cdots + s_n q_n$ is a cycle. Moreover, proposition \ref{prop:truncated_injectivity} says that any boundary of the form $s_1 q_1 + \cdots + s_n q_n$ must be zero. Writing $\bar{s}_i \bar{q}$ for the homology classes of each $s_i q$, we have obtained the following.
\begin{prop}
\label{prop:truncated_homology}
Let $n$ be a positive integer. As a $\Z_2$-module, $H(\A_+^{<n} \otimes \X \otimes \B_+)$ is freely generated by the elements $\bar{s}_i \bar{q}$, over all integers $i$ satisfying $1 \leq i \leq n$, and all clean monomials $\bar{q}$ in $H(\X)$.
\qed
\end{prop}

\subsection{Homology of insular string diagrams III: direct limit and module structure}
\label{sec:homology_of_++_computed}

We have now down the hard work in computing the homology of $\A_+ \otimes \X \otimes \B_+$, the subcomplex of $\widehat{CS}(A, F)$ consisting of insular string diagrams. For the chain complex $\A_+ \otimes \X \otimes \B_+$ is the direct limit of the $\A_+^{<n} \otimes \X \otimes \B_+$, and direct limits commute with homology. Thus $H(\A_+ \otimes \X \otimes \B_+)$ is the direct limit of the $H(\A_+^{<n} \otimes \X \otimes \B_+)$. From proposition \ref{prop:truncated_homology}, we immediately obtain the following.

\begin{prop}
As a $\Z_2$-module, $H(\A_+ \otimes \X \otimes \B_+)$ is freely generated by the elements $\bar{s}_i \bar{q}$, over all positive integers $i$ and all clean monomials $\bar{q}$.
\qed
\end{prop}

Thus, as $\Z_2$-modules at least,
\begin{equation}
\label{eqn:ABX_homology_as_free_module}
H(\A_+ \otimes \X \otimes \B_+) = \frac{ \Z[\ldots, \bar{x}_{-3}, \bar{x}_{-1}, \bar{x}_3, \bar{x}_5, \ldots] }{ (\ldots, \bar{x}_{-3}^2, \bar{x}_{-1}^2, \bar{x}_3^2, \bar{x}_5^2, \ldots) } \langle s_1, s_2, \ldots \rangle
= H(\X)_{\neq 1} \langle s_1, s_2, \ldots \rangle.
\end{equation}
Note the absence of $\bar{x}_1$; recall definition \ref{defn:polynomial_subrings} of $H(\X)_{\neq 1}$.

We would like to explain the module structure in $H(\A_+ \otimes \X \otimes \B_+)$, as well as the anomalous behaviour of $\bar{x}_1$. We have seen that $\A_+ \otimes \X \otimes \B_+$ is a differential $\X$-module, so $H(\A_+ \otimes \X \otimes \B_+)$ is an $H(\X)$-module; and we computed $H(\X)$ in theorem \ref{thm:homology_of_X} as
\[
\frac{ \Z_2 [ \ldots, \bar{x}_{-3}, \bar{x}_{-1}, \bar{x}_1, \bar{x}_3, \ldots ] }{ ( \ldots, \bar{x}_{-3}^2, \bar{x}_{-1}^2, \bar{x}_1^2, \bar{x}_3^2, \ldots ) },
\]
Now the $H(\X)$-module structure on $H(\A_+ \otimes \X \otimes \B_+)$ is inherited from the $\X$-module structure on $\A_+ \otimes \X \otimes \B_+$; multiplication by $x_j$ on $\A_+ \otimes \X \otimes \B_+$ becomes multiplication by $\bar{x}_j$ in $H(\A_+ \otimes \X \otimes \B_+)$. The multiplication by each $\bar{x}_j$, for $j$ odd and $j \neq 1$, is clear enough, since multiplication by $x_j$ sends each clean monomial either to another clean monomial, or to a monomial with a $x_j^2$ factor, which becomes zero in $H(\X)$. For $j$ even, $x_j$ does not appear in homology, so there is no $\bar{x}_j$ by which to multiply! 

As $H(\X)_{\neq 1}$ is a subring of $H(\X)$, $H(\A_+ \otimes \X \otimes \B_+)$ has the structure of a $H(\X)_{\neq 1}$-module. In fact we have now shown it is a free $H(\X)_{\neq 1}$-module with basis $\{s_i\}_{i=1}^\infty$, as equation (\ref{eqn:ABX_homology_as_free_module}) suggests.

To understand the $H(\X)$-module structure, it remains only to understand the action of multiplication by $\bar{x}_1$; we claim this action is as follows on the $\bar{s}_i$.
\begin{align*}
\bar{x}_1 \bar{s}_1 = 0, \quad \bar{x}_1 \bar{s}_2 = 0, \quad \bar{x}_1 \bar{s}_3 = \bar{x}_3 \bar{s}_1, \quad \bar{x}_1 \bar{s}_4 = \bar{x}_3 \bar{s}_2, \\
\bar{x}_1 \bar{s}_5 = \bar{x}_3 \bar{s}_3 + \bar{x}_5 \bar{s}_1, \quad
\bar{x}_1 \bar{s}_6 = \bar{x_3} \bar{s}_4 + \bar{x}_5 \bar{s}_2, \quad \ldots
\end{align*}
In general, the pattern continues as specified in the following lemma.
\begin{lem}
In $H(\A_+ \otimes \X \otimes \B_+)$ we have
\begin{align*}
\bar{x}_1 \bar{s}_n &= \bar{x}_3 \bar{s}_{n-2} + \bar{x}_5 \bar{s}_{n-4} + \cdots %\\
= \sum_{\substack{j \geq 3 \text{ odd}, \\ j+k=n+1 }} \bar{x}_j \bar{s}_k
= \sum_{k=1}^{\lfloor \frac{n-1}{2} \rfloor} \bar{x}_{2k+1} \bar{s}_{n-2k}.
\end{align*}
\end{lem}
The last two equalities above are just ways of rewriting the sum. These sums are linear combinations of clean monomials times $\bar{s}_i$, so are in standard form.

\begin{proof}
Consider the element $h \in \A_+ \otimes \X \otimes \B_+$ consisting of every second term in $s_{n+1}$ as shown:
\[
h = a_{\frac{1}{2}} b_{n+\frac{1}{2}} + a_{\frac{5}{2}} b_{n-\frac{3}{2}} + a_{\frac{9}{2}} b_{n-\frac{7}{2}} + \cdots = \sum_{\text{pos.}} a_{\frac{1}{2}+2m} b_{n+\frac{1}{2}-2m}.
\]
This $h$ consists of every second term in $s_{n+1}$; the ``pos" in the sum indicates to sum over integers $m$ such that the indices are positive, i.e.$\frac{1}{2}+2m >0$ and $n+\frac{1}{2}-2m >0$. 

When we take $\partial$ of $h$, we obtain a sum of terms of the form $a_{i+\frac{1}{2}} b_{j+\frac{1}{2}} x_k$, where $i,j,k$ are positive integers, and $i+j+k=n$. 

Now for each pair $i,j$ of positive integers with $i+j \leq n-1$, the term $a_{i+\frac{1}{2}} b_{j+\frac{1}{2}} x_{n-i-j}$ appears in the differential of two terms of $s_{n+1}$, namely $a_{i+\frac{1}{2}} b_{n+\frac{1}{2}-i}$ and $a_{n+\frac{1}{2}-j} b_j$. These two terms may or may not appear in $h$. However, if $(i+\frac{1}{2})-(n+\frac{1}{2}-j)=i+j-n$ is even, then they either both appear, or both do not appear, in $h$. And if $i+j-n$ is odd, then precisely one of them appears. Since $i+j-n \equiv k$ mod $2$, we see that $\partial h$ is precisely a sum of these $a_{i+\frac{1}{2}} b_{j+\frac{1}{2}} x_k$ where $k$ is odd, and $i,j,k$ are positive integers, with $i+j+k=n$. These are precisely the terms appearing in $x_1 s_n + x_3 s_{n-2} + x_5 s_{n-4} + \cdots$. We conclude that
\[
\partial h = x_1 s_n + x_3 s_{n-2} + x_5 s_{n-4} + \cdots,
\]
giving the desired result upon passing to homology.
\end{proof}

We now have a complete description of $H(\A_+ \otimes \X \otimes \B_+)$. To summarise:
\begin{thm}
\label{thm:++_homology}
The homology $H(\A_+ \otimes \X \otimes \B_+)$, is:
\begin{enumerate}
\item
a free $\Z_2$-module with basis $\bar{s}_i \bar{q}$, over all positive integers $i$ and clean monomials $\bar{q}$;
\item
a free $H(\X)_{\neq 1}$-module with basis $\bar{s}_i$, over all positive integers $i$;
\item
an $H(\X)$-module generated by the elements $\bar{s}_i$, over all positive integers $i$, where for any odd integer $j \neq 1$, $\bar{x}_j$ acts by polynomial multiplication, and $\bar{x}_1$ acts by
\[
\bar{x}_1 \bar{s}_n = \bar{x}_3 \bar{s}_{n-2} + \bar{x}_5 \bar{s}_{n-4} + \cdots.
\]
\end{enumerate}
\end{thm}

\subsection{Homology of insular string diagrams IV: completing the calculation}
\label{sec:completing_insular_calculation}

We have now computed the homology of $\A_+ \otimes \X \otimes \B_+$. But recall from section \ref{sec:ABX_homology} that this is just one of four summands of $\A \otimes \X \otimes \B$:
\[
\A \otimes \X \otimes \B
\cong
\left( \A_+ \otimes \X \otimes \B_+ \right)
\oplus
\left( \A_+ \otimes \X \otimes \B_- \right)
\oplus
\left( \A_- \otimes \X \otimes \B_+ \right)
\oplus
\left( \A_- \otimes \X \otimes \B_- \right),
\]
where $\A_\pm, \B_\pm$ are freely generated by the $a_i, b_j$ with $i,j$ positive and negative respectively.

After dealing with $\A_+ \otimes \X \otimes \B_+$ the other three summands are easier. In fact, the homology of $\A_- \otimes \X \otimes \B_-$ is now immediately isomorphic to $\A_+ \otimes \X \otimes \B_+$.
\begin{prop}
\label{prop:++_--_isomorphism}
The map $\iota: \A \otimes \X \otimes \B \To \A \otimes \X \otimes \B$ defined by $a_i \mapsto a_{-i}$, $b_j \mapsto b_{-j}$, $x_k \mapsto x_{-k}$ and extended linearly, gives isomorphisms of chain complexes
\[
\A_+ \otimes \X \otimes \B_+ \cong \A_- \otimes \X \otimes \B_-, \quad
\A_+ \otimes \X \otimes \B_- \cong \A_- \otimes \X \otimes \B_+.
\]
\end{prop}

\begin{proof}
It is clear that $\iota$ is an isomorphism of $\Z_2$-modules, and is an involution sending $\A_+ \otimes \X \otimes \B_+ \leftrightarrow \A_- \otimes \X \otimes \B_-$ and $\A_+ \otimes \X \otimes \B_- \leftrightarrow \A_- \otimes \X \otimes \B_+$; we check it commutes with $\partial$. 

We first consider $\iota$ on $\A_+$: explicitly, for a non-negative integer $n$,
\[
\iota \partial a_{n + \frac{1}{2}} = \iota \sum_{i,j>0, i+j = n} a_{i + \frac{1}{2}} x_j 
= \sum_{i,j>0, i+j = n} a_{-i - \frac{1}{2}} x_{-j}
= \sum_{i,j<0, i+j = -n} a_{i-\frac{1}{2}} x_j
= \partial a_{-n-\frac{1}{2}} = \partial \iota a_{n+\frac{1}{2}}.
\]
By a similar calculation we have $\iota \partial = \partial \iota$ on $\A_-, \B_+, \B_-$. We also have, by a similar argument, $\iota \partial x_k = \partial \iota x^e$ for monomials $x^e \in \X$. By the Leibniz rule then $\iota$ commutes with $\partial$ on $\A \otimes \X \otimes \B$. Thus $\iota$ gives an involution on homology which induces the desired isomorphisms.
\end{proof}

Thus, our description of $H(\A_+ \otimes \X \otimes \B_+)$ in theorem \ref{thm:++_homology} is also a description of $H(\A_- \otimes \X \otimes \B_-)$, upon exchanging each $a_i, b_j, x_k$ with $a_{-i}, b_{-j}, x_{-k}$. It is a free $\Z_2$-module with basis $\bar{s}_i \bar{q}$, over all negative integers $i$ and \emph{negatively clean} monomials $q$. It is also a free $H(\X)_{\neq -1}$-module (recall definition \ref{defn:polynomial_subrings}) with basis $\bar{s}_i$, over all negative integers $i$. And it is finally a $H(\X)$-module generated by the elements $\bar{s}_i$, over all negative integers $i$, where for any odd negative integer $j \leq -1$, $\bar{x}_j$ acts by polynomial multiplication, and $\bar{x}_{-1}$ acts on $\bar{s}_{-n}$, for $-n<0$, by $\bar{x}_{-1} \bar{s}_{-n} = \bar{x}_{-3} \bar{s}_{-n+2} + \bar{x}_{-5} \bar{s}_{-n+4} + \cdots$.

It remains to consider $\A_+ \otimes \X \otimes \B_-$; from above, $\A_- \otimes \X \otimes \B_+$ is similar. This complex behaves more simply than $\A_+ \otimes \X \otimes \B_+$. In particular, the presence of both positive and negative indices in $a_i b_j$ allows us to simplify cycles into a considerably more straightforward standard form, more like $\A_+ \otimes \X$ (section \ref{sec:homology_A_+_tensor_X}) than $\A_+ \otimes \X \otimes \B_+$. As in previous computations, the first and main step is a technical lemma (similar to lemma \ref{lem:a_simplification}) which, given a cycle in $\A_+ \otimes \X \otimes \B_-$, reduces it modulo a boundary to one of lower $a$-degree. (Note that definition \ref{def:a-degree} of $a$-degree applies to $\A_+ \otimes \X \otimes \B_-$.)

We start with a an element $f \in \A_+ \otimes \X \otimes \B_-$ of $a$-degree $n-\frac{1}{2}$, so
\[
f = a_{\frac{1}{2}} p_\frac{1}{2} + a_{\frac{3}{2}} p_\frac{3}{2} + \cdots + a_{n - \frac{1}{2}} p_{n-\frac{1}{2}}
= \sum_{i=1}^n a_{i - \frac{1}{2}} p_{i-\frac{1}{2}},
\]
where each $p_1, \ldots, p_n \in \B_- \otimes \X$. The reduction is as follows.

\begin{lem}
If $f \in \A_+ \otimes \X \otimes \B_-$ has $a$-degree $n-\frac{1}{2} \geq \frac{3}{2}$ and satisfies $\partial f = 0$, then $f = \partial g + O(a_{n-\frac{3}{2}})$ for some $g \in \A_+ \otimes \X \otimes \B_-$.
\end{lem}

\begin{proof}
Let $f = a_{n-\frac{1}{2}} p_{n-\frac{1}{2}} + a_{n-\frac{3}{2}} p_{n-\frac{3}{2}} + O(a_{n-\frac{5}{2}})$, where $p_{n-\frac{1}{2}}, p_{n-\frac{3}{2}} \in \B_- \otimes \X$. Then
\begin{align*}
0 = \partial f
%&= \left( \partial a_{n - \frac{1}{2}} \right) p_n + a_{n - \frac{1}{2}} \left( \partial p_n \right) + \left( \partial a_{n - \frac{3}{2}} \right) p_{n-1} + a_{n - \frac{3}{2}} \left( \partial p_{n-1} \right) + O(a_{n - \frac{5}{2}}) \\
%&= \left( a_{n - \frac{3}{2}} x_1 + O(a_{n - \frac{5}{2}}) \right) p_n + a_{n - \frac{1}{2}} \partial p_n + O(a_{n - \frac{5}{2}}) p_{n-1} + a_{n - \frac{3}{2}} \partial p_{n-1} + O(a_{n - \frac{5}{2}}) \\
&= a_{n - \frac{1}{2}} \partial p_{n-\frac{1}{2}} + a_{n - \frac{3}{2}} \left( x_1 p_{n-\frac{1}{2}} + \partial p_{n-\frac{3}{2}} \right) + O(a_{n - \frac{5}{2}})
\end{align*}
Equating coefficients of $a_{n - \frac{1}{2}}$ and $a_{n - \frac{3}{2}}$ gives
\[
\partial p_{n-\frac{1}{2}} = 0,
\quad \quad
x_1 p_{n-\frac{1}{2}} = \partial p_{n-\frac{3}{2}}.
\]
Thus $p_{n-\frac{1}{2}}$ is a cycle in $\B_- \otimes \X$ and $x_1 p_{n-\frac{1}{2}}$ is a boundary. As remarked in section \ref{sec:insular_differential}, $\B \otimes \X \cong \A \otimes \X$ as differential $\X$-modules, and indeed $\B_- \otimes \X \cong \A_+ \otimes \X$ under the isomorphism $b_i x_k \leftrightarrow a_{-i} x_{-k}$. So by proposition \ref{prop:strong_homology_with_a1} applied to $\B_- \otimes \X$, $\partial p_{n-\frac{1}{2}} = 0$ implies
\[
p_{n-\frac{1}{2}} = b_{-\frac{1}{2}} q + \partial r
\]
where $q \in \X$ is a negatively clean polynomial and $r \in \B_- \otimes \X$.

We claim that $q$ is divisible by $x_1$. To see this, split $q$ into terms which do and do not contain $x_1$, i.e. $q = t + x_1 u$ where $t,u$ are \emph{totally clean} polynomials (definition \ref{def:clean}). Then we have
\[
x_1 p_{n-\frac{1}{2}} = b_{-\frac{1}{2}} x_1 \left( t + x_1 u \right) + x_1 \partial r = b_{-\frac{1}{2}} x_1 t + \partial \left( b_{-\frac{1}{2}} x_2 u + x_1 r \right),
\]
expressing $x_1 p_{n-\frac{1}{2}}$ in standard form, which by by proposition \ref{lem:stronger_homology_with_a1} is unique. As $x_1 p_{n-\frac{1}{2}}$ is a boundary we must have $b_{-\frac{1}{2}} x_1 t = 0$, so $t=0$. Hence $q = x_1 u$ and $q$ is indeed divisible by $x_1$.

From $q = x_1 u$, we now have $p_{n-\frac{1}{2}} = b_{-\frac{1}{2}} x_1 u + \partial r$. This gives our original $f$ as
\[
f = a_{n - \frac{1}{2}} \left( b_{-\frac{1}{2}} x_1 u + \partial r \right) + O(a_{n - \frac{3}{2}}) = a_{n - \frac{1}{2}} b_{-\frac{1}{2}} x_1 u + a_{n - \frac{1}{2}} \partial r + O \left( a_{n - \frac{3}{2}} \right).
\]
Now note $\partial \left( a_{n + \frac{1}{2}} b_{-\frac{1}{2}} u \right) = a_{n - \frac{1}{2}} b_{-\frac{1}{2}} x_1 u + O(a_{n - \frac{3}{2}})$ (here we used $\partial b_{-\frac{1}{2}} = \partial u = 0$) and $\partial \left( a_{n - \frac{1}{2}} r \right) = a_{n - \frac{1}{2}} \partial r + O(a_{n - \frac{3}{2}})$. Thus we have
\[
f = \partial \left( a_{n + \frac{1}{2}} b_{-\frac{1}{2}} u + a_{n - \frac{1}{2}} r \right) + O \left( a_{n- \frac{3}{2}} \right),
\]
giving the desired form for $f$, with $g = a_{n+\frac{1}{2}} b_{-\frac{1}{2}} u + a_{n-\frac{1}{2}} r$.
\end{proof}

Repeated use of this result allows us to reduce any cycle $f \in \A_+ \otimes \X \otimes \B_-$ to one of $a$-degree $\frac{1}{2}$, so
\[
f = a_{\frac{1}{2}} p' + \partial g',
\]
where $p' \in \B_- \otimes \X$ and $g' \in \A_+ \otimes \X \otimes \B_-$. We can then simplify further, ``reducing" (actually, increasing) the $b$-degree of $p'$. Since $f$ is a cycle, $0 = \partial f = a_{\frac{1}{2}} \partial p'$, so $\partial p' = 0$. But now $p' \in \B_- \otimes \X$ is a cycle, and proposition \ref{lem:stronger_homology_with_a1} puts $p'$ in a standard form,
\[
p' = b_{-\frac{1}{2}} q + \partial g'',
\]
where $q$ is a negatively clean polynomial, and $g'' \in \B_- \otimes \X$. We then have $f$ in the form
\[
f = a_{\frac{1}{2}} b_{-\frac{1}{2}} q + a_{\frac{1}{2}} \partial g'' + \partial g'
%= a_{\frac{1}{2}} b_{-\frac{1}{2}} q + \partial \left( a_{\frac{1}{2}} g' + g_1 \right)
= a_{\frac{1}{2}} b_{-\frac{1}{2}} q + \partial g''',
\]
where $g''' = a_{\frac{1}{2}} g'' + g' \in \A_+ \otimes \X \otimes \B_-$.
The following proposition further improves $f$, removing factors of $x_{1}$ and providing a standard form for cycles in $\A_+ \otimes \X \otimes \B_-$.
\begin{prop}
\label{prop:+-_standard_form}
Suppose $f \in \A_+ \otimes \X \otimes \B_-$ satisfies $\partial f = 0$. Then
\[
f = a_{\frac{1}{2}} b_{-\frac{1}{2}} p + \partial g,
\]
where $p \in \X$ is a totally clean polynomial, and $g \in \A_+ \otimes \X \otimes \B_-$.
\end{prop}

\begin{proof}
From above, we have $f = a_{\frac{1}{2}} b_{-\frac{1}{2}} q + \partial g'''$, where $q \in \X$ is negatively clean and $g''' \in \A_+ \otimes \X \otimes \B_-$. We may separate the terms of $q$ which do and do not contain $x_1$, to write $q = p + x_1 u$, where $p,u$ are totally clean. We then have
\[
f = a_{\frac{1}{2}} b_{-\frac{1}{2}} p + a_{\frac{1}{2}} b_{-\frac{1}{2}} x_1 u + \partial g'''
= a_{\frac{1}{2}} b_{-\frac{1}{2}} p + \partial \left( a_{\frac{3}{2}} b_{-\frac{1}{2}} u + g''' \right)
\]
so, taking $g = a_{\frac{3}{2}} b_{-\frac{1}{2}} u + g'''$, we have the desired result.
\end{proof}

Thus, every cycle in $\A_+ \otimes \X \otimes \B_-$ is homologous to the standard form $a_{\frac{1}{2}} b_{-\frac{1}{2}} p$, with $p$ totally clean. We can easily check that each such $a_{\frac{1}{2}} b_{-\frac{1}{2}} p$ is a cycle. Thus every homology class in $H(\A_+ \otimes \X \otimes \B_-)$ has a representative of the form $a_{\frac{1}{2}} b_{-\frac{1}{2}} p$, and every $a_{\frac{1}{2}} b_{-\frac{1}{2}} p$ represents some homology class. We will show that these representatives are unique. We will do this by use of a map to the simpler chain complex $\X$.

\begin{defn}
The map $\Psi: \A_+ \otimes \X \otimes \B_- \To \X$ is defined by
\[
a_{i - \frac{1}{2}} b_{-j+\frac{1}{2}} p \mapsto x_i p x_{-j},
\]
for positive integers $i,j$ and $p \in \X$, and extended by $\Z_2$-linearity.
\end{defn}

Geometrically, $\Psi$ corresponds to gluing annuli to each boundary of the annulus $(\An, F_{2,2})$ with string diagrams closing off $a_{i - \frac{1}{2}}$ and $b_{-j+\frac{1}{2}}$ into closed curves $x_i$, $x_{-j}$ respectively. This gives a map $\widehat{CS}(\An, F_{2,2}) \To \widehat{CS}(\An, \emptyset)$. Since there are no crossings in the glued-on annuli, $\Psi$ is a chain map. We can also prove the result we need on the subcomplex $\A_+ \otimes \X \otimes \B_-$ purely algebraically.
\begin{lem}
The map $\Psi$ is a chain map: $\Psi \partial = \partial \Psi$.
\end{lem}

\begin{proof}
It suffices to show that the result holds on generators $a_{i - \frac{1}{2}} b_{-j+\frac{1}{2}} p$, where $i,j$ are positive integers and $p \in \X$. We have
\begin{align*}
\Psi \partial \left( a_{i - \frac{1}{2}} b_{-j+\frac{1}{2}} p \right)
&= \Psi \left[ \left( \partial a_{i- \frac{1}{2}} \right) b_{-j+\frac{1}{2}} p + a_{i - \frac{1}{2}} \left( \partial b_{-j+\frac{1}{2}} \right) p + a_{i - \frac{1}{2}} b_{-j+\frac{1}{2}} \left( \partial p \right) \right] \\
&= \Psi \left[ \left( \sum_{k=1}^{i-1} a_{k-\frac{1}{2}} x_{i-k} \right) b_{-j+\frac{1}{2}} p + a_{i - \frac{1}{2}} \left( \sum_{k=1}^{j-1} b_{-k + \frac{1}{2}} x_{-j+k} \right) p + a_{i - \frac{1}{2}} b_{-j+\frac{1}{2}} \left( \partial p \right) \right] \\
\\
&= \left( \sum_{k=1}^{i-1} x_k x_{i-k} \right) x_{-j} p + a_{i - \frac{1}{2}} \left( \sum_{k=1}^{j-1} x_{-k} x_{-j+k} \right) p + a_{i - \frac{1}{2}} x_{-j} \left( \partial p \right) \\
&= \left( \partial x_i \right) x_{-j} p + a_{i - \frac{1}{2}} \left( \partial x_{-j} \right) p + a_{i - \frac{1}{2}} x_{-j} \left( \partial p \right) \\
&=  \partial \left( x_i x_{-j} p \right) 
= \partial \Psi \left( a_{i - \frac{1}{2}} b_{-j+\frac{1}{2}} p \right).
\end{align*}
\end{proof}

Thus, $\Psi$ gives a map on homology, which by abuse of notation we also call $\Psi$.
\begin{prop}
The map $\Psi: H(\A_+ \otimes \X \otimes \B_-) \To H(\X)$ is injective and has image $\bar{x}_1 \bar{x}_{-1} H(\X)$.
\end{prop}

\begin{proof}
By proposition \ref{prop:+-_standard_form}, a nonzero homology class in $H(\A_+ \otimes \X \otimes \B_-)$ has a representative of the form $a_{\frac{1}{2}} b_{-\frac{1}{2}} p$, where $p \neq 0$ is totally clean. 

Under $\Psi$ this homology class $\bar{a}_\frac{1}{2} \bar{b}_{-\frac{1}{2}} \bar{p}$ maps to $\bar{x}_1 \bar{x}_{-1} \bar{p}$, which is a fermionic polynomial, hence nonzero in $H(\X)$. Thus $\Psi$ is injective. Moreover any element $\bar{x}_1 \bar{x}_{-1} \bar{p} \in \bar{x}_1 \bar{x}_{-1} H(\X)$ is the image of $\bar{a}_{\frac{1}{2}} \bar{b}_{-\frac{1}{2}} \bar{p}$.
\end{proof}

In fact, $\Psi$ is actually an $\X$-module homomorphism and gives an $H(\X)$-module homomorphism on homology. Thus we obtain an an explicit description of the homology as an $H(\X)$-module. 
\begin{thm}
\label{thm:+-_homology}
The homology $H(\A_+ \otimes \X \otimes \B_-)$ is isomorphic to $\bar{x}_{-1} \bar{x}_1 H(\X)$ as an $H(\X)$-module via $\Psi$. Every homology class has a unique representative of the form $a_{\frac{1}{2}} b_{-\frac{1}{2}} p$, where $p \in \X$ is totally clean.
\qed
\end{thm}
Thus $H(\A_+ \otimes \X \otimes \B_-)$ is a free $H(\X)_{-1,1}$-module (definition \ref{defn:polynomial_subrings}) with basis $1$ (isomorphic to $H(\X)_{-1,1}$ as a module over itself). As a $\Z_2$-module, $H(\A_+ \otimes \X \otimes \B_-)$ is free with basis $\bar{a}_{\frac{1}{2}} \bar{b}_{-\frac{1}{2}} \bar{q}$, over all totally clean monomials $\bar{q}$. 

To summarise:
\[
H(\A_+ \otimes \X \otimes \B_-) = \bar{a}_{\frac{1}{2}} \bar{b}_{-\frac{1}{2}} H(\X)_{\neq -1,1} \cong H(\X)_{\neq -1,1} \cong \bar{x}_{-1} \bar{x}_1 H(\X) \cong 
\bar{x}_1 \bar{x}_{-1} \frac{ \Z[\ldots, x_{-3}, x_{-1}, x_1, x_3, \ldots] }{ \left( \ldots, x_{-3}^2, x_{-1}^2, x_1^2, x_3^2, \ldots \right) }
\]

Using the isomorphism $\A_+ \otimes \X \otimes \B_- \cong \A_- \otimes \X \otimes \B_+$ we immediately also have the homology of $\A_- \otimes \X \otimes \B_+$: it is isomorphic to $\bar{x}_{-1} \bar{x}_1 H(\X)$ as an $H(\X)$-module; every homology class has a unique representative $a_{-\frac{1}{2}} b_{\frac{1}{2}} p$, where $p \in \X$ is totally clean; and as a $\Z_2$-module it is free with basis $\bar{a}_{-\frac{1}{2}} b_{\frac{1}{2}} \bar{q}$ over all totally clean monomials $\bar{q}$.

\subsection{Full homology}
\label{sec:towards_full_22_homology}

Let us now return to the sutured background $(A,F_{2,2})$, and recall that
\[
\widehat{CS}(A,F) \cong \left( \A \otimes \X \otimes \B \right) \oplus \left( \C \otimes \X \otimes \D \right)
\]
where $\A \otimes \X \otimes \B$, describing insular string diagrams, is a subcomplex; while $\C \otimes \X \otimes \D$, describing traversing string diagrams, is not. We have the further decomposition into subcomplexes
\[
\A \otimes \X \otimes \B \cong
\left( \A_+ \otimes \X \otimes \B_+ \right)
\oplus
\left( \A_+ \otimes \X \otimes \B_- \right)
\oplus
\left( \A_- \otimes \X \otimes \B_+ \right)
\oplus
\left( \A_- \otimes \X \otimes \B_- \right).
\]
After sections \ref{sec:ABX_homology} to \ref{sec:completing_insular_calculation}, specifically theorems \ref{thm:++_homology} and \ref{thm:+-_homology}, we know the homology of these complexes explicitly. Over $\Z_2$ they have bases, respectively
\begin{enumerate}
\item
$H(\A_+ \otimes \X \otimes \B_+)$: basis $\bar{s}_i \bar{q}$, over all positive integers $i$ and positively clean monomials $\bar{q}$;
\item
$H(\A_+ \otimes \X \otimes \B_-)$: basis $\bar{a}_{\frac{1}{2}} \bar{b}_{-\frac{1}{2}} \bar{q}$, over all totally clean monomials $\bar{q}$;
\item
$H(\A_- \otimes \X \otimes \B_+)$: basis $\bar{a}_{-\frac{1}{2}} \bar{b}_{\frac{1}{2}} \bar{q}$, over all totally clean monomials $\bar{q}$.
\item
$H(\A_- \otimes \X \otimes \B_-)$: basis $\bar{s}_i \bar{q}$, over all negative integers $i$ and negatively clean monomials $\bar{q}$.
\end{enumerate}
Over $H(\X)$, these modules are non-free, with rank $\infty, 1, 1, \infty$ respectively. We have now proved parts (i) and (ii) of theorem \ref{thm:22_annulus_description}.

Unfortunately, the situation becomes more difficult when we extend to the full chain complex $(\A \otimes \X \otimes \B) \oplus (\C \otimes \X \otimes \D)$. We are unable to give a complete description of its homology. The differential on $\C \otimes \X \otimes \D$, as we saw in section \ref{sec:traversing_differential}, does not obey a Leibniz rule, and the differential of a term $c_i d_j x^e$ will in general include terms of both types $c_i d_j x^e$ and $a_i b_j x^e$. So the full homology $\widehat{HS}(\An, F_{2,2})$ has the structure of a $\Z_2$-module, but not an $H(\X)$-module. 

We will however give some partial results regarding the full complex.

First, the homology $H(\A \otimes \X \otimes \B)$ does not escape unscathed from the effect of the differential on $\C \otimes \X \otimes \D$: for instance, for any integer $n \neq 0$, the element $s_n \in \A \otimes \X \otimes \B$ has nonzero homology class $\bar{s}_n \in H(\A \otimes \X \otimes \B)$, but is a boundary in $\widehat{HS}(A, F_{2,2})$ since $\partial c_0 d_n = s_n$. There are certainly however elements of $H(\A \otimes \X \otimes \B)$ that do survive in $\widehat{HS}(A, F_{2,2})$; and there are cycles in $\A \otimes \X \otimes \B$ not homologous to any elements of $\C \otimes \X \otimes \D$ . Indeed, as the differential $\partial c_m d_n x^e$ of a generator of $\C \otimes \X \otimes \D$ only involves terms of the form $c_{m'} d_{n'} x^{e'}$ and $a_i b_j x^e$ with $i,j$ of the same sign, we have a decomposition
\[
\widehat{CS}(\An, F_{2,2}) \cong \left[ \A_+ \otimes \X \otimes \B_- \right] \oplus \left[ \A_- \otimes \X \otimes \B_+ \right] \oplus \left[ \X \otimes \left( \left( \A_+ \otimes \B_+ \right) \oplus \left( \A_- \otimes \B_- \right) \oplus \left( \C \otimes \D \right) \right) \right].
\]
into three chain complexes, the first two of which are differential $\X$-modules. Thus $H(\A_+ \otimes \X \otimes \B_-)$ and $H(\A_- \otimes \X \otimes \B_+)$ are summands of $\widehat{HS}(\An, F_{2,2})$. In particular, $\bar{a}_{\frac{1}{2}} \bar{b}_{-\frac{1}{2}} \bar{q}$ and $\bar{a}_{-\frac{1}{2}} \bar{b}_{\frac{1}{2}} \bar{q}$, for any totally clean polynomial $\bar{q}$, is nonzero in $\widehat{HS}(\An, F_{2,2})$. This proves theorem \ref{thm:22_annulus_description}(iii).

We will also show (proposition \ref{prop:some_more_nonzero_homology}) that any $\bar{a}_i \bar{b}_j \bar{x}_{-i-j}$, for $i,j \in \Z + \frac{1}{2}$ of the same sign, is nonzero in $\widehat{HS}(\An, F_{2,2})$.

On the other hand, $\C \otimes \X \otimes \D$ certainly contributes some homology on its own account: there are cycles in $\C \otimes \X \otimes \D$ which are nonzero in $\widehat{HS}(\An, F_{2,2})$. For instance, we shall show below (propositions \ref{prop:some_nonzero_homology} and \ref{prop:some_more_nonzero_homology}) that the elements
\[ \begin{array}{c}
c_n d_{-n}, \quad c_0 d_0 + c_{1} d_{-1} \\
c_0 d_0 x_3 + \left( c_2 d_0 + c_1 d_1 + c_0 d_2 \right) x_1 + a_{\frac{7}{2}} b_{\frac{1}{2}} + a_{\frac{3}{2}} b_{\frac{5}{2}}
\end{array}
\]
for any $n \in \Z$, are all cycles, but not boundaries, in $\widehat{HS}(\An, F_{2,2})$. We will also see (proposition \ref{prop:some_nonzero_homology}) that $c_n d_{-n}$ is not homologous to any element of $\A \otimes \X \otimes \B$.

We will demonstrate these nonzero elements using two tools: firstly, in section \ref{sec:diagonal_sums}, the \emph{diagonal sum sequence}; and secondly, in section \ref{sec:homomorphism_from_disc}, by finding a ($\Z_2$-module) homomorphism $\Phi$ from $\widehat{HS}(A, F_{2,2})$ to a $\Z_2$-module $H(\E)$. But neither of these tools is complete; for instance the diagonal sum sequence cannot detect that $\bar{c}_0 \bar{d}_0 + \bar{c}_1 \bar{d}_{-1}$ is nonzero in $\widehat{HS}(\An, F_{2,2})$; and the homomorphism $\Phi$ cannot detect that $\bar{c}_0 \bar{d}_0 \bar{x}_3 + \left( \bar{c}_2 \bar{d}_0 + \bar{c}_1 \bar{d}_1 + \bar{c}_0 \bar{d}_2 \right) \bar{x}_1 + \bar{a}_{\frac{7}{2}} \bar{b}_{\frac{1}{2}} + \bar{a}_{\frac{3}{2}} \bar{b}_{\frac{5}{2}}$ is nonzero.

\subsection{Properties of full string homology; diagonal sums}
\label{sec:diagonal_sums}

We will attempt to gain some insight into $\C \otimes \X \otimes \D$ by considering the ``diagonals" $c_i d_j$, over $i,j$ such that $i+j$ is constant.

\begin{lem}
\label{lem:c_i_d_j_boundary_facts}
For all integers $i,j$:
\begin{enumerate}
\item
$c_i d_j$ and $c_{i+2} d_{j-2}$ differ by a boundary;
\item
if $i+j$ is odd, then $c_i d_j$ and $c_{i+1} d_{j-1}$ differ by a boundary.
\end{enumerate}
\end{lem}

Recall (section \ref{sec:ABCDX_chain_complex}) that $c_i d_j p$, for $p \in \X$, can be considered as $p \in \X$ placed at $(i,j) \in \Z \times \Z$. Part (i) of this lemma says that along a diagonal $i+j =$ constant in this lattice, a $1$ at each second point is equal, up to a boundary. Part (ii) says that along an \emph{odd} diagonal $i+j = $ odd constant, a $1$ at any point is equal to any other, up to boundaries. Note however, that $c_i d_j$ are not generally cycles, so do not generally have homology classes.

\begin{proof}
The first part is immediate from 
\[
\partial ( c_i d_{j-1} x_1 + c_{i+1} d_{j-2} x_1 ) = c_i d_j + c_{i+2} d_{j-2}
\]
(the two $s_{i+j-1} x_1$ terms in the differential cancel). If $m$ is an odd integer then we have
\[
\partial ( c_0 d_0 x_m ) = c_0 d_m + c_m d_0,
\]
so two terms $c_i d_j$ on the diagonal $i+j = m$, spaced an odd distance apart, are equal, up to a boundary. Combining this with (i) gives that all $c_i d_j$ on the diagonal are equal up to boundaries.
\end{proof}

\begin{defn}
Given $n \in \Z$ and $f \in \widehat{CS}(\An, F_{2,2})$ with $\C \otimes \X \otimes \D$ component $\sum_{(i,j) \in \Z \times \Z} c_i d_j p_{i,j}$, where each $p_{i,j} \in \X$, the \emph{$n$'th diagonal sum} map $\sigma_n: \widehat{CS}(\An, F_{2,2}) \To \X$ is defined by
\[
\sigma_n f = \sum_{i+j = n} p_{i,j} \in \X.
\]
The \emph{diagonal sum sequence} of $f$ is the sequence $(\sigma_n f)_{n \in \Z}$.
\end{defn}
Thus each $\sigma_n$ takes the sum of the coefficients of $f$ in the diagonal $i+j=n$ of the ``$(c_i, d_j)$-lattice". Clearly $\sigma_n$ is a $\Z_2$-module homomorphism; we now show it is a chain map.

\begin{lem}
For each $n$, $\sigma_n \partial = \partial \sigma_n$. 
\end{lem}

\begin{proof}
Each term $c_i d_j p_{i,j}$ of $f$ contributes $p_{i,j}$ to the diagonal sum $\sigma_{i+j}$. Now $\partial f$ has $\C \otimes \X \otimes \D$ component given by $c_i d_j \partial p_{i,j}$, plus some terms (possibly none) of the form $(c_{i+k} d_j + c_i d_{j+k}) x_k^{-1} p_{i,j}$. The term $c_i d_j \partial p_{i,j}$ contributes $\partial p_{i,j}$ to the $(i+j)$'th diagonal sum of $\partial f$. Any terms of the form $(c_{i+k} d_j + c_i d_{j+k}) x_k^{-1} p_{i,j}$ give the same coefficient $x_k^{-1} p_{i,j}$ occurring in two locations $(i+k,j)$, $(i,j+k)$ on the same diagonal. Hence they cancel and contribute zero to the diagonal sums of $\partial f$. Terms of the form $a_i b_j p$ contribute neither to $\sigma_n f$ nor $\sigma_n \partial f$. So the diagonal sum sequence of $\partial f$ is $(\partial \sigma_n)_{n \in \Z}$ as desired.
\end{proof}

It follows that, for any $f \in \widehat{CS}(\An, F_{2,2})$, $\partial f$ has diagonal sum sequence $(\partial \sigma_n f)_{n \in \Z}$, giving the following proposition immediately.
\begin{prop}
If $f$ is a cycle in $\widehat{CS}(\An, F_{2,2})$, then every $\sigma_n f$ is a cycle in $\X$. If $f$ is a boundary in $\widehat{CS}(\An, F_{2,2})$, then every $\sigma_n f$ is a boundary in $\X$.
\qed
\end{prop}

We can now show that two of our claimed elements are nonzero in $\widehat{HS}(\Sigma,F)$.
\begin{prop}
\label{lem:some_nonzero_homology}
\label{prop:some_nonzero_homology}
The elements
\[
c_m d_{-m}
\quad \text{and} \quad
 c_0 d_0 x_3 + \left( c_2 d_0 + c_1 d_1 + c_0 d_2 \right) x_1 + a_{\frac{7}{2}} b_{\frac{1}{2}} + a_{\frac{3}{2}} b_{\frac{5}{2}}
\]
for any $m \in \Z$, are cycles in $\widehat{CS}(\An, F_{2,2})$, but not boundaries, and are not homologous to any element of $\A \otimes \X \otimes \B$.
\end{prop}

\begin{proof}
Direct computation shows that these elements are cycles. For any $m$, $c_m d_{-m}$ has diagonal sum $\sigma_0 = 1$, which is not a boundary in $\X$. The second element has diagonal sums $\sigma_0 = x_3, \sigma_1 = x_1$, which are also not boundaries. Two cycles which differ by a boundary map under $\sigma_n$ to cycles in $\X$ which differ by boundaries, i.e. to homologous elements of $\X$. An element $\A \otimes \X \otimes \B$ maps under any $\sigma_n$ to zero, hence is not homologous to any of the elements stated here.
\end{proof}
The proves theorem \ref{thm:22_annulus_description}(iv).

Note that $c_0 d_0 + c_1 d_{-1}$ is a cycle with zero diagonal sequence, but (as we will see in proposition \ref{prop:some_more_nonzero_homology}) is nonzero in $\widehat{HS}(\An,F_{2,2})$; the diagonal sequence cannot distinguish this element from $0$.

\subsection{Homomorphism from a disc}
\label{sec:homomorphism_from_disc}

We will define a chain complex $\E$, and chain maps $\widehat{CS}(\An, F_{2,2}) \leftrightarrow \E$. The chain complex $\E$ is motivated by considering string diagrams on the disc with $6$ alternating endpoints, glued into $(\An, F_{2,2})$. Drawing the disc as a rectangle, up to homotopy relative to endpoints there are precisely $6$ diagrams without contractible loops on the disc, which we label as $A_+, A_-, B, T_0, T_1, U$ as shown in figure \ref{fig:E_diagrams}. Gluing left and right sides together, we respectively obtain string diagrams $a_{\frac{1}{2}} b_{-\frac{1}{2}}$, $a_{-\frac{1}{2}} b_{\frac{1}{2}}$, $a_{\frac{1}{2}} b_{\frac{1}{2}} x_{-1}$, $c_0 d_0$, $c_{-1} d_1$ and $c_0 d_1 x_{-1}$ on $(\An, F_{2,2})$. 

\begin{figure}[ht]
\begin{center}
\begin{tikzpicture}[
scale=1.2, 
string/.style={thick, draw=red, postaction={nomorepostaction, decorate, decoration={markings, mark=at position 0.5 with {\arrow{>}}}}}]

\draw [xshift = -4 cm] (0,0) -- (1.5,0) -- (1.5,1) -- (0,1) -- cycle;
\draw [xshift = -4 cm, string] (1,0) arc (0:180:0.25);
\draw [xshift = -4 cm, string] (0,0.5) arc (-90:0:0.5);
\draw [xshift = -4 cm, string] (1,1) arc (180:270:0.5);

\draw [xshift = -2 cm] (0,0) -- (1.5,0) -- (1.5,1) -- (0,1) -- cycle;
\draw [xshift = -2 cm, string] (1,1) arc (0:-180:0.25);
\draw [xshift = -2 cm, string] (0,0.5) arc (90:0:0.5);
\draw [xshift = -2 cm, string] (1,0) arc (180:90:0.5);

\draw [xshift = 0 cm] (0,0) -- (1.5,0) -- (1.5,1) -- (0,1) -- cycle;
\draw [xshift = 0 cm, string] (1,1) arc (0:-180:0.25);
\draw [xshift = 0 cm, string] (0,0.5) -- (1.5,0.5);
\draw [xshift = 0 cm, string] (1,0) arc (0:180:0.25);

\draw [xshift = 2 cm] (0,0) -- (1.5,0) -- (1.5,1) -- (0,1) -- cycle;
\draw [xshift = 2 cm, string] (1,1) -- (0.5,0);
\draw [xshift = 2 cm, string] (0,0.5) arc (-90:0:0.5);
\draw [xshift = 2 cm, string] (1,0) arc (180:90:0.5);

\draw [xshift = 4 cm] (0,0) -- (1.5,0) -- (1.5,1) -- (0,1) -- cycle;
\draw [xshift = 4 cm, string] (1,0) -- (0.5,1);
\draw [xshift = 4 cm, string] (0,0.5) arc (90:0:0.5);
\draw [xshift = 4 cm, string] (1,1) arc (180:270:0.5);

\draw [xshift = 6 cm] (0,0) -- (1.5,0) -- (1.5,1) -- (0,1) -- cycle;
\draw [xshift = 6 cm, string] (0,0.5) to [bend right = 30] (1.5,0.5);
\draw [xshift = 6 cm, string] (1,0) to [bend right = 30] (0.5,1);
\draw [xshift = 6 cm, string] (1,1) to [bend right = 30] (0.5,0);

\draw (-3.25,-0.5) node {$A_+$};
\draw (-1.25,-0.5) node {$A_-$};
\draw (0.75,-0.5) node {$B$};
\draw (2.75,-0.5) node {$T_0$};
\draw (4.75,-0.5) node {$T_1$};
\draw (6.75,-0.5) node {$U$};

\end{tikzpicture}
\caption{Diagrams motivating $\E$.}
\label{fig:E_diagrams}
\end{center}
\end{figure}
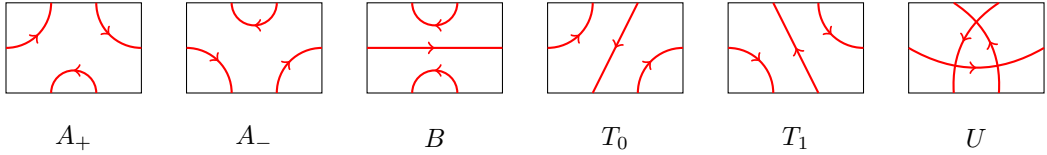

\begin{defn}
The $\Z_2$-module $\E$ is freely generated by $\{A_+, A_-, B, T_0, T_1, U\}$, and $\partial: \E \To \E$ is defined on generators and extended linearly as
\[
\begin{array}{c}
\partial A_+ = \partial A_- = \partial B = \partial T_0 = \partial T_1 = 0, \\
\partial U = B + T_0 + T_1.
\end{array}
\]
\end{defn}

\begin{defn}
The map $\Psi: \E \To \widehat{CS}(\An,F_{2,2})$ is defined on generators, extended linearly, by
\[
\begin{array}{c}
A_+ \mapsto a_{\frac{1}{2}} b_{-\frac{1}{2}}, \quad
A_- \mapsto a_{-\frac{1}{2}} b_{\frac{1}{2}}, \quad
B \mapsto a_{\frac{1}{2}} b_{\frac{1}{2}} x_{-1} \\
T_0 \mapsto c_0 d_0, \quad
T_1 \mapsto c_{-1} d_1, \quad
U = c_0 d_1 x_{-1}.
\end{array}
\]
\end{defn}

It's clear that $\partial^2 = 0$ on $\E$, so $\E$ is a chain complex. Its homology is easily computed: $H(\E) \cong \Z_2^4$, with free basis given by the homology classes of $A_+, A_-, T_0$ and $T_1$.

\begin{lem}
$\Psi$ is a chain map.
\end{lem}

\begin{proof}
We check explicitly on generators. Since $\partial A_+ = \partial A_- = \partial B= \partial T_0 = \partial T_1 = 0$, for these generators it is sufficient to check that $\Psi$ maps them to cycles, which is clear. For the remaining generator $U$ we have $\partial \Psi U = \partial (c_0 d_1 x_{-1}) = c_0 d_0 + c_{-1} d_1 + a_{\frac{1}{2}} b_{\frac{1}{2}} x_{-1} = \Psi \left( T_0 + T_1 + B \right) = \Psi \partial U$.
\end{proof}

The key to making deductions about $\widehat{HS}(\An, F_{2,2})$ is to have a homomorphism in the other direction $\Phi: \widehat{CS}(\An, F_{2,2}) \To \E$, which we now define. We define $\Phi$ on the $\Z_2$-basis of $\widehat{CS}(A, F_{2,2})$, which consists of elements of the form $a_i b_j x^e$ and $c_k d_l x^e$, over all $i,j \in \Z + \frac{1}{2}$, all $k,l \in \Z$ and all monomials $x^e \in \X$.
\begin{defn}
The map $\Phi: \widehat{CS}(\Sigma, F_{2,2}) \To \E$ is defined on generators and extended linearly by
\begin{align*}
\Phi a_{\frac{1}{2}} b_{-\frac{1}{2}} &= A_+ \\
\Phi a_{-\frac{1}{2}} b_{\frac{1}{2}} &= A_- \\
\Phi a_i b_j x_{-i-j} &= B \quad \text{for all $i,j$ of the same sign} \\
\Phi a_i b_j p &= 0 \quad \text{for any $i,j$ and monomial $p$ not covered by the previous cases} \\
\Phi c_{2n} d_{-2n} &= T_0 \quad \text{for all integers $n$} \\
\Phi c_{2n-1} d_{1-2n} &= T_1 \quad \text{for all integers $n$} \\
\Phi c_i d_j x_{-i-j} &= U \quad \text{for all pairs of integers $(i,j)$ such that $i+j$ is odd} \\
\Phi c_i d_j p &= 0 \quad \text{for any $i,j$ and monomial $p$ not covered by the previous cases}.
\end{align*}
\end{defn}

It is clear from the definition that $\Phi$ and $\Psi$ are partial inverses:
\begin{equation}
\label{eqn:composition_identity}
\Phi \circ \Psi = 1.
\end{equation}

In order to prove that $\Phi$ is a chain map, we develop some lemmas. The first lemma is about two elements $c_i d_j p$ and $c_{i+n} d_{j-n} p$ in the same diagonal.
\begin{lem}
\label{lem:phi_on_diagonal}
Suppose $i,j \in \Z$ such that $i+j \neq 0$, and $p \in \X$ is any monomial. Then for any $n \in \Z$,
\[
\Phi \left( \left( c_i d_j + c_{i+n} d_{j-n} \right) p \right) = 0.
\]
\end{lem}

\begin{proof}
If $i+j$ is even or $p \neq x_{-i-j}$ then $\Phi c_i d_j p = \Phi c_{i+n} d_{j-n} p = 0$ and the result holds. Otherwise $i+j \neq 0$ and $p = x_{-i-j}$, in which case $\Phi c_i d_j p = \Phi c_{i+n} d_{j-n} p = U$, so $\Phi(c_i d_j p + c_{i+n} d_{j-n} p ) = 0$.
\end{proof}

The second lemma concerns $\Phi$ applied to $s_n$, which (definition \ref{defn:s_n_defn}) is a linear combination of the $a_i b_j$ along the diagonal $i+j=n$. 
\begin{lem}
\label{lem:phi_s_n_p}
Let $n$ be an integer, and $p \in \X$ a monomial. Then
\[
\Phi \left( s_n p \right) = \left\{ \begin{array}{ll}
B & \text{$n$ odd and $p = x_{-n}$} \\
0 & \text{otherwise}.
\end{array} \right.
\]
\end{lem}

\begin{proof}
The element $s_n p$ is the sum of all $a_i b_j p$, over $i,j$ with the same sign as $n$ and $i+j= n$. If $p \neq x_{-n} = x_{-i-j}$, then the image under $\Phi$ of each $a_i b_j p$ is zero, giving the result.

We may now assume $p = x_{-n} = x_{-i-j}$. In this case each $a_i b_j p$ maps to $B$ under $\Phi$. There are $|n|$ such terms, so $\Phi(s_n p)$ is $0$ if $n$ is even, and $B$ if $n$ is odd.
\end{proof}

The third lemma shows that $\Phi \partial = 0$ for many elements.
\begin{lem}
\label{lem:phi_partial_often_zero}
For any $i,j$ and any monomial $p \in \X$:
\begin{enumerate}
\item $\Phi \partial \left( a_i b_j p \right) = 0$.
\item $\Phi \left( c_i d_j (\partial p) \right) = 0$
\end{enumerate}
\end{lem}

\begin{proof}
First consider $\partial (a_i b_j p) = (\partial a_i) b_j p + a_i (\partial b_j) p + a_i b_j (\partial p)$. Every term in $\partial a_i$ or $\partial b_j$ contains a factor $a_k x_l$ where $k,l$ are of the same sign; and hence very term of $(\partial a_i) b_j p$ or $a_i (\partial b_j) p$ also contains a factor $a_k x_l$ where $k,l$ are of the same sign. On all such terms $\Phi = 0$. 

Thus it remains to show $\Phi (a_i b_j (\partial p)) = \Phi (c_i d_j (\partial p)) = 0$. Now any term of $\partial p$ (if there is any such term) is a product of at least two $x_k$. Thus every term of $a_i b_j (\partial p)$ or $c_i d_j (\partial p)$ has at least two $x_k$ factors, and on any such term $\Phi = 0$. 
\end{proof}

\begin{prop}
$\Phi$ is a chain map.
\end{prop}

\begin{proof}
We check on the $\Z_2$-basis of $\widehat{CS}(\Sigma,F_{2,2})$ that $\partial \Phi = \Phi \partial$.

First, take $a_i b_j p$ where $p$ is a monomial. By lemma \ref{lem:phi_partial_often_zero}, we have $\Phi \partial \left( a_i b_j p \right) = 0$. And $\Phi(a_i b_j p)$ is either $A_+$, $A_-$ or $B$, all of which map to $0$ under $\partial$.

Next, take $c_i d_j p$, where $p$ is a monomial. We have $\Phi (c_i d_j p)$ is either $T_0, T_1$ $U$ or $0$, depending on $i,j$ and $p$. However $\partial T_0 = \partial T_1 = 0$. So $\partial \Phi (c_i d_j p)$ is nonzero precisely when $\Phi (c_i d_j p) = U$, in which case $i+j$ is odd, $p = x_{-i-j}$ and $\partial \Phi(c_i d_j p) = T_0 + T_1 + B$. Thus, we must show $\Phi \partial (c_i d_j p)$ is nonzero precisely when $i+j$ is odd and $p=x_{-i-j}$, in which case $\Phi \partial (c_i d_j p) = T_0 + T_1 + B$.

So, consider $\partial (c_i d_j p)$. Let $p = x^e = \prod_{k \in \Z \backslash \{0\}} x_k^{e_k}$. Then
\[
\partial \left( c_i d_j p \right) = c_i d_j (\partial p) + s_{i+j} p + \sum_{k \in \Z \backslash \{0\}} k e_k \left( c_{i+k} d_j + c_i d_{j+k} \right) x_k^{-1} p.
\]
Applying $\Phi$ to this expression, the first term maps to $0$ by lemma \ref{lem:phi_partial_often_zero}, and the second term, by lemma \ref{lem:phi_s_n_p}, maps to $B$ when $i+j$ is odd and $p=x_{-i-j}$, otherwise maps to $0$. So it remains to show that
\[
\Phi \left( \sum_{k \in \Z \backslash \{0\}} k e_k \left( c_{i+k} d_j + c_i d_{j+k} \right) x_k^{-1} p \right) = \left\{ \begin{array}{ll}
T_0 + T_1 & \text{$i+j$ odd and $p = x_{-i=j}$} \\
0 & \text{otherwise.} \end{array} \right.
\]
Now by lemma \ref{lem:phi_on_diagonal}, when $k \neq -i-j$, we have $\Phi( (c_{i+k} d_j + c_i d_{j+k}) q)=0$ for any monomial $q$. Thus we only need consider the terms with $k = -i-j$. In this case the expression above reduces (mod 2) to
\[
(i+j) e_{-i-j}  \; \Phi \left(  \left( c_{-j} d_j + c_i d_{-i} \right) x_{-i-j}^{-1} p \right).
\]
If $i+j$ is even then we obtain $0$ mod $2$, so we may assume $i+j$ is odd. By definition, $\Phi(c_l d_{-l} q) = 0$ for any integer $l$ and any monomial $q \neq 1$; so we may assume $p=x_{-i-j}$. Then $e_{-i-j} = 1$ so the above expression becomes $\Phi (c_{-j} d_j + c_i d_{-i})$. With $i+j$ odd, one of $i,j$ is even and other is odd, so $\Phi(c_{-j} d_j + c_i d_{-i}) = T_0 + T_1$ as desired.
\end{proof}

As a chain map, $\Phi$ descends to homology and we obtain a map $\widehat{HS}(\An, F_{2,2}) \To H(\E)$, also denoted $\Phi$. Thus $\Phi$ and $\Psi$ are partial inverses on homology, $\Phi \circ \Psi = 1$. In particular, $\Psi$ is injective and $\Phi$ is surjective on homology.

\begin{prop}
\label{prop:some_more_nonzero_homology}
The elements 
\[
\bar{a}_i \bar{b}_j \bar{x}_{-i-j}, \quad
\bar{c}_n \bar{d}_{-n},
\]
for any $i,j \in \Z + \frac{1}{2}$ of the same sign, and any $n \in \Z$, are nonzero in $\widehat{HS}(\An, F_{2,2})$. The elements
\[
\bar{c}_0 \bar{d}_0,  \quad
\bar{c}_1 \bar{d}_{-1}, \quad
\bar{c}_0 \bar{d}_0 + \bar{c}_1 \bar{d}_{-1} =
\bar{a}_{\frac{1}{2}} \bar{b}_{\frac{1}{2}} \bar{x}_{-1}, \quad
\bar{a}_{\frac{1}{2}} \bar{b}_{-\frac{1}{2}}, \quad
\bar{a}_{-\frac{1}{2}} \bar{b}_{\frac{1}{2}}, \quad
\]
are all distinct nonzero elements of $\widehat{HS}(A, F_{2,2})$.
\end{prop}

\begin{proof}
We first verify that the corresponding elements of $\widehat{CS}(A, F_{2,2})$ are cycles, hence represent homology classes; and $\partial (c_1 d_0 x_{-1}) = c_0 d_0 + c_1 d_{-1} + a_{\frac{1}{2}} b_{\frac{1}{2}} x_{-1}$ explains the claimed equality.

Under $\Phi$, $a_i b_j x_{-i-j} \mapsto B$ for any $i,j$ of the same sign, and $c_n d_{-n} \mapsto T_0$ or $T_1$ depending on the parity of $n$. As the homology classes of $B,T_0,T_1$ are all nonzero in $H(\E)$, the homology classes of $a_i b_j x_{-i-j}$ and $c_n d_{-n}$ must be nonzero in $\widehat{HS}(\An, F_{2,2})$. 

Under the injective map $\Psi$, the elements $T_0, T_1, T_0 + T_1, A_+, A_-$, which represent distinct nonzero homology classes of $\E$, map to the second list of elements of $H(\widehat{HS}(\An, F_{2,2})$, which must therefore be nonzero and distinct.
\end{proof}

Note that $\Phi$ certainly has nonzero kernel; for instance, for any totally clean monomial other than $1$, $a_{\frac{1}{2}} b_{-\frac{1}{2}} q$ and $a_{-\frac{1}{2}} b_{\frac{1}{2}} q$ map to zero under $\Phi$, but are nonzero in $\widehat{HS}(\An,F_{2,2})$ as discussed in section \ref{sec:towards_full_22_homology}. Also, the element $c_0 d_0 x_3 + \left( c_2 d_0 + c_1 d_1 + c_0 d_2 \right) x_1 + a_{\frac{7}{2}} b_{\frac{1}{2}} + a_{\frac{3}{2}} b_{\frac{5}{2}}$, shown to be nonzero in $\widehat{HS}(\An, F_{2,2})$ in proposition \ref{prop:some_nonzero_homology}, maps to zero under $\Phi$.

\section{Adding marked points}
\label{sec:adding_marked_points}

We now show how to extend our results to weakly marked annuli with more fixed points. Having proved theorem \ref{thm:non-sutured_zero_iso} in section \ref{sec:non-alternating_annuli}, we restrict our attention to \emph{alternating} weakly marked annuli.

We will show how, in favourable circumstances, we can add marked points to a boundary component, and keep track of the effect on $\widehat{HS}$. In fact our main result in this regard (theorem \ref{prop:HS_doubling}) applies not just to annuli, but to general alternating weakly marked surfaces.

We will show that once there are two marked points on a boundary component $C$ of an alternating weakly marked surface $(\Sigma, F)$, we can add two more marked points to $C$ (keeping the points alternating), and the effect on string homology is to tensor (over $\Z_2$) with $\Z_2^2$. In fact the results to show this essentially appeared in our previous paper \cite{Mathews_Schoenfeld12_string}, although in that paper the result was stated only for discs.

To this end, we recall the constructions and results of sections 7--8 of \cite{Mathews_Schoenfeld12_string}, but in a more general context. In section \ref{sec:creation_annihilation} we will recall the notions of creation and annihilation operators defined there. Then in section \ref{sec:effect_of_creation} we will use them to prove the main proposition in this section, and in section \ref{sec:results_for_annuli_with_more} deduce results for annuli.

\subsection{Creation and annihilation operators}
\label{sec:creation_annihilation}

Let $(\Sigma, F)$ be an alternating weakly marked surface. Let $C$ be a boundary component of $\Sigma$ which contains at least one (hence at least two) marked points. 

Suppose we add two new adjacent marked points $f_{in}, f_{out}$ on $C$, labelled ``in" and ``out" respectively, to obtain a new alternating weakly marked surface $(\Sigma, F')$. A \emph{creation operator} $\widehat{CS}(\Sigma, F) \To \widehat{CS}(\Sigma,F')$ takes a (homotopy class of) string diagram and inserts an extra boundary-parallel open string from $f_{in}$ to $f_{out}$, not intersecting itself or any other strings. Since there are pre-existing points of $F$ on $C$, the homotopy class of this newly ``created" boundary-parallel string is unique.

(Note that if $F$ has no marked points on $C$, then there are two choices for the new non-intersecting boundary-parallel open string from $f_{in}$ to $f_{out}$, proceeding around $C$ in either direction. The requirement of marked points on $C$ is essential for a well-defined creation operator.)

Similarly, suppose $(\Sigma,F)$ contains at least \emph{four} points on the boundary component $C$, and let $(\Sigma,F')$ be an alternating weakly marked surface obtained from $F$ by \emph{removing} two adjacent marked points $f_{in} \in F_{in} \cap C$, $f_{out} \in F_{out} \cap C$. An \emph{annihilation operator} $\widehat{CS}(\Sigma,F) \To \widehat{CS}(\Sigma,F')$ takes a (homotopy class of) string diagram and joins the strings previously ending at $f_{in}, f_{out}$, without introducing any new intersections of strings. The other points of $F$ on $C$ constrain the homotopy class of the new string diagram to be unique. The ``annihilation" closes off the strings at $f_{in}, f_{out}$ by a small arc.

(Note again that if $|F \cap C| = 2$, then there are two choices for the new arc, proceeding either direction around $C$. Again, marked points on $C$ are essential for a well-defined annihilation operator.)

Suppose one of the ``in" marked points on $F \cap C$ is chosen as a basepoint $f_0 \in F_{in} \cap C$. Then we consider two specific creation operators on $(\Sigma, F)$, creating new strings in two sites adjacent to the basepoint, which we call $a_-^*, a_+^*$; and we consider two specific annihilation operators, annihilating strings at the two possible sites including the basepoint. These operators are defined explicitly in figure \ref{fig:a_and_a_star}. After creating or annihilating we choose a new basepoint in the resulting diagram, adjacent to the original location. Basepoints are shown with a dot. (We use the orientation on $\Sigma$ to orient $C$ and obtain well-defined maps.)

\begin{figure}[ht]
\begin{center}
\begin{tikzpicture}[
scale=0.8, 
string/.style={thick, draw=red, postaction={nomorepostaction, decorate, decoration={markings, mark=at position 0.5 with {\arrow{>}}}}},
boundary/.style={ultra thick}]

\draw [boundary] (0,0) circle (1.5 cm);
\draw [string] (30:1.5) -- (30:1);
\draw [string] (60:1) -- (60:1.5);
\draw [string] (90:1.5) -- (90:1);
\draw [string] (120:1) -- (120:1.5);
\draw [string] (150:1.5) -- (150:1);
\fill [draw=red, fill=red] (90:1.7) circle (2pt);

\draw [shorten >=1mm, -to, decorate, decoration={snake,amplitude=.4mm, segment length = 2mm, pre=moveto, pre length = 1mm, post length = 2mm}]
(2,1) -- (3,1.5);
\draw (2,1.5) node {$a_+^*$};

\draw [xshift = 5 cm, yshift = 2 cm, boundary] (0,0) circle (1.5 cm);
\draw [xshift = 5 cm, yshift = 2 cm, string] (30:1.5) -- (30:1);
\draw [xshift = 5 cm, yshift = 2 cm, string] (50:1) -- (50:1.5);
\draw [xshift = 5 cm, yshift = 2 cm, string] (70:1.5) -- (70:1);
\draw [xshift = 5 cm, yshift = 2 cm, string] (110:1.5) .. controls (110:1) and (90:1) .. (90:1.5);
\draw [xshift = 5 cm, yshift = 2 cm, string] (130:1) -- (130:1.5);
\draw [xshift = 5 cm, yshift = 2 cm, string] (150:1.5) -- (150:1);
\fill [xshift = 5 cm, yshift = 2 cm, draw=red, fill=red] (110:1.7) circle (2pt);

\draw [shorten >=1mm, -to, decorate, decoration={snake,amplitude=.4mm, segment length = 2mm, pre=moveto, pre length = 1mm, post length = 2mm}]
(2,-1) -- (3,-1.5);
\draw (2,-1.5) node {$a_+$};

\draw [xshift = 5 cm, yshift = -2 cm, boundary] (0,0) circle (1.5 cm);
\draw [xshift = 5 cm, yshift = -2 cm, string] (30:1.5) -- (30:1);
\draw [xshift = 5 cm, yshift = -2 cm, string] (60:1) -- (60:1.5);
\draw [xshift = 5 cm, yshift = -2 cm, string] (90:1) .. controls (90:1.5) and (120:1.5) .. (120:1);
\draw [xshift = 5 cm, yshift = -2 cm, string] (150:1.5) -- (150:1);
\fill [xshift = 5 cm, yshift = -2 cm, draw=red, fill=red] (30:1.7) circle (2pt);

\draw [shorten >=1mm, -to, decorate, decoration={snake,amplitude=.4mm, segment length = 2mm, pre=moveto, pre length = 1mm, post length = 2mm}]
(-2,1) -- (-3,1.5);
\draw (-2,1.5) node {$a_-^*$};

\draw [xshift = -5 cm, yshift = 2 cm, boundary] (0,0) circle (1.5 cm);
\draw [xshift = -5 cm, yshift = 2 cm, string] (30:1.5) -- (30:1);
\draw [xshift = -5 cm, yshift = 2 cm, string] (50:1) -- (50:1.5);
\draw [xshift = -5 cm, yshift = 2 cm, string] (70:1.5) .. controls (70:1) and (90:1) .. (90:1.5);
\draw [xshift = -5 cm, yshift = 2 cm, string] (110:1.5) -- (110:1);
\draw [xshift = -5 cm, yshift = 2 cm, string] (130:1) -- (130:1.5);
\draw [xshift = -5 cm, yshift = 2 cm, string] (150:1.5) -- (150:1);
\fill [xshift = -5 cm, yshift = 2 cm, draw=red, fill=red] (70:1.7) circle (2pt);

\draw [shorten >=1mm, -to, decorate, decoration={snake,amplitude=.4mm, segment length = 2mm, pre=moveto, pre length = 1mm, post length = 2mm}]
(-2,-1) -- (-3,-1.5);
\draw (-2,-1.5) node {$a_-$};

\draw [xshift = -5 cm, yshift = -2 cm, boundary] (0,0) circle (1.5 cm);
\draw [xshift = -5 cm, yshift = -2 cm, string] (30:1.5) -- (30:1);
\draw [xshift = -5 cm, yshift = -2 cm, string] (60:1) .. controls (60:1.5) and (90:1.5) .. (90:1);
\draw [xshift = -5 cm, yshift = -2 cm, string] (120:1) -- (120:1.5);
\draw [xshift = -5 cm, yshift = -2 cm, string] (150:1.5) -- (150:1);
\fill [xshift = -5 cm, yshift = -2 cm, draw=red, fill=red] (150:1.7) circle (2pt);

\end{tikzpicture}
\caption{Creation/annihilation operators. We only show the boundary component $C$ of $(\Sigma,F)$; $\Sigma$ in general has more topology than shown.} 
\label{fig:a_and_a_star}
\end{center}
\end{figure}
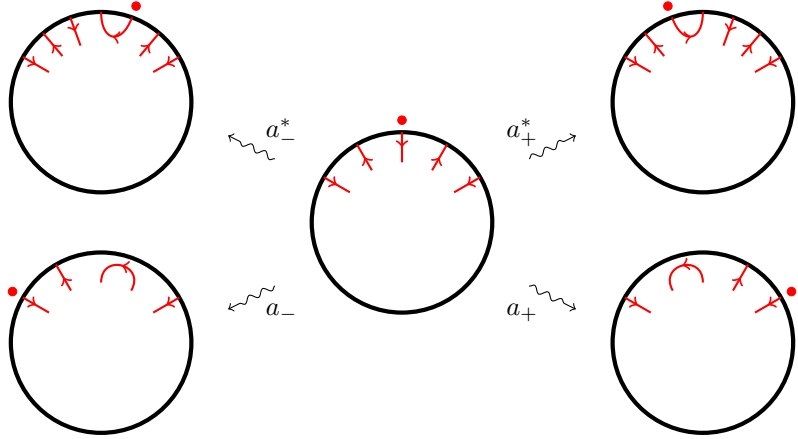

Thus, given an alternating weakly marked surface $(\Sigma, F)$, once we chose a boundary component $C$ with at least two marked points, and a basepoint $f_0 \in F_{in} \cap C$, we have two well-defined creation operators
\[
a_\pm^* \; : \; \widehat{CS}(\Sigma,F) \To \widehat{CS}(\Sigma,F'),
\]
where $F'$ is obtained from $F$ by adding two adjacent marked points on $C$, and $(\Sigma, F')$ has a well-defined basepoint in $F'_{in} \cap C$.

Similarly, given an alternating weakly marked $(\Sigma, F)$, once we choose a boundary component $C$ with at least \emph{four} marked points, and a basepoint $f_0 \in F_{in} \cap C$, we have two well-defined annihilation operators
\[
a_\pm \; : \; \widehat{CS}(\Sigma, F) \To \widehat{CS}(\Sigma,F'),
\]
where $F'$ is obtained from $F$ by deleting two marked points on $C$, and $(\Sigma, F')$ has a well-defined basepoint in $F'_{in} \cap C$.

Thus, starting from an alternating weakly marked $(\Sigma, F)$ and a choice of basepoint as above, we can compose $a_\pm^*, a_\pm$ operators, applied to the resulting chain complexes, as long as our basepoint continues to share its boundary component with a sufficient number of marked points.

The four maps $\widehat{a}_\pm^*, \widehat{a}_{\pm}$ satisfy various relations, including the following.
\[
a_- a_-^* = a_+ a_+^* = 1,
\quad
a_- a_+^* = a_+ a_-* = 0
\]
Identity maps arise as a ``$\pm$-creation" followed by ``$\pm$-annihilation" result in a string diagram homotopic to the original; and zeroes arise as a ``$\pm$-creation" followed by ``$\mp$-annihilation" produce a closed contractible string. 

As neither creation nor annihilation operators create any new crossings in a string diagram, they commute with the differential and hence give maps on homology, also denoted $a_\pm^*, a_\pm$. The above identities also hold on homology. In particular, creation operators are injective on both the chain level and on homology.

\subsection{Effect of creation on string homology}
\label{sec:effect_of_creation}

We now use the creation and annihilation operators described above to give an explicit result about string homology. When we add two marked points to $F$ to obtain $F'$, with creation operators as described above, it turns out we can describe $\widehat{HS}(\Sigma,F')$ rather simply and explicitly in terms of $\widehat{HS}(\Sigma,F)$.
\begin{thm}
\label{prop:HS_doubling}
\label{thm:HS_doubling}
Let $(\Sigma,F)$ be an alternating weakly marked surface, and $C$ a boundary component of $\Sigma$ with $F \cap C \neq \emptyset$. Let $f_0 \in F_{in} \cap C$ be a basepoint, and $(\Sigma,F')$ an alternating weakly marked surface obtained from $(\Sigma,F)$ by adding two marked points on $C$. Let $a_\pm^*$ be the corresponding creation operators. Then (as a $\Z_2$-module)
\begin{align*}
\widehat{HS} (\Sigma, F') &= a_+^* \widehat{HS} (\Sigma, F) \oplus a_-^* \widehat{HS}(\Sigma,F) \\
&\cong \widehat{HS}(\Sigma, F) \oplus \widehat{HS}(\Sigma,F) \\
&\cong \left( \Z_2 \oplus \Z_2 \right) \otimes_{\Z_2} \widehat{HS}(\Sigma,F).
\end{align*}
\end{thm}

To prove this theorem, we need the following ``crossed wires lemma". On discs, it is stated as lemma 8.4 of \cite{Mathews_Schoenfeld12_string}; the same methods establish the result in our more general setting.
\begin{lem}
Let $\Sigma, F, F'$ and $a_\pm^*$ be as above. Suppose $x \in \widehat{CS}(\Sigma,F')$ satisfies $\partial x = 0$. Then there exist $y, z \in \widehat{CS}(\Sigma,F)$ and $u \in \widehat{CS}(\Sigma,F')$ such that
\[
\partial y = \partial z = 0
\quad \text{and} \quad
x = a_-^* y + a_+^* z + \partial u.
\]
\qed
\end{lem}

\begin{proof}[Proof of theorem \ref{prop:HS_doubling}]
The ideas of the proof are contained in \cite{Mathews_Schoenfeld12_string}.  Take $\bar{x} \in \widehat{HS}(\Sigma,F')$, represented by $x \in \widehat{CS}(\Sigma, F')$. Then $\partial x = 0$, so by the crossed wires lemma we have $y,z \in \widehat{CS}(\Sigma,F)$ and $u \in \widehat{CS}(\Sigma,F')$ such that $x = a_-^* y + a_+^* z + \partial u$. In homology then we have $\bar{x} = a_-^* \bar{y} + a_+^* \bar{z}$. Thus $a_-^* \widehat{HS}(\Sigma,F)$ and $a_+^* \widehat{HS}(\Sigma,F)$ span $\widehat{HS}(\Sigma,F')$.

Now suppose we have an element $t$ in the intersection $a_+^* \widehat{HS}(\Sigma,F) \cap a_-^* \widehat{HS}(\Sigma,F)$. So there exist $p,q \in \widehat{HS}(\Sigma,F)$ such that
\[
t = a_-^* p = a_+^* q.
\]
Now applying $a_-$ and $a_+$ respectively, we obtain
\[
a_- t = p = 0
\quad \text{and} \quad
a_+ t = 0 = q.
\]
Here we have used the relations $a_- a_-^* = a_+ a_+^* = 1$ and $a_- a_+^* = a_+ a_-^* = 0$. Since $p = q = 0$ we have $t = 0$, so $a_+^* \widehat{HS}(\Sigma,F) \cap a_-^* \widehat{HS}(\Sigma,F) = 0$ and we have the first direct sum claimed. As creation operators are injective, we have the first claimed isomorphism, and the final isomorphism then follows.
\end{proof}

We have now proved theorem \ref{thm:tensor_with_Z_2_squared}.

It follows from this proposition that if $\{v_i \; : \; i \in I \}$ is a basis for $\widehat{HS}(\Sigma,F)$, then all $a_-^* v_i, a_+^* v_i$ are distinct and $\{a_-^* v_i, a_+^* v_i \; : \; i \in I \}$ forms a basis for $\widehat{HS}(\Sigma,F')$.

\subsection{Results for annuli with more marked points}
\label{sec:results_for_annuli_with_more}

For present purposes, we only need theorem \ref{prop:HS_doubling} in so far as it relates to annuli. We can use it to immediately deduce the string homology of $(\An, F_{0,2n+2})$ from that of $(\An, F_{0,2})$; and to deduce the string homology of $(\An, F_{2m+2,2n+2})$ from that of $(\An, F_{2,2})$.

We will need to apply creation operators $a_\pm^*$ repeatedly. Following notation of \cite{Me09Paper}, for any word $w$ on the symbols $\{-,+\}$, we define $a_w^*$ to be the corresponding composition of $a_-^*$ and $a_+^*$. Thus for instance $a_{-++}^* = a_-^* a_+^* a_+^*$. We denote the set of such words of length $n$ by $\{-,+\}^n$.

First consider $(\An, F_{0,2n+2})$, which by proposition \ref{prop:HX-modules} is an $H(\X)$-module. When $n=0$ we have (theorem \ref{thm:annulus_two_marked_points_same_boundary_calculation}), as an $H(\X)$-module,
\[
\widehat{HS}(A, F_{0,2}) = H(\A \otimes \X) = \bar{x}_1 H(\X) \oplus \bar{x}_{-1} H(\X)
\]
where
\[
H(\X) = \frac{ \Z[ \ldots, \bar{x}_{-3}, \bar{x}_{-1}, \bar{x}_1, \bar{x}_3, \ldots] }{ \ldots, \bar{x}_{-3}^2, \bar{x}_{-1}^2, \bar{x}_1^2, \bar{x}_3^2, \ldots }.
\]
We note that, on an annulus $(\An, F_{0,2n+2})$, a creation operator is compatible with the differential $\X$-module structure: each creation operator inserts an arc and commutes with inserting closed curves, without introducing any new intersections. Thus the operators
\[
a_\pm^* \; : \; \widehat{CS}(\An, F_{0,2n+2}) \To \widehat{CS}(\An, F_{0,2n+4})
\]
are in fact differential $\X$-module homomorphisms.

Repeatedly applying proposition \ref{prop:HS_doubling} gives the following.
\begin{prop}
\label{prop:F_0_2n+2_homology}
Let $n \geq 0$. Then as $H(\X)$-module,
\begin{align*}
\widehat{HS}(A, F_{0,2n+2}) 
&= \bigoplus_{w \in \{-,+\}^n} a_w^* \widehat{HS}(A, F_{0,2}) \\
&= \bigoplus_{w \in \{-,+\}^n} a_w^* \left( \bar{x}_1 H(\X) \oplus \bar{x}_{-1} H(\X) \right) \\
&\cong \left( \bar{x}_1 H(\X) \oplus \bar{x}_{-1} H(\X) \right)^{\oplus 2^n} \\
&\cong \left( \Z_2 \oplus \Z_2 \right)^{\otimes n} \otimes_{\Z_2} \widehat{HS}(A, F_{0,2}).
\end{align*}
\end{prop}

\begin{proof}
As a $\Z_2$-module, the first equality follows immediately from repeatedly applying proposition \ref{prop:HS_doubling}. The second equality then follows from theorem \ref{thm:annulus_two_marked_points_same_boundary_calculation}. As each $a_\pm^*$ is injective, so too is each $a_w^*$, so each $a_w^*(\bar{x}_1 H(\X) \oplus \bar{x}_{-1} H(\X)) \cong \bar{x}_1 H(\X) \oplus \bar{x}_{-1} H(\X)$, giving direct sum in the third isomorphism. As $(\Z_2 \oplus \Z_2)^{\otimes n}$ is a free $\Z_2$-module of rank $2^n$, we then have the final isomorphism.

The operators $a_\pm^*$ are compatible with the $\X$-module structure on each $\widehat{CS}(A, F_{0,2n})$, so we have isomorphisms of of $H(\X)$-modules.
\end{proof}

Next consider alternating marked annuli of the form $(\An, F_{2m+2,2n+2})$. Although the computation of $\widehat{HS} (\An, F_{2,2})$ in the foregoing presents various difficulties and we have not completed it, we know explicitly how to go from $(\An, F_{2,2})$ to a higher number of marked points. Recall that $\widehat{CS}(\An, F_{2,2})$ is not an $\X$-module, so $\widehat{HS}(\An, F_{2,2})$ is only a $\Z_2$-module, not an $H(\X)$-module.

To increase the number of marked points on both boundary components of $\An$, we consider creation operators on each boundary. Choosing a basepoint on each boundary component we obtain creation operators
\begin{align*}
a_{\pm}^{0*} \; &: \; \widehat{CS}(\An, F_{2m+2,2n+2}) \To \widehat{CS}(\An, F_{2m+4, 2n+4}),  \\
a_{\pm}^{1*} \; &: \; \widehat{CS}(\An, F_{2m+2,2n+2}) \To \widehat{CS}(\An, F_{2m+2,2n+4}).
\end{align*}
Applying proposition \ref{prop:HS_doubling} repeatedly to both boundary components gives the following.
\begin{prop}
\label{prop:F_2m+2_2n+2_homology}
Let $m,n \geq 0$. Then as a $\Z_2$-module,
\begin{align*}
\widehat{HS}(\An, F_{2m+2,2n+2}) 
&= \bigoplus_{w_0 \in \{-,+\}^m} \bigoplus_{w_1 \in \{-,+\}^n} a_{w_0}^{0*} a_{w_1}^{1*} \widehat{HS} (\An, F_{2,2}) \\
&\cong \left( \widehat{HS}(\An, F_{2,2}) \right)^{\oplus 2^{m+n}} \\
&\cong \left( \Z_2 \oplus \Z_2 \right)^{\otimes (m+n)} \otimes_{\Z_2} \widehat{HS}(\An, F_{2,2})
\end{align*}
\end{prop}

\begin{proof}
The first equality follows immediately from applying proposition \ref{prop:HS_doubling} $m$ times to one boundary component and $n$ times to the other. The next isomorphism follows since creation operators are injective and the set of pairs of creation operators $(w_0, w_1)$ with $w_0 \in \{-,+\}^m$ and $w_1 \in \{-,+\}^n$ has cardinality $2^{m+n}$. The final isomorphism follows since $(\Z_2 \oplus \Z_2)^{\otimes (m+n)}$ is free over $\Z_2$ of rank $2^{m+n}$.
\end{proof}

Having proved propositions \ref{prop:F_0_2n+2_homology} and \ref{prop:F_2m+2_2n+2_homology}, we have proved theorem \ref{thm:general_annulus_computations}.

%For the end..
\addcontentsline{toc}{section}{References}

\small

\bibliography{danbib}
\bibliographystyle{amsplain}

\end{document}